\documentclass{amsart}

\usepackage[olditem,oldenum]{paralist}  
\usepackage{graphicx}                   
\usepackage{amssymb}                    
\usepackage{natbib}                     
\usepackage{footmisc}                   
\usepackage{bm}
\usepackage{amscd}
\usepackage{units}                      
\usepackage{verbatim}                   
\usepackage[pdftex]{hyperref}

\theoremstyle{plain}
  \newtheorem{theorem}{Theorem}[section]
  \newtheorem{corollary}[theorem]{Corollary}
  \newtheorem{proposition}[theorem]{Proposition}
  
  \newtheorem{lemma}[theorem]{Lemma}

\theoremstyle{definition}
  \newtheorem{definition}[theorem]{Definition}
  \newtheorem*{notation}{Notation}

\theoremstyle{remark}
  \newtheorem{remark}[theorem]{Remark}

\newcommand\abs[2][]{#1\lvert#2#1\rvert}
\newcommand\norm[2][]{#1\lVert#2#1\rVert}

\newcommand\ceil[2][]{#1\lceil#2#1\rceil}

\newcommand\lb[1]{#1^-\!} 
\newcommand\ub[1]{#1^+\!} 

\newcommand\bndry{\partial}

\newcommand\Index[1]{#1\index{#1}}
\newcommand\idxsym[1]{} 

\DeclareMathOperator{\intr}{int}
\DeclareMathOperator{\clos}{cl}
\DeclareMathOperator{\dom}{dom}
\DeclareMathOperator{\distortion}{Dist}
\DeclareMathOperator{\graph}{graph}

\newcommand\distance{d}

\newcommand\uint{[0,1]}
\newcommand\setC{\mathcal{C}}
\newcommand\setD{\mathcal{D}}
\newcommand\setL{\mathcal{L}}
\newcommand\setK{\mathcal{K}}
\newcommand\setLS{\mathcal{L}^S}

\newcommand\setY{\mathcal{Y}}

\newcommand\opD{D}                
\newcommand\opN{N}                
\newcommand\opS{S}                
\newcommand\opZ{Z}                
\newcommand\opO{O}
\newcommand\intd{d}               
\newcommand\e{e}                  
\newcommand\inv{^{-1}}
\newcommand\bigoh{\mathcal{O}}

\newcommand\reals{\mathbb{R}}

\newcommand\nats{\mathbb{N}}

\newcommand\ints{\mathbb{Z}}

\newcommand\eps{\varepsilon}
\newcommand\opR{\mathcal{R}}
\newcommand\opT{T}                
\newcommand\slice{\mathcal{S}}
\newcommand\id{\mathrm{id}}

\newcommand\decomp[1]{\bar{#1}}
\newcommand\decomps{\mathcal{\bar{D}}}
\newcommand\pdecomps{\mathcal{\bar{Q}}}
\newcommand\pures{\mathcal{Q}}
\newcommand\lorenzd{\mathcal{\bar{L}}}
\newcommand\cone{H}
\newcommand\du{\oplus}

\newcommand\rdomd{\mathcal{\bar{L}}}
\newcommand\rtype{\bar{\omega}}
\newcommand\attr{\mathcal{A}}

\newcommand\fstret[1]{\tilde{#1}}
\newcommand\word[1]{w(#1)}
\newcommand\dinvset{\bar{\setK}}
\newcommand\trivl{\gamma_\mathrm{triv}^-}
\newcommand\trivr{\gamma_\mathrm{triv}^+}
\newcommand\mfd{\mathcal{W}}
\newcommand\limitset{\Lambda}

\newcommand\graphs{\mathcal{G}}
\newcommand\blok{\bar{\mathcal{B}}}
\newcommand\lbb{b_0}

\newcommand\propref[1]{Proposition~\ref{prop:#1}}
\newcommand\thmref[1]{Theorem~\ref{thm:#1}}
\newcommand\lemref[1]{Lemma~\ref{lem:#1}}
\newcommand\secref[1]{Section~\ref{sec:#1}}
\newcommand\appref[1]{Appendix~\ref{sec:#1}}
\newcommand\remref[1]{Remark~\ref{rem:#1}}
\newcommand\figref[1]{Figure~\ref{fig:#1}}
\newcommand\defref[1]{Definition~\ref{def:#1}}
\newcommand\corref[1]{Corollary~\ref{cor:#1}}

\newcommand\eqnref[1]{(\ref{eq:#1})}

\newcommand\lcv{c_1^-} 
\newcommand\rcv{c_1^+} 

\newcommand\xonn[2]{{#1/#2}}
\newcommand\xonp[2]{{#1/(#2)}}

\newcommand\xnp[2]{\xonp{#1}{#2}}

\newcommand\xo[1]{\xonn{1}{#1}}
\newcommand\xop[1]{\xonp{1}{#1}}

\newcommand\ooan[1]{{\alpha^{-#1}}}
\newcommand\smashooan[1]{{\alpha^{\mrlap{-#1}}}}
\newcommand\ooa{\xonn{1}{\alpha}}

\newcommand\TODO[1]{} 

\def\clap#1{\hbox to 0pt{\hss#1\hss}}

\def\mrlap{\mathpalette\mathrlapinternal}

\def\mathrlapinternal#1#2{\rlap{$\mathsurround=0pt#1{#2}$}}

\begin{document}

\title{On the Hyperbolicity of Lorenz Renormalization}
\author{Marco Martens}
\address{Department of Mathematics \\
         Stony Brook University \\
         Stony Brook NY, USA}
\email{marco@math.sunysb.edu}
\author{Bj\"orn Winckler}
\address{Deparment of Mathematics \\
         KTH \\
         Stockholm, Sweden}
\email{winckler@kth.se}
\date{\today}

\begin{abstract}
  We consider infinitely renormalizable Lorenz maps with real critical exponent
  $\alpha>1$ and combinatorial type which is monotone and satisfies a long
  return condition.  For these combinatorial types we prove the existence of
  periodic points of the renormalization operator, and that each map in the
  limit set of renormalization has an associated unstable manifold.  An
  unstable manifold defines a family of Lorenz maps and we prove that each
  infinitely renormalizable combinatorial type (satisfying the above
  conditions) has a unique representative within such a family.  We also prove
  that each infinitely renormalizable map has no wandering intervals and that
  the closure of the forward orbits of its critical values is a Cantor
  attractor of measure zero.
\end{abstract}

\maketitle

\def\IMSmarkvadjust{0 pt}
\def\IMSmarkhadjust{0 pt}
\def\IMSmarkhpadding{0 pt}
\def\IMSpubltext{Published in modified form:}
\def\SBIMSMark#1#2#3{
 \font\SBF=cmss10 at 10 true pt
 \font\SBI=cmssi10 at 10 true pt
 \setbox0=\hbox{\SBF \hbox to \IMSmarkhpadding{\relax}
                Stony Brook IMS Preprint \##1}
 \setbox2=\hbox to \wd0{\hfil \SBI #2}
 \setbox4=\hbox to \wd0{\hfil \SBI #3}
 \setbox6=\hbox to \wd0{\hss
             \vbox{\hsize=\wd0 \parskip=0pt \baselineskip=10 true pt
                   \copy0 \break%
                   \copy2 \break%
                   \copy4 \break}}
 \dimen0=\ht6   \advance\dimen0 by \vsize \advance\dimen0 by 8 true pt
                \advance\dimen0 by -\pagetotal
	        \advance\dimen0 by \IMSmarkvadjust
 \dimen2=\hsize \advance\dimen2 by .25 true in
	        \advance\dimen2 by \IMSmarkhadjust

%
%
  \openin2=publishd.tex
  \ifeof2\setbox0=\hbox to 0pt{}
  \else 
     \setbox0=\hbox to 3.1 true in{
                \vbox to \ht6{\hsize=3 true in \parskip=0pt  \noindent  
                {\SBI \IMSpubltext}\hfil\break
                \input publishd.tex 
                \vfill}}
  \fi
  \closein2
  \ht0=0pt \dp0=0pt
 \ht6=0pt \dp6=0pt
 \setbox8=\vbox to \dimen0{\vfill \hbox to \dimen2{\copy0 \hss \copy6}}
 \ht8=0pt \dp8=0pt \wd8=0pt
 \copy8
 \message{*** Stony Brook IMS Preprint #1, #2. #3 ***}
}

\SBIMSMark{2012/5}{April 2012}{}

\section{Introduction}




Flows in three and higher dimensions can exhibit chaotic behavior and are far
from being classified.  Understanding higher dimensional flows is important
since these have ties to physical systems, or at least simplifications thereof.
The simplest example is that of the Lorenz equations. This three-dimensional
flow is an approximate model for a convection flow in a box.  In this paper we
study \emph{geometric} Lorenz flows since this class:
\begin{inparaenum}
  \item exhibits a wide range of dynamically complex behavior,
  \item is ``large'' as a subset in the set of three-dimensional flows (in
    particular, it is open), and
  \item is intimately connected with the Lorenz equations and as such has a
    physical significance.\footnote{The reason why we consider geometric
    Lorenz flows instead of the Lorenz equations is that first-return maps of
    geometric flows automatically have nice properties, whereas for the Lorenz
    equations we would have to prove that such first-return maps exist, which
    is hard.}
\end{inparaenum}
We will describe the dynamics of individual infinitely renormalizable geometric
Lorenz flows as well as the structure of the class of these infinitely
renormalizable geometric Lorenz flows. The precise renormalization structure
will be discussed later.

Recall that a \emph{geometric Lorenz flow} is a flow whose associated vector
field has a singularity of saddle type with a two-dimensional stable manifold
$\mfd^s$ and one-dimensional unstable manifold.  The global dynamics of the
flow should be such that there exists a two-dimensional transverse section $S$
to the stable manifold which is divided into two components by the stable
manifold and such that the first-return map $F: S\setminus\mfd^s \to S$ is
well-defined (points on $\mfd^s$ can never return as they end up on the saddle
point which is why $F$ is undefined on~$\mfd^s$).  The final condition is that
$S$ has a smooth  $F$--invariant foliation whose leaves are exponentially
contracted by~$F$.

Under the above conditions $F$ is a well-defined map on the leaves of the
invariant foliation and by taking a quotient over leaves we get an interval map
$f: I\setminus\{c\} \to I$, where $I \subset \reals$ and $c$ corresponds to the
stable manifold.  Such a map is called a \emph{Lorenz map}.  Let $L$ and~$R$ be
the left and right components of $I\setminus\{c\}$, respectively.  From the
construction of~$f$ it follows that $f|_L$ is equal to $-\abs{x}^\alpha$ in a
left neighborhood of~$0$ up to a rescaling of the domain and range, and $f|_R$
is equal to $\abs{x}^\alpha$ in a right neighborhood of~$0$, again up to a
rescaling of the domain and range.  The parameter $\alpha>0$ is the
\emph{critical exponent} which by construction equals the
absolute value of the ratio between the weak stable eigenvalue and the unstable
eigenvalue of the saddle point of the flow.  Note in particular that there is
no ``preferred'' value for $\alpha$, for example it is \emph{not} an integer
generically, it really is just an arbitrary (positive) real number.  The
expanding case $\alpha < 1$ has been studied extensively elsewhere; we will
consider the significantly harder case $\alpha>1$ where there is a delicate
interplay between both expansion \emph{and} contraction.

A Lorenz map is \emph{renormalizable} if there exists an open interval around
the critical point on which the first-return map is again a Lorenz map and the
operator which takes a map to its first-return map is called a
\emph{renormalization operator}.  The critical point divides the return
interval into two halves, the forward orbits of which determine the
\emph{combinatorial type} of the renormalization.  For example, the type
$(01,100)$ encodes that the left half ($\underline{0}1$) is mapped to the right
of the critical point ($0\underline{1}$) and then returns, whereas the right
half ($\underline{1}00$) is first mapped to the left of the critical point
($1\underline{0}0$) and then left again ($10\underline{0}$) before it returns.
We will consider infinitely renormalizable maps (i.e.\ maps with a full forward
orbit under the renormalization operator) of monotone type (i.e.\ types of the
form $(01\dotsm1,10\dotsm0)$) where the number of steps taken to return is much
larger for one half of the return interval than it is for the other half (the
precise condition can be found in \secref{inv-set}).  It is \emph{very}
difficult to deal with arbitrary combinatorial types, so we had to make some
restrictions in order to make any progress.

The main part of this study relates to the hyperbolic properties of the
renormalization operator and shows in particular that this operator has an
expanding invariant cone-field on a renormalization invariant domain.  This
implies that each monotone Lorenz family (see \secref{island-structure}) has a
unique representative for every infinitely renormalizable combinatorial type,
see \thmref{cantor-archipelago}.  Contrast this with the important result of
monotonicity of entropy for families of unimodal maps which essentially states
that every nonperiodic kneading sequence is realized by a unique map in the
family.

We also show that every point in the limit set of the renormalization operator
has an associated unstable manifold and that the intersection of an unstable
manifold and the set of infinitely renormalizable maps is a Cantor set, see
\thmref{monotonicity}.  We believe the unstable manifolds to be
two-dimensional, but are only able to show that their dimension is at least
two.

Regarding the topological properties of the renormalization operator we show
that there exists a periodic point of the renormalization operator for every
periodic combinatorial type, see \thmref{periodic-points}.

The main conclusion for the dynamics of an individual infinitely renormalizable
Lorenz map is the absence of wandering intervals: two Lorenz maps of the same
infinite renormalization type are topologically conjugated.  We prove this
result by showing that infinitely renormalizable maps satisfy the weak Markov
property of \citep{Mar94} and hence cannot have a wandering interval, see
Theorems \ref{thm:noWI} and~\ref{thm:weak-markov}.  This is the first
nonwandering interval result for Lorenz maps.  The nonexistence of wandering
intervals for general Lorenz maps is still wide open and deserves attention.

We also prove that the closure of the orbits of the critical points of an
infinitely renormalizable map is a Cantor attractor of zero Lebesgue measure,
see \thmref{cantorattractor}.

Finally, let us briefly discuss the techniques employed in the proofs.  The
general idea is that by making one return time large we get a first-return map
which is essentially $\abs{x}^\alpha$ up to scaling by maps which are close to
being affine.  This allows us to explicitly calculate an almost invariant set
of the renormalization operator.  This is done in \secref{inv-set} which is the
first major part of this paper.  After this hurdle we are able to prove
properties of individual infinitely renormalizable maps (no wandering
intervals, Cantor attractor, periodic points of renormalization) in Sections
\ref{sec:apb} and~\ref{sec:ppts}.

The second major part is calculating the derivative of the renormalization
operator on a neighborhood of the limit set of renormalization.  This is done
in \secref{the-derivative}.  However, calculating the derivative of the
renormalization operator defined on interval maps is rather hopeless so we need
a better representation of the domain of the renormalization operator.  The
representation we choose are the so-called \emph{decompositions} which are
families of diffeomorphisms parametrized by ordered countable sets (see
\secref{decompositions}).  The renormalization operator is semi-conjugate to an
operator on (essentially) a space of decompositions.  There are two main
reasons why the derivative of this operator is easier to compute:
\begin{inparaenum}
  \item the limit set is essentially a Hilbert cube (see \propref{conv-to-Q}),
    and
  \item deformations in any of the countably many directions are monotone in a
    sense explained in \secref{the-derivative}.
\end{inparaenum}
The first point means that the derivative is just an infinite matrix and the
second allows us to calculate just a few partial derivatives and then make
sweeping estimates for the remaining (countably infinite) directions.

After having computed the derivative we are able to construct an invariant cone
field in \secref{cone-field}.  A by-product of the derivative calculations is
that the derivative is orientation-preserving in the unstable direction and
using this together with the invariant cone field we are able to prove the
association of each combinatorial type with a unique representative of a
monotone family of maps in \secref{island-structure}.  The invariant cone-field
also implies the existence of unstable manifolds in the limit set of
renormalization, see \secref{unstable}.

As a closing remark we point out that in order to deal with arbitrary
critical exponents $\alpha>1$ we had to invent real analytic methods.
To our knowledge, this work is the first to analyze the
hyperbolic structure of the limit of renormalization for arbitrary
critical exponents.

\medskip

\noindent\textbf{Acknowledgments.} A special thank you to Michael Benedicks who
were instrumental in checking many of the technical details, Welington De Melo
who pointed out mistakes in earlier versions of the manuscript, and to Masha
Saprykina who proof read and provided travel funds for Bj\"orn to visit Stony
Brook.  We would also like to thank Viviane Baladi, Kristian Bjerkl\"ov, Denis
Gaidashev, and Joerg Schmeling.

We thank all the institutions and foundations that have supported us in the
course of this work: Institut Mittag-Leffler, NSF, Simons Mathematics and
Physics  Endowment.




\section{The renormalization operator} 
\label{sec:renorm-op}

In this section we define the renormalization operator on Lorenz maps and
introduce notation that will be used throughout.

\begin{definition}
  \idxsym{$Q(x), u, v, c, \alpha$}
  The \Index{standard Lorenz family} $(u,v,c) \mapsto Q(x)$ is
  defined by
  \begin{equation} \label{eq:standard}
    Q(x) = \begin{cases}
      u\cdot\left( 1 - \left( \frac{c-x}{c} \right)^\alpha \right),
        &\text{ if } x \in [0,c), \\
      1 + v\cdot\left( -1 + \left( \frac{x-c}{1-c} \right)^\alpha \right),
        &\text{ if } x \in (c,1], \\
    \end{cases}
  \end{equation}
  where $u \in \uint$, $v \in \uint$, $c \in (0,1)$, and $\alpha>1$.  The
  parameter $\alpha$ is called the critical exponent\index{critical exponent}
  and will be fixed once and for all.
\end{definition}

\begin{remark}
  The parameters $(u,v,c)$ are chosen so that:
  \begin{inparaenum}[(i)]
    \item $u$ is the length of the image of~$[0,c)$,
    \item $v$ is the length of the image of~$(c,1]$,
    \item $c$ is the critical point\index{critical point} (which is the same as
      the point of discontinuity).
  \end{inparaenum}
  Note that $u$ and $1-v$ are the critical values of~$Q$.
\end{remark}

\begin{figure}
  \begin{center}
    \includegraphics{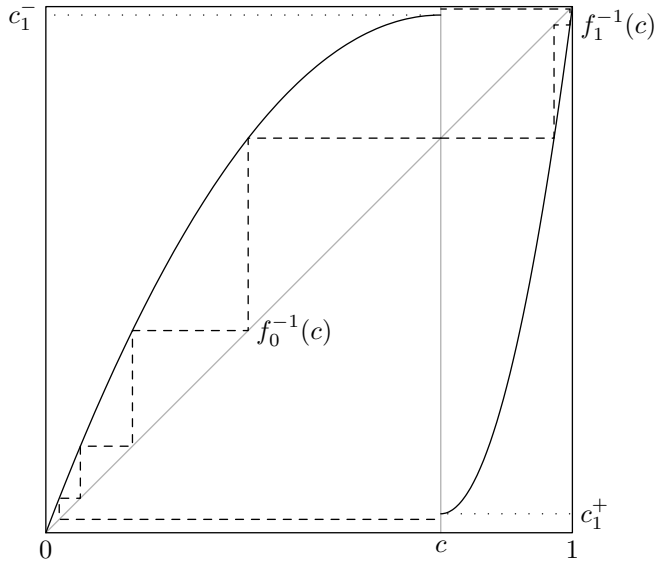}
  \end{center}
  \caption{Illustration of the graph of a $(01,1000)$--renormalizable Lorenz
           map.}
  \label{fig:graph}
\end{figure}

\begin{definition}
  A $\setC^k$--Lorenz map \index{Lorenz map} $f$ on $\uint\setminus\{c\}$ is any
  map which can be written as
  \begin{equation} \label{eq:lorenz-map}
    f(x) = \begin{cases}
      \phi \circ Q(x), &\text{ if }x \in [0,c), \\
      \psi \circ Q(x), &\text{ if }x \in (c,1], \\
    \end{cases}
  \end{equation}
  where $\phi,\psi \in \setD^k$ are orientation-preserving
  $\setC^k$--diffeomorphisms on~$\uint$, called the \Index{diffeomorphic parts}
  of~$f$.  See \figref{graph} for an illustration of a Lorenz map.  The set of
  $\setC^k$--Lorenz maps is denoted $\setL^k$; the subset $\setLS \subset
  \setL^3$ denotes the Lorenz maps with negative Schwarzian derivative (see
  \appref{schwarz} for more information on the Schwarzian derivative).

  A Lorenz map has two \Index{critical values} which we denote
  \[
    \lcv = \lim_{x \uparrow c} f(x)
    \quad\text{and}\quad
    \rcv = \lim_{x \downarrow c} f(x).
  \]
  If $\rcv < c < \lcv$ then $f$ is nontrivial,
  \index{Lorenz map!nontrivial}
  otherwise all points converge to some fixed point under iteration and for
  this reason $f$ is called trivial.
  \index{Lorenz map!trivial}
  Unless otherwise noted, we will always assume all maps to be nontrivial.

  \idxsym{$f = (u,v,c,\phi,\psi)$}
  We make the identification
  \[
    \setL^k = \uint^2 \times (0,1) \times \setD^k \times \setD^k,
  \]
  by sending $(u,v,c,\phi,\psi)$ to $f$ defined by~\eqnref{lorenz-map}.
  Note that $(u,v,c)$ defines $Q$ in~\eqnref{lorenz-map} according
  to~\eqnref{standard}.  For $k\geq2$ this identification turns $\setL^k$ into
  a subset of the Banach space $\reals^3 \times \setD^k \times \setD^k$.  Here
  $\setD^k$ is endowed with the Banach space structure of $\setC^{k-2}$ via the
  nonlinearity operator.  In particular, this turns $\setL^k$ into a metric
  space.  For $k<2$ we turn $\setL^k$ into a metric space by using the usual
  $\setC^k$ metric on $\setD^k$.  See \appref{nonlin} for more information on
  the Banach space $\setD^k$.
\end{definition}

\begin{remark} \label{rem:linstruct}
  It may be worth emphasizing that for $k\geq2$ we are \emph{not} using the
  linear structure induced from $\setC^k$ on the diffeomorphisms~$\setD^k$.
  Explicitly, if $\phi,\psi \in \setD^k$ and $\opN$ denotes the nonlinearity
  operator, then
  \[
    a\phi + b\psi = \opN\inv\left(a\opN\phi + b\opN\psi\right),
    \qquad\forall a,b\in\reals,
  \]
  and
  \[
  \norm{\phi}_{\setD^k} = \norm{\opN\phi}_{\setC^{k-2}}.
  \]
  We call this norm on $\setD^k$ the $\setC^{k-2}$--nonlinearity
  norm.\index{nonlinearity norm} The nonlinearity
  operator\index{nonlinearity operator} $\opN: \setD^k \to \setC^{k-2}$ is a
  bijection and is defined by
  \[
    \opN\phi(x) = \opD \log \opD\phi(x).
  \]
  See \appref{nonlin} for more details on the nonlinearity operator.
\end{remark}

We now define the renormalization operator for Lorenz maps.

\begin{definition} \label{def:renorm}
  A Lorenz map $f$ is \Index{renormalizable} if there exists an interval
  $C\subsetneq\uint$ (properly containing $c$) such that the first-return map
  to~$C$ is affinely conjugate to a nontrivial Lorenz map.  Choose
  $C$ so that it is maximal with respect to these properties.  The
  first-return map affinely rescaled to~$\uint$ is called the
  \Index{renormalization} of~$f$ and is denoted $\opR f$.  The operator $\opR$
  which sends $f$ to its renormalization is called the \Index{renormalization
  operator}.

  \idxsym{$\opR$}
  \idxsym{$C = \clos L \cup R, a, b$}
  Explicitly, if $f$ is renormalizable then there exist minimal positive
  integers $a$ and~$b$ such that the first return map $\fstret{f}$ to~$C$ is
  given by
  \[
    \fstret{f}(x) = \begin{cases}
      f^{a+1}(x), &\text{if $x \in L$,} \\
      f^{b+1}(x), &\text{if $x \in R$,}
    \end{cases}
  \]
  where $L$ and~$R$ are the left and right components of $C \setminus \{c\}$,
  respectively.  The renormalization of~$f$ is defined by
  \[
    \opR f(x) = h\inv \circ \fstret{f} \circ h(x),
    \qquad x \in \uint \setminus \{h\inv(c)\},
  \]
  where $h: \uint \to C$ is the affine orientation-preserving map taking
  $\uint$ to~$C$.  Note that $C$ is chosen maximal so that $\opR f$ is uniquely
  defined.
\end{definition}

\begin{remark}
  We would like to emphasize that the renormalization is assumed to be a
  nontrivial Lorenz map.  It is possible to define the renormalization operator
  for maps whose renormalization is trivial but we choose not to include these
  in our definition.  Such maps can be thought of as degenerate and including
  them makes some arguments more difficult which is why we choose to exclude
  them.
\end{remark}

Next, we wish to describe the combinatorial information encoded in a
renormalizable map.

\begin{definition}
  A \Index{branch} of $f^n$ is a maximal open interval $B$ on which $f^n$ is
  monotone (here maximality means that if $A$ is an open interval which
  properly contains~$B$, then $f^n$ is not monotone on~$A$).

  To each branch $B$ of~$f^n$ we associate a word
  $\word{B} = \sigma_0\dotsm\sigma_{n-1}$ on symbols $\{0,1\}$ by
  \[
  \sigma_j = \begin{cases}
      0 & \text{if $f^j(B)\subset(0,c)$,} \\
      1 & \text{if $f^j(B)\subset(c,1)$,}
      \end{cases}
  \]
  for $j=0,\dots,n-1$.
\end{definition}

\begin{definition}
  Assume $f$ is renormalizable and let $a$, $b$, $L$ and~$R$ be as in
  \defref{renorm}.  The forward orbits of $L$ and~$R$ induce a pair of words
  $\omega = (\word{\hat L},\word{\hat R})$ called the type of
  renormalization\index{renormalization!type of}, where $\hat L$ is the branch
  of~$f^{a+1}$ containing~$L$ and $\hat R$ is the branch of~$f^{b+1}$
  containing~$R$.  In this situation we say that $f$ is
  $\omega$--renormalizable.  See \figref{dynint} for an illustration of these
  definitions.

  Let $\rtype = (\omega_0,\omega_1,\dotsc)$.  If $\opR^n f$ is
  $\omega_n$--renormalizable for $n=0,1,\dotsc$, then we say that $f$ is
  \Index{infinitely renormalizable} and that $f$ has \Index{combinatorial type}
  $\rtype$.  If the length of both words of $\omega_k$ is uniformly bounded
  in~$k$, then $f$ is said to have bounded combinatorial type.
  \index{combinatorial type!bounded}

  \idxsym{$\setL_\omega, \setL_{\rtype}, \setL_\Omega$}
  The set of $\omega$--renormalizable Lorenz maps is denoted $\setL_{\omega}$.
  We will use variations of this notation as well; for $\rtype =
  (\omega_0,\dotsc,\omega_{n-1})$ we let $\setL_{\rtype}$ denote the set of
  Lorenz maps $f$ such that $\opR^i f$ is $\omega_i$--renormalizable, for $i =
  0,\dotsc,n-1$, and similarly if $n=\infty$.  Furthermore, if $\Omega$ is a
  set of types of renormalization, then $\setL_\Omega$ denotes the set of
  Lorenz maps which are $\omega$--renormalizable for some $\omega \in \Omega$.

  We will almost exclusively restrict our attention to \Index{monotone
  combinatorics}, that is renormalizations of type
  \[
    \omega = (0\overbrace{1\dotsm1}^a,1\overbrace{0\dotsm0}^b).
  \]
\end{definition}

\begin{figure}
  \begin{center}
    \includegraphics{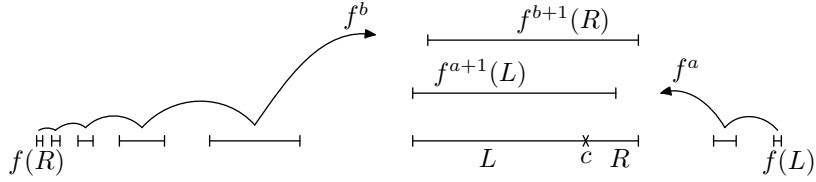}
  \end{center}
  \caption{Illustration of the dynamical intervals of a Lorenz map which is
           $\omega$--renormalizable, with $\omega = (011,100000)$,
           $a=2$, $b=5$.}
  \label{fig:dynint}
\end{figure}

In what follows we will need to know how the five-tuple representation of a
Lorenz map changes under renormalization.  It is not difficult to write down
the formula for any type of renormalization but it becomes a bit messy so we
restrict ourselves to \Index{monotone combinatorics}.  However, first we need
to introduce the zoom operator.

\begin{definition} \label{def:zoom}
  \idxsym{$\opZ(g;I)$}
  The \Index{zoom operator} $\opZ$ takes a diffeomorphism and rescales it
  affinely to a diffeomorphism on~$\uint$.  Explicitly, let $g$ be a map and
  $I$ an interval such that $g|_I$ is an orientation-preserving diffeomorphism.
  Define
  \[
    \opZ(g; I) = \zeta\inv_{g(I)} \circ g \circ \zeta_I,
  \]
  where $\zeta_A: \uint \to A$ is the orientation-preserving affine map which
  takes $\uint$ onto~$A$.  See \appref{nonlin} for more information on zoom
  operators.
\end{definition}

\begin{remark}
  The terminology ``zoom operator'' is taken from \citet{Mar98}, but our
  definition is somewhat simpler since we only deal with orientation-preserving
  diffeomorphisms.  We will use the words `rescale' and `zoom' synonymously.
\end{remark}

\begin{lemma} \label{lem:tuple-renorm}
  If $f = (u,v,c,\phi,\psi)$ is renormalizable of monotone combinatorics, then
  \[
    \opR f = (u',v',c',\phi',\psi')
  \]
  is given by
  \begin{align*}
    u' &= \frac{\abs{Q(L)}}{\abs{U}},&
    v' &= \frac{\abs{Q(R)}}{\abs{V}},&
    c' &= \frac{\abs{L}}{\abs{C}}, \\
    \phi' &= \opZ(f_1^a \circ \phi; U),&
    \psi' &= \opZ(f_0^b \circ \psi; V),
  \end{align*}
  where $U = \phi\inv \circ f_1^{-a}(C)$ and $V = \psi\inv \circ f_0^{-b}(C)$.
\end{lemma}

\begin{proof}
  This follows from two properties of zoom operators:
  \begin{inparaenum}[(i)]
    \item the map $q(x) = x^\alpha$ on~$\uint$ is `fixed' under zooming on
      intervals adjacent to the critical point, that is $\opZ(q; (0,t)) = q$
      for $t\in(0,1)$ (technically speaking we have not defined $\opZ$ in this
      situation, but applying the formula for $\opZ$ will give this result),
      and
    \item zoom operators satisfy $\opZ(h\circ g; I) = \opZ(h; g(I)) \circ
      \opZ(g; I)$.
  \end{inparaenum}
\end{proof}

\begin{notation}
  The notation introduced in this section will be used repeatedly throughout.
  Here is a quick summary.

  A Lorenz map is denoted either $f$ or $(u,v,c,\phi,\psi)$ and these two
  notations are used interchangeably.  Sometimes we write $f_0$ or~$f_1$ to
  specify that we are talking about the left or right branch of~$f$,
  respectively.  Similarly, when talking about the inverse branches of~$f$, we
  write $f_0\inv$ and~$f_1\inv$.  The subscript notation is also used for the
  standard family~$Q$ (so $Q_0$ denotes the left branch, etc.).

  A Lorenz map has one critical point~$c$ and two critical values which we
  denote $\lcv = \lim_{x \uparrow c} f(x)$ and $\rcv = \lim_{x \downarrow c}
  f(x)$.  The critical exponent is denoted~$\alpha$ and is always assumed to be
  fixed to some $\alpha>1$.

  In general we use primes for variables associated with the renormalization
  of~$f$.  For example $(u',v',c',\phi',\psi') = \opR f$.  Sometimes we use
  parentheses instead of primes, for example $\lcv(\opR f)$ denotes the left
  critical value of~$\opR f$.  In order to avoid confusion, we try to use
  $\opD$ consistently to denote derivative instead of using primes.

  \idxsym{$U, V, U_i, V_j$}
  With a renormalizable $f$ we associate a return interval $C$ such that
  $C\setminus\{c\}$ has two components which we denote $L$ and~$R$.  We use the
  notation $a+1$ and~$b+1$ to denote the return times of the first-return map
  to $C$ from $L$ and~$R$, respectively.  The letters $U$ and~$V$ are reserved
  to denote the pull-backs of~$C$ as in \lemref{tuple-renorm}.  We let $U_1 =
  \phi(U)$, $U_{i+1} = f^i(U_1)$ for $i=1,\dots,a$, and $V_1 = \psi(V)$,
  $V_{j+1} = f^j(V_1)$ for $j=1,\dots,b$ (note that $U_{a+1} = C = V_{b+1}$).
  We call $\{U_i\}$ and $\{V_j\}$ the \Index{cycles of renormalization}.
\end{notation}


\section{Generalized renormalization} 
\label{sec:genren}

In this section we adapt the idea of generalized renormalization introduced
by~\citet{Mar94}.  The central concept is the weak Markov property which is
related to the distortion of the monotone branches of iterates of a map.

\begin{definition}
  An interval $C$ is called a \Index{nice interval} of~$f$ if:
  \begin{inparaenum}[(i)]
    \item $C$ is open,
    \item the critical point of~$f$ is contained in $C$, and
    \item the orbit of the boundary of~$C$ is disjoint from~$C$.
  \end{inparaenum}
\end{definition}

\begin{remark}
  A `nice interval' is analogous to a `nice point' for unimodal maps
  \citep[see][]{Mar94}.
  The difference is that for unimodal maps one
  point suffices to define an interval around the critical point (the `other'
  boundary point is a preimage of the first), whereas for Lorenz maps the
  boundary points of a nice interval are independent.  The term `nice' is
  perhaps a bit vague but its use has become established by now.
\end{remark}

\begin{definition} \label{def:transfer}
  Fix $f$ and a nice interval $C$.  The \Index{transfer map} to~$C$ induced
  by~$f$,
  \[
    \opT : \bigcup_{n \geq 0} f^{-n}(C) \to C,
  \]
  is defined by $\opT(x) = f^{\tau(x)}(x)$, where
  \[
    \tau: \bigcup_{n \geq 0} f^{-n}(C) \to \nats
  \]
  is the \Index{transfer time} to~$C$; that is $\tau(x)$ is the smallest
  nonnegative integer~$n$ such that $f^n(x) \in C$.
\end{definition}

\begin{remark}
  Note that:
  \begin{inparaenum}[(i)]
    \item the domain of $\opT$ is open, since $C$ is open by assumption,
      and $f\inv(U)$ is open if $U$ is open (even if $U$ contains a critical
      value), since the point of discontinuity of~$f$ is not in the domain
      of~$f$,
    \item $\opT$ is defined on~$C$ and $\opT|_C$ equals the identity map
      on~$C$.
  \end{inparaenum}
\end{remark}

\begin{proposition} \label{prop:Tbranches}
  Let $\opT$ be the transfer map of~$f$ to a nice interval~$C$.  If $I$ is a
  component of the domain of~$\opT$, then $\tau|_I$ is constant and $I$ is
  mapped monotonically onto~$C$ by~$f^{\tau(I)}$.  Furthermore
  $I,f(I),\dots,f^{\tau(I)}(I)$ are pairwise disjoint.
\end{proposition}

\begin{remark}
  This means in particular that the components of the domain of~$\opT$ are the
  same as the branches of~$\opT$.  In what follows we will use the terminology
  ``a branch of~$\opT$'' interchangeably with ``a component of the domain
  of~$\opT$''.
\end{remark}

\begin{proof}
  If $I = C$ then the proposition is trivial since $\opT|_C$ is the identity
  map on~$C$, so assume that $I \neq C$.

  Pick some $x \in I$ and let $n = \tau(x)$.  Note that $n > 0$ since
  $I \neq C$.  We claim that the branch $B$ of~$f^n$ containing $x$ is mapped
  over~$C$.  From this it immediately follows that $\tau|_I = n$
  and~$f^n(I) = C$.

  Since $f^n|_B$ is monotone and $f(x) \in C$ it suffices to show that
  $f^n(\bndry B) \cap C = \emptyset$.  To this end, let $y \in \bndry B$.  Then
  there exists $0 \leq i < n$ such that $f^i(y) \in \{0,c,1\}$.

  If $f^i(y) \in \{0,1\}$ then we are done, since these points are fixed
  by~$f$.

  So assume that $f^i(y) = c$ and let $J = (x,y)$.  Then
  $f^i(J) \cap \bndry C \neq \emptyset$ since $f^i(x) \notin C$ by minimality
  of $\tau(x)$.  Consequently $f^n(y) \notin C$, otherwise $f^n(J) \subset C$
  which would imply $f^{n-i}(\bndry C) \cap C \neq \emptyset$.  But this is
  impossible since $C$ is nice and hence the claim follows.

  From $\tau(I) = n$ it follows that $I,\dots,f^n(I)$ are pairwise disjoint.
  Suppose not, then $J = f^i(I) \cap f^j(I)$ is nonempty for some $0 \leq i <
  j \leq n$.  But then the transfer time on $I \cap f^{-i}(J)$ is at most
  $i+(n-j)$ which is strictly smaller than~$n$, and this contradicts the fact
  that~$\tau(I) = n$.
\end{proof}

\begin{proposition} \label{prop:Tdomain}
  Assume that $f$ has no periodic attractors and that $\opS f < 0$.  Let $\opT$
  be the transfer map of~$f$ to a nice interval~$C$.  Then the complement of
  the domain of $\opT$ is a compact, $f$--invariant and hyperbolic set (and
  consequently it has zero Lebesgue measure).
\end{proposition}

\begin{proof}
  Let $U = \dom \opT$ and let $\Gamma = \uint \setminus U$.

  Since $U$ is open $\Gamma$ is closed and hence compact (since it is obviously
  bounded).

  By definition $f\inv(U) \subset U$ which implies $f(\Gamma) \subset \Gamma$.

  We can characterize $\Gamma$ as the set of points $x$ such that
  $f^n(x) \notin C$ for all $n \geq 0$.  Since $\opS f < 0$ it follows  that
  $f$ cannot have nonhyperbolic periodic points in~$\Gamma$
  \citep[Theorem~1.3]{Mis81}
  and by assumption $f$ has no periodic
  attractors so $\Gamma$ must be hyperbolic
  \citep[Theorem~III.3.2]{dMvS93}.\footnote{
    The theorems from \citet{dMvS93} that are referenced in this proof are
    stated for maps whose domain is an interval but their proofs go through,
    mutatis mutandis, for Lorenz maps. \label{fn:Tdomain}
  }

  Finally, it is well known that a compact, invariant and hyperbolic set has
  zero Lebesgue measure if $f$ is at least $\setC^{1+\text{H\"older}}$
  \citep[Theorem~III.2.6]{dMvS93}.\footref{fn:Tdomain}
\end{proof}

\begin{definition} \label{def:weak-Markov}
  A map~$f$ is said to satisfy the \Index{weak Markov property} if there exists
  a $\delta > 0$ and a nested sequence of nice intervals $C_1 \supset C_2
  \supset \cdots$, such that $C_n$ contains a $\delta$--scaled neighborhood
  of~$C_{n+1}$ and such that the transfer map to $C_n$ is defined almost
  everywhere, for every $n > 0$.
\end{definition}

\begin{remark}
  If $I \subset J$ are two intervals, then $J$ is said to contain a
  $\delta$--scaled neighborhood of $I$ if $J \setminus I$ consists of two
  components $I_0$ and~$I_1$, and if $\abs{I_k} > \delta\abs{I}$ for $k=0,1$.

  The relevance of this property in \defref{weak-Markov} is that it can be used
  in conjunction with the Koebe lemma to control the distortion of the transfer
  map to~$C_n$.
\end{remark}

\begin{theorem} \label{thm:noWI}
  If $f$ satisfies the weak Markov property, then $f$ has no wandering
  intervals. \index{wandering interval}
\end{theorem}

\begin{proof}
  In order to reach a contradiction assume that there exists a wandering
  interval~$W$ which is not contained in a strictly larger wandering interval.

  Note that the orbit of~$W$ must accumulate on at least one side
  of~$c$.  Otherwise there
  would exist an interval $I$ disjoint from the orbit of~$W$ with
  $c \in \clos I$.  We could then modify $f$ on~$I$ in such a way that the
  resulting map would be a bimodal $\setC^2$--map with nonflat critical points
  and $W$ would still be a wandering interval for the modified map, see
  \figref{wandering}.  However,
  such maps do not have wandering intervals \citep{MdMvS92}.

  Now let $\{C_k\}$ be the sequence of nice intervals that we get from the weak
  Markov property and let $T_k$ denote the transfer map to~$C_k$.  We claim
  that $W \subset \dom T_k$.  To see this, note that $f^{n_k}(W) \cap C_k \neq
  \emptyset$ for some minimal $n_k$, since the orbit of~$W$ accumulates on the
  critical point.  But $C_k$ is a nice interval, so in fact we must have
  $f^{n_k}(W) \subset C_k$, else there would exist $x \in W$ such that
  $f^{n_k}(x) \in \bndry C_k$ and hence the orbit of~$x$ would never enter
  $C_k$ which is impossible since $W$ is wandering and its orbit accumulates on
  the critical point.  This shows that $W$ is contained in the domain of the
  transfer map to $C_k$ as claimed.

  Let $B_k$ be the component of $\dom T_k$ which contains $W$.  By
  \propref{Tbranches} $T_k(B_k) = C_k$.  From the weak Markov property we get a
  $\delta$ (not depending on~$k$) such that $C_k$ contains a $\delta$--scaled
  neighborhood of $C_{k+1}$.  Applying the Macroscopic Koebe lemma we can pull
  this space back to get that $B_k$ contains a $\delta'$--scaled neighborhood
  of~$B_{k+1}$, where $\delta'$ only depends on~$\delta$.

  Now let $B = \bigcap B_k$.  By the above $B_k$ contains a $\delta'$--scaled
  neighborhood of $W$ for every $k$, hence $B$ strictly contains $W$.  By
  \propref{Tbranches} the collection $\{f^i(B_k)\}_{i=0}^{n_k}$ is pairwise
  disjoint for every $k$.  Thus $B$ is a wandering interval which strictly
  contains the wandering interval $W$ (note that $n_k \to \infty$ since
  $\abs{C_k} \to 0$).  This contradicts the maximality of~$W$ and hence $f$
  cannot have wandering intervals.
\end{proof}

\begin{figure}
  \begin{center}
    \includegraphics{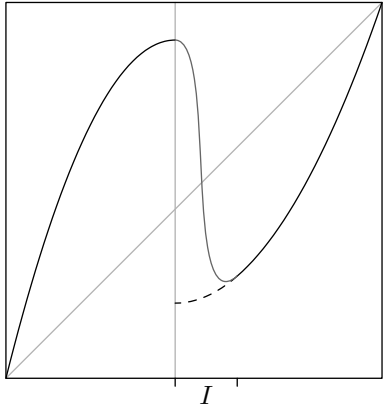}
  \end{center}
  \caption{Illustration showing why the orbit of a wandering interval
  must accumulate on the
  critical point.  If $f$ has a wandering interval whose orbit does not
  intersect some (one-sided) neighborhood $I$ of the critical point, then by
  modifying $f$ on~$I$ according to the gray curve we create a bimodal map
  with a wandering interval.  This is impossible since bimodal maps with
  nonflat critical points do not have wandering intervals.}
  \label{fig:wandering}
\end{figure}

\begin{theorem} \label{thm:ergodicity}
  If $f$ satisfies the weak Markov property, then $f$ is \Index{ergodic}.
\end{theorem}

\begin{proof}
  In order to reach a contradiction, assume that there exist two invariant
  sets $X$ and~$Y$ such that $\abs{X}>0$, $\abs{Y}>0$ and
  $\abs{X \cap Y} = 0$.  Let $\{C_k\}$ be the sequence of nice intervals
  that we get from the weak Markov property.  We claim that
  \[
    \frac{\abs{X \cap C_k}}{\abs{C_k}} \to 1 \quad\text{and}\quad
    \frac{\abs{Y \cap C_k}}{\abs{C_k}} \to 1, \quad\text{as } k \to \infty.
  \]
  Thus we arrive at a contradiction since this shows that $\abs{X \cap Y} > 0$.

  Let $\Gamma_k$ be the complement of the domain of the transfer map to~$C_k$.
  By the weak Markov property $\abs{\Gamma_k} = 0$, hence $\bigcup \Gamma_k$
  also has zero measure.  This and the assumption that $\abs{X} > 0$ implies
  that there exists a density point $x$ which lies in $X$ as well as in the
  domain of the transfer map to $C_k$, for every~$k$.

  Let $B_k$ be the branch of the transfer map to $C_k$ containing $x$, and let
  $\tau_k$ be the transfer time for $B_k$.  We contend that $\abs{B_k} \to 0$.
  If not, there would exist a subsequence $\{k_i\}$ such that
  $B = \bigcap B_{k_i}$ had positive measure, and thus $B$ would be contained
  in a wandering interval (which is impossible by \thmref{noWI}).  Here we
  have used that $C_k$ is a nice interval so the orbit of~$B_k$ satisfies the
  disjointness property of \propref{Tbranches}.

  Since $f^{\tau_k}(B_k) = C_k$ we can use the weak Markov property and the
  Koebe lemma to get that there exists $K < \infty$ (not depending on~$k$) such
  that the distortion of $f^{\tau_k}$ on~$B_k$ is bounded by $K$.  This,
  together with the assumption that $f(X) \subset X$, shows that
  \[
    \frac{\abs{C_k \setminus X}}{\abs{C_k}}
    \leq \frac{\abs{f^{\tau_k}(B_k \setminus X)}}{\abs{f^{\tau_k}(B_k)}}
    \leq K \frac{\abs{B_k \setminus X}}{\abs{B_k}} \to 0,
    \quad\text{as } k \to \infty.
  \]
  The last step follows from $x$ being a density point, since
  $\abs{B_k} \to 0$.

  Now apply the same argument to~$Y$ and the claim follows.
\end{proof}



\section{The invariant set} 
\label{sec:inv-set}

In this section we construct an `invariant' and relatively compact set for the
renormalization operator.  This construction works for types of renormalization
where the return time of one branch is much longer than the other.  This
result will be exploited in the following sections.

\begin{definition} \label{def:Omega-K}
  Fix $\alpha>1$, $\sigma \in (0,1)$, $\beta \in (0,(\sigma/\alpha)^2)$
  and let $\lbb \in \nats$ be a free parameter.
  Define $\Omega$ to be the following (finite) set of monotone types
  \idxsym{$\Omega,\lbb$}
  \begin{equation} \label{eq:Omega}
    \Omega = \big\{ \vphantom{\bigg(}\smash{(0\overbrace{1\dotsm1}^a,1\overbrace{0\dotsm0}^b)}
    \mid \alpha + \sigma \leq a + 1 \leq 2\alpha - \sigma,\;
    \lbb \leq b \leq (1 + (\sigma/\alpha)^2 - \beta)\lbb \big\}.
  \end{equation}
  (Note that $a$ and~$b$ are integers.)  We assume that $\sigma$ has been
  chosen so that the two inequalities involving $a$ have at least one integer
  solution.\footnote{For $\alpha \in (1,2]$ there is exactly one integer
  solution if $\sigma$ is small enough.  For $\alpha>2$ it is possible to
  choose $\sigma$ so that there are at least two solutions.}

  Let $\delta = (1/\lbb)^2$, $\eps = 1 - c$ (note that $\eps$
  depends on $f$) and define
  \idxsym{$\setK,\delta,\eps$}
  \begin{equation} \label{eq:K}
    \setK = \big\{ f \in \setL^1 \mid
    \alpha^{-\lbb/\alpha} \leq \eps \leq
    \theta\alpha^{-\lbb \sigma/\alpha^2},\;
    \distortion{\phi} \leq \delta,\; \distortion{\psi} \leq \delta \big\},
  \end{equation}
  where $\theta>1$ is a constant not depending on $\lbb$.\footnote{The
  constant $\theta$ is given by \propref{ubeps}.}  We assume that $\lbb$ is
  large enough for the two inequalities involving $\eps$ to have at least one
  solution.
\end{definition}

We are going to show that $\setK$ is `invariant' under the restriction
of~$\opR$ to types in~$\Omega$ as long as $\lbb$ is large enough.  Recall
that $f \in \setLS_\Omega$ if and only if $f$ has negative Schwarzian
derivative and is $\omega$--renormalizable for some $\omega\in\Omega$.

\begin{theorem} \label{thm:K-inv}
  If $f \in \setLS_\Omega$ and $1 - \rcv(\opR f) \geq \lambda > 0$ for some
  constant $\lambda$ (not depending on $\lbb$), then
  \[
    f \in \setK \implies \opR f \in \setK,
  \]
  for $\lbb$ large enough.
\end{theorem}

The condition on $\rcv(\opR f)$ is a bit unpleasant but we need it to exclude
maps such that $\eps(\opR f)$ is too small for us to deal with.  This situation
occurs when the right branch of the renormalization is trivial.  We
can work around this problem by considering twice renormalizable maps because
such maps will automatically satisfy the condition on $\rcv(\opR f)$, and this
leads us to:

\begin{theorem} \label{thm:K-inv2}
  If both $f \in \setLS_\Omega$ and $\opR f \in \setLS_\Omega$, then
  \[
    f \in \setK \implies \opR f \in \setK.
  \]
  for $\lbb$ large enough.
\end{theorem}

The proofs of Theorems \ref{thm:K-inv} and~\ref{thm:K-inv2} can be found at the
end of this section.

\begin{remark}
  The full family theorem \citep{MdM01} implies that:
  \begin{inparaenum}[(i)]
    \item for every $\lambda\in(0,1)$ there exists $f \in \setLS_\Omega \cap
      \setK$ such that $\rcv(\opR f) \geq 1 - \lambda$ (e.g.\ any
      $f\in\setK$ can be deformed in the $(u,v)$ directions in such a
      way that $f$ is renormalizable to a map such that $\rcv(\opR f)=0$), and
    \item $\setK$ intersects the set of twice renormalizable maps (of
      \emph{any} combinatorics).
  \end{inparaenum}
  This shows that both theorems above are not vacuous.
\end{remark}

The main reason for introducing the set $\setK$ is the following:

\begin{proposition} \label{prop:K-relcpt}
  $\setK$ is relatively compact in $\setL^0$.
\end{proposition}

\begin{proof}
  Clearly $\eps(f)$ for $f \in \setK$ lies inside a compact set in $(0,1)$.
  Hence we only need to show that the ball $B = \{ \phi \in \setD^1(\uint) \mid
  \distortion \phi \leq \delta \}$ is relatively compact in~$\setD^0(\uint)$.
  This is an application of the Arzel\`a--Ascoli theorem; if $\{\phi_n \in B\}$
  then $\abs{\phi_n(y)-\phi_n(x)} \leq \e^\delta \abs{y-x}$ hence this sequence
  is equicontinuous (as well as uniformly bounded), so it has a uniformly
  convergent subsequence.
\end{proof}

The rest of this section is devoted to the proof of \thmref{K-inv}.  We will
need the following expressions for the inverse branches of~$f$ which can be
derived from equations \eqnref{standard} and~\eqnref{lorenz-map}:

\begin{align}
  f_0\inv(x) &= c - c \left(
    \frac{\abs{\phi\inv([x,\lcv])}}{\abs{\phi\inv([0,\lcv])}}
    \right)^{\mrlap\ooa}, \label{eq:f0inv} \\
  f_1\inv(x) &= c + (1-c) \left( 1 -
    \frac{\abs{\psi\inv([x,1])}}{\abs{\psi\inv([\rcv,1])}}
    \right)^{\mrlap\ooa}. \label{eq:f1inv}
\end{align}

The following lemma gives us control over certain backward orbits of the
critical point.  This is later used to control the critical values and the
derivative of the first-return map.  The underlying idea for these results is
that the backward orbit of~$c$ under $f_0$ initially behaves like a root and
eventually like a linear map whose multiplier is determined by $\opD f(0)$,
whereas the backward orbit of~$c$ under $f_1$ behaves like a linear map whose
multiplier is determined by $\opD f(1)$.

\begin{lemma} \label{lem:invorbits}
  There exist $\mu > 0$ and $\nu, k \in (0,1)$ such that if $f \in
  \setL^1_\Omega$, $\distortion\phi \leq \delta$, $\distortion\psi \leq
  \delta$, $\eps \leq k$, and $\eps \geq \gamma^{\alpha^{\lbb}}$ for some
  $\gamma\in(0,1)$ not depending on~$\lbb$, then
  \begin{equation*}
    \frac{f_1\inv(c) - c}{\eps} \geq 1 - \mu\eps
    \quad\text{and}\quad
    \frac{c - f_0^{-n}(c)}{c} \geq \nu \eps^\smashooan{n},
  \end{equation*}
  for $\lbb$ large enough.
\end{lemma}

\begin{proof}
  We claim that
  \begin{align}
    f_1\inv(c) - c &\geq \eps \cdot \left( 1 -
      \frac{\e^\delta \eps}{c-f_0^{-b}(c)}
      \right)^{\mrlap\ooa}, \label{eq:invorbc1} \\
    c - f_0^{-n}(c) &\geq c \e^{-\delta/(\alpha-1)}
      \big( f_1\inv(c) - c \big)^\smashooan{n}. \label{eq:invorbc0}
  \end{align}
  Assume for the moment that these equations hold.  The idea of the proof is
  that if we have some initial lower bound on $f_1\inv(c)-c$ then we can plug
  that into \eqnref{invorbc0} (with $n=b$), and this bound in turn can be
  plugged back into \eqnref{invorbc0} to get a new lower bound on
  $f_1\inv(c)-c$.  We will show that iterating the initial bound in this way
  will actually improve it and that this iterative procedure will lead to the
  desired statement.  Finally we show that there exists an intial bound that is
  good enough to start off the iteration.

  To begin with consider \eqnref{invorbc1}.  This equation follows from a
  computation using \eqnref{f1inv} and the fact that $1 - \rcv > c -
  f_0^{-b}(c)$ holds for monotone combinatorics.

  Next, we prove \eqnref{invorbc0}.  Apply \eqnref{f0inv} to get
  \[
    f_0\inv(x) \leq c - c \left( \e^{-\delta} \frac{\lcv-x}{\lcv}
    \right)^{\mrlap\ooa}, \qquad x \leq \lcv.
  \]
  Since $\lcv \leq 1$ this implies that
  \begin{equation} \label{eq:base}
    f_0\inv(c) \leq c - c\cdot \e^{-\delta/\alpha} (\lcv-c)^\ooa,
  \end{equation}
  and if $x<c$ then we can use that $c < \lcv$ to get
  \begin{equation} \label{eq:ubf0inv}
    f_0\inv(x) \leq c - c\cdot \e^{-\delta/\alpha}
    \left( 1 - \frac{x}{c} \right)^{\mrlap\ooa}.
  \end{equation}
  Using \eqnref{base} and \eqnref{ubf0inv} we get (note that $f_0\inv(c) < c$):
  \[
    f_0^{-2}(c) \leq c - c\cdot \e^{-\delta/\alpha}
      \left( 1 - \frac{f_0\inv(c)}{c} \right)^\ooa
      \leq c - c\cdot \e^{-\delta(1+\alpha\inv)/\alpha}
      \left( \lcv - c \right)^\ooan{2}.
  \]
  By repeately applying \eqnref{ubf0inv} to the above inequality we arrive at
  \[
    f_0^{-n}(c) \leq c - c \cdot \exp\left\{-\frac{\delta}{\alpha}
      \left(1+\dotsm+\alpha^{-(n-1)}\right)\right\} \cdot
      \left( \lcv - c \right)^\smashooan{n},
  \]
  which together with the fact that $1+\dots+\alpha^{-n} < \alpha/(\alpha-1)$
  proves \eqnref{invorbc0}.

  Having proved \eqnref{invorbc1} and~\eqnref{invorbc0} we now continue the
  proof of the lemma.  Note that the left-hand side of \eqnref{invorbc1}
  appears in the right-hand side of~\eqnref{invorbc0} and vice versa.  Thus we
  can iterate these inequalities once we have \emph{some} bound for either of
  them.  To this end, suppose $f_1\inv(c)-c \geq t \eps$, for some $t > 0$.  If
  we plug this into \eqnref{invorbc0} and then plug the resulting bound into
  \eqnref{invorbc1}, we get that
  \begin{equation} \label{eq:h}
    f_1\inv(c) - c \geq \eps \cdot \left( 1 -
    \frac{\e^{\delta\alpha/(\alpha-1)}
    \eps^{1-\alpha^{-b}}}{c t^{\alpha^{-b}}} \right)^\ooa
    = \eps h(t).
  \end{equation}

  We claim that the map $h$ has two fixed points: a repeller $t_0$ close to~$0$
  and an attractor $t_1$ close to~$1$.  To see this, solve the fixed point
  equation $t = h(t)$ to get
  \begin{equation} \label{eq:h-fp}
    t^\ooan{b} (1 - t^\alpha) = \eps^{1-\alpha^{-b}}
    \e^{\delta\alpha/(\alpha-1)} / c.
  \end{equation}
  Let $g(t) = t^\ooan{b} (1 - t^\alpha)$ and let $\rho = \eps^{1-\alpha^{-b}}
  \e^{\delta\alpha/(\alpha-1)} / c$.  Note that $g(0)=0$, $g(1)=1$, and $g$ has
  exactly one turning point $\tau$ at which $g(\tau) > \rho$ for $b$ large
  enough.  This shows that $g(t)=\rho$ has two solutions $t_0 < t_1$.  That
  $t_0$ is repelling and $t_1$ attracting (for $h$) follows from the fact that
  $h(t)^\alpha \to -\infty$ as $t \downarrow 0$ and $h(t) \to 1$ as $t \uparrow
  \infty$.

  We now find bounds on the fixed points of~$h$.  Solving $\opD g(\tau) = 0$
  gives
  \begin{equation} \label{eq:turning}
    \tau = (\alpha^{b+1} + 1)^{-1/\alpha}.
  \end{equation}
  Hence \eqnref{h-fp} shows that
  \begin{equation} \label{eq:t0}
    \rho = t_0^\ooan{b} (1 - t_0^\alpha) > t_0^\ooan{b} (1 - \tau^\alpha)
    \implies t_0 < \left( \frac{\rho}{1-\tau^\alpha} \right)^{\alpha^b}
  \end{equation}
  and
  \begin{equation} \label{eq:t1}
    \rho = t_1^\ooan{b} (1 - t_1^\alpha) > \tau^\ooan{b} (1 - t_0^\alpha)
    \implies t_1 > \left( 1 - (\tau\eps)^{-\ooan{b}} \tilde\rho \eps
    \right)^\ooa,
  \end{equation}
  where $\rho = \tilde\rho \eps^{1-\ooan{b}}$ so that $\tilde\rho$ is a
  constant not depending on~$\lbb$.
  By assumption
  \[
    \eps^{-\ooan{b}} = (1/\eps)^\ooan{b} \leq (1/\gamma^{\alpha^{\lbb}})^\ooan{b} \leq 1/\gamma
  \]
  and $\tau^{-\ooan{b}} \to 1$ as $b \to \infty$ by \eqnref{turning}, so
  \eqnref{t1} shows that there exists a constant $\mu$ such that
  \begin{equation} \label{eq:t1-2}
    t_1 > 1 - \mu\eps.
  \end{equation}

  All that is need to complete the proof is \emph{some} initial bound
  $f_1\inv(c)-c \geq t' \eps$ such that $t' > t_0$, because then $h^i(t') \to
  t_1$ as $i \to \infty$, which together with \eqnref{h} and~\eqnref{t1-2}
  shows that
  \[
    f_1\inv(c)-c \geq \eps h(t_1) = \eps t_1 = \eps(1 - \mu \eps).
  \]
  Plugging this into \eqnref{invorbc0} also shows that
  \[
    c - f_0^{-n}(c) \geq c \e^{\delta/(\alpha-1)}
    \big( (1-\mu\eps)\eps \big)^\ooan{n} > 
    c \e^{\delta/(\alpha-1)} (1-\mu\eps) \eps^\ooan{n}
    = c \nu \eps^\smashooan{n}.
  \]

  To get an initial bound $t'$ we use the fact that $f_1\inv(c) - c > \abs{R}$
  and look for a bound on $\abs{R}$.  Since $\opR f$ is nontrivial we have
  $f^{b+1}(R) \supset R$, which implies
  \[
    \abs{R} \leq \abs{f^b\left( f(R) \right)}
      \leq \max_{x<c} f'(x)^b \cdot \e^\delta \abs{Q(R)}
      \leq (\e^\delta u \alpha / c)^b \e^\delta v \left( \abs{R}/\eps
        \right)^\alpha
  \]
  and thus
  \begin{equation}
    f_1\inv(c) - c > \abs{R} \geq \eps \cdot \left(
      \frac{c \eps^\xo{b}}{\alpha \e^{\delta(b+1)/b}}
      \right)^\xnp{b}{\alpha-1} = \eps t'.
    \end{equation}
  Here $t'$ is of the order $\eps^{1/(\alpha-1)}\alpha^{-b}$ whereas $t_0$ is
  of the order $\eps^{\alpha^b}$, so $t' > t_0$ for $\lbb$ large enough.  To
  see this, solve $t' > (\rho/(1-\tau^\alpha))^{\alpha^b}$ for $\eps$ to get
  \[
    \log\eps < \frac{\alpha-1}{(\alpha-1)\alpha^b - \alpha}
    \cdot \log\left\{
    \left(\frac{c}{\alpha\e^{\delta(b+1)/b}}\right)^{b/(\alpha-1)}
    \left( \frac{c(1-\tau^\alpha)}{\e^{\delta\alpha/(\alpha-1)}}
    \right)^{\alpha^b} \right\}.
  \]
  The right-hand side tends to $\log\{c\e^{-\delta\alpha/(\alpha-1)}\}$ as $b
  \to \infty$, so it suffices to choose
  \[
    \eps \leq \bigoh\left(
    \exp\left\{-\frac{\delta\alpha}{\alpha-1}\right\}\right)
  \]
  and $t' > t_0$ will hold (for $\lbb$ large enough).
\end{proof}

The next lemma is the reason why we chose $\eps$ to be of the order
$\alpha^{-\lbb\cdots}$.  The previous lemma is first used to show that the
region where the backward orbit of~$c$ under $f_0$ is governed by a root
behavior is escaped after $\log b$ steps, and that the remaining $b - \log b$
steps are then governed by the fixed point at~$0$.  The choice of $\eps$ will
make sure that the linear behavior dominates the root behavior and hence $\rcv$
will approach the fixed point at~$0$ as $b$ is increased.

\begin{lemma} \label{lem:crit-vals}
  There exists $K$ such that if $f \in \setL^1_\Omega \cap \setK$, then $1-\lcv
  < K\eps^2$.  Also, $\rcv \to 0$ exponentially in $\lbb$ as $\lbb \to
  \infty$.
\end{lemma}

\begin{remark} \label{rem:uv}
  This lemma also implies that the parameters $u$ and~$v$ are close to one for
  $f \in \setL^1_\Omega \cap \setK$ since $\phi(u) = \lcv$ and $\psi(1-v) =
  \rcv$.  Hence, for example
  \[
    1-u = \phi\inv(1) - \phi\inv(\lcv) = \abs{\phi\inv([\lcv,1])} \leq
    \e^\delta \abs[\big]{[\lcv,1]} < K\e^\delta \eps^2,
  \]
  and
  \[
    1-v = \psi\inv(\rcv) - \psi\inv(0) = \abs{\psi\inv([0,\rcv])} \leq
    \e^\delta \abs[\big]{[0,\rcv]} \leq K' \e^{-\lbb},
  \]
  for some $K'$.
\end{remark}

\begin{proof}
  The proof is based on the fact that $\rcv < f_0^{-b}(c)$ and $\lcv >
  f_1^{-a}(c)$ for monotone combinatorics, so we can use \lemref{invorbits} to
  bound the position of the critical values.

  \lemref{invorbits} shows that
  \[
    1 - \lcv < 1 - f_1\inv(c) \leq 1 - c - (1-\mu\eps)\eps = \mu\eps^2,
  \]
  which proves the statement about about~$\lcv$.

  Next, let $n = \ceil{\log_\alpha \lbb}$.  Then $\alpha^{-n} \leq 1/\lbb$
  and $\eps^\ooan{n} \geq \eps^{1/\lbb}\!\!$, so applying \lemref{invorbits}
  again we get
  \[
    \frac{f_0^{-n}(c)}{c} \leq 1 - \nu(\lb{\eps})^{1/\lbb}
    = 1 - \nu \alpha^{-\sigma}.
  \]
  Thus $f_0^{-n}(c)$ is a uniform distance away from~$c$.  Since $\lbb -
  \ceil{\log_\alpha \lbb} \to \infty$, and since $0$ is an attracting fixed
  point for $f_0\inv$ with uniform bound on the multiplier, it follows that
  $f_0^{-b}(c)$ approaches~$0$ exponentially as $\lbb \to \infty$.  This
  proves the statement about~$\rcv$.
\end{proof}

Now that we have control over the critical values we can estimate the
derivative of the return map.  The derivative of $f_1^a$ is easy to control
since $f_1^a$ is basically a linear map on a neighborhood of~$f(L)$.  However,
the derivative of $f_0^b$ is a bit more delicate and we are only able to
estimate it on a subset of $f(R)$.  The idea is to split the derivative
calculation into two regions; one expanding region governed by the fixed point
at~$0$ and one contracting region in the vicinity of the critical point.  The
choice of $\eps$ will ensure that the expanding region dominates the
contracting region if $b$ is sufficiently large.

We will need the following expressions for the derivatives of the inverse
branches of~$f$:
\begin{align}
  \opD f_0\inv(x) = \frac{c}{\alpha} \cdot
    \frac{\opD \phi\inv(x)}{u}
    \left( \frac{ \abs{\phi\inv([0,\lcv])} }{ \abs{\phi\inv([x,\lcv])} }
    \right)^{\mrlap{1-1/\alpha}}, \label{eq:df0inv} \\
  \opD f_1\inv(x) = \frac{\eps}{\alpha} \cdot
    \frac{\opD \psi\inv(x)}{v}
    \left( \frac{ \abs{\psi\inv([\rcv,1])} }{ \abs{\psi\inv([\rcv,x])} }
    \right)^{\mrlap{1-1/\alpha}}. \label{eq:df1inv}
\end{align}
The above equations can be derived from \eqnref{f0inv} and~\eqnref{f1inv}.

\begin{lemma} \label{lem:derivs}
  There exists $K$ such that if $f \in \setL^1_\Omega \cap \setK$, then
  \begin{align*}
    K\inv(\eps/\alpha)^a \leq &\opD f_1^{-a}(x)
      \leq K (\eps/\alpha)^a,&
      &\forall x > f_0\inv(c), \\
    K\inv \alpha^{-b} \eps^{-1+\ooan{b}} \leq &\opD f_0^{-b}(c)
      \leq K \alpha^{-b} \eps^{-1+\smashooan{b}}.
  \end{align*}
\end{lemma}

\begin{proof}
  We start by proving the lower bound on $\opD f_1^{-a}$.  From \eqnref{df1inv}
  we get $\opD f_1\inv(x) \geq \e^{-\delta} \eps/\alpha$ and hence
  \[
    \opD f_1^{-a}(x) \geq \e^{-a\delta} (\eps/\alpha)^a,
    \qquad \forall x \in [\rcv,1].
  \]
  Note that $\e^{-a\delta}$ has a lower bound that does not depend on $\lbb$,
  so the above equation shows that $\opD f_1^{-a}(x) \geq K\inv
  (\eps/\alpha)^a$ for some $K$ (not depending on~$\lbb$).

  Next consider the upper bound on $\opD f_1^{-a}$.  Use $\lcv \leq 1$
  and~\eqnref{f0inv} to see that
  \begin{equation} \label{eq:lbf0invc}
    f_0\inv(c) \geq c \left( 1 - (\e^\delta \eps)^\ooa \right).
  \end{equation}
  Equation \eqnref{df1inv}, the fact that $1-\rcv\leq1$, and the assumption
  that $x>f_0\inv(c)$ together imply that
  \begin{equation} \label{eq:lbDf1inv}
    \opD f_1\inv(x) \leq \frac{\eps \e^\delta}{\alpha v}
    \left( \e^\delta \frac{1-\rcv}{x-\rcv} \right)^{1-1/\alpha}
    \leq \frac{\e^\delta}{v}
    \left( \frac{\e^\delta}{f_0\inv(c)-\rcv} \right)^{1-1/\alpha}
    \cdot \frac{\eps}{\alpha}
  \end{equation}
  Equation \eqnref{lbf0invc} and \lemref{crit-vals} show that $f_0\inv(c)-\rcv$
  has a lower bound that is independent of $\lbb$ and \remref{uv} can be used
  to bound $v$.  Hence the expression in front of $\eps/\alpha$ in
  \eqnref{lbDf1inv} has an upper bound that does not depend on $\lbb$.  Since
  $x > f_1\inv(c)$ implies that $f_1^{-i}(x)
  > f_0\inv(c)$ for all $i=1,\dotsc,a$, the previous argument and
  \eqnref{lbDf1inv} shows that
  \[
    \opD f_1^{-a}(x) \leq K (\eps/\alpha)^a,
  \]
  for some $K$ (not depending on $\lbb$).

  \medskip
  We now turn to proving the bounds on $\opD f_0^{-b}(c)$.  Equation
  \eqnref{df0inv} shows that
  \begin{equation} \label{eq:estDf0inv}
    \frac{c\e^{-\delta}}{\alpha} \left(
    \e^{-\delta} \frac{\lcv}{\lcv - x}
    \right)^{1-1/\alpha}\!\!\!\! \leq
    \opD f_0\inv(x) \leq \frac{\e^\delta}{\alpha u} \left(
    \e^\delta \frac{\lcv}{\lcv - x}
    \right)^{\mrlap{1-1/\alpha}}.
  \end{equation}

  The upper bound in \eqnref{estDf0inv} gives
  \begin{equation} \label{eq:ubDf0b}
    \opD f_0^{-b}(x) = \prod_{i=0}^{b-1} \opD f_0\inv\big( f_0^{-i}(x) \big)
    \leq \left( \frac{\e^{2\delta}}{\alpha u} \right)^{\! b} \cdot
    \prod_{i=0}^{b-1} \left( \frac{\lcv}{\lcv - f_0^{-i}(x)}
    \right)^{\mrlap{1-1/\alpha}}.
  \end{equation}
  The expression before the last product is bounded by $K\alpha^{-b}$ for some
  constant $K$ since:
  \begin{inparaenum}[(i)]
    \item $b\delta \leq (1+(\sigma/\alpha)^2-\beta)\lbb /
      (\lbb)^2 \to 0$ by \eqnref{K}, and
    \item $u^b \geq (1-\bigoh(\eps^2))^b$ by \remref{uv} and $b\eps \leq
      b\theta\alpha^{-\lbb \sigma/\alpha^2} \to 0$, so $u^b$ has a lower bound
      which does not depend on $b$.
  \end{inparaenum}

  The lower bound in \eqnref{estDf0inv} similarly shows that
  \begin{equation} \label{eq:lbDf0b}
    \opD f_0^{-b}(x) \geq \left( \frac{c}{\alpha\e^{2\delta}} \right)^{b} \cdot
    \prod_{i=0}^{b-1} \left( \frac{\lcv}{\lcv - f_0^{-i}(x)}
    \right)^{\mrlap{1-1/\alpha}}.
  \end{equation}
  The expression before the product is bounded by $K\inv \alpha^{-b}$ for some
  constant $K$ since $b\delta \to 0$ as in (i) above, and $c^b \geq
  (1-\theta\alpha^{-\lbb \sigma/\alpha^2})^b$ so $c^b$ has a lower bound
  independent of $b$.

  The product in \eqnref{ubDf0b} and \eqnref{lbDf0b} is the same, so we will
  look for bounds on this product next.  We claim that there exists constants
  $\gamma,\rho > 0$ such that
  \begin{equation} \label{eq:f0invorb}
    \e^{-\delta / (\alpha - 1)} \leq
    \frac{\lcv - f_0^{-n}(x)}{\lcv} \cdot
    \left( \frac{\lcv - x}{\lcv} \right)^{-\ooan{n}}
    \!\!\! \leq \e^{\delta/(\alpha-1)}
    \big( 1 + \gamma\eps^\rho\big)^{\alpha/(\alpha-1)}\!,
  \end{equation}
  for $x \leq c$.  Assume that this holds for the moment (we will prove it
  shortly).
  
  Equations \eqnref{ubDf0b} and~\eqnref{f0invorb} show that
  \[
    \opD f_0^{-b}(c) \leq \frac{K}{\alpha^b} \left\{
    \prod_{i=0}^{b-1} \e^{\delta/(\alpha-1)}
    \left( \frac{\lcv-c}{\lcv} \right)^{-\ooan{i}} \right\}^{1-1/\alpha}
    \!\!
    = \frac{K \e^{b\delta/\alpha}}{\alpha^b}
    \left( \frac{\lcv-c}{\lcv}
    \right)^{\mrlap{-1+\ooan{b}}}.
  \]
  (The equality follows from a computation using the fact that the logarithm of
  the above product is a geometric sum.)  From \lemref{crit-vals} we get
  \[
    \left( \frac{\lcv-c}{\lcv} \right)^{-1+\ooan{b}}
    \!\!\!\leq
    \left( \eps \frac{1-K_1\eps}{1-K_1\eps^2} \right)^{-1+\ooan{b}}
    \!\!\!\leq K_2\eps^{-1+\smashooan{b}},
  \]
  and hence
  \[
    \opD f_0^{-b}(c) \leq 
    \frac{K_3 \e^{b\delta/\alpha}}{\alpha^b} \eps^{{-1+\ooan{b}}}.
  \]
  Since $b\delta \to 0$ this finishes the proof of the upper bound on $\opD
  f_0^{-b}(c)$.
  
  Similarly, \eqnref{lbDf0b} and~\eqnref{f0invorb} show that
  \[
    \opD f_0^{-b}(c) \geq
    \left( K\alpha^b \e^{b\delta/\alpha} (1+\gamma\eps^\rho)^b
    \right)\inv
    \left( \frac{\lcv-c}{\lcv}
    \right)^{\mrlap{-1+\ooan{b}}}.
  \]
  Use $\lcv < 1$ to get $(\lcv-c)/\lcv = 1 - c/\lcv < 1-c =\eps$.
  This finishes the proof of the lower bound of $\opD f_0^{-b}(c)$, since:
  \begin{inparaenum}[(i)]
    \item $b\delta \to 0$, and
    \item $(1+\gamma\eps^\rho)^b \to 1$ since $b\eps^\rho \leq b \theta
      \alpha^{-\rho \lbb \sigma/\alpha^2} \to 0$ for any $\rho>0$.
  \end{inparaenum}

  \medskip
  It only remains to prove the claim \eqnref{f0invorb}.  We start with the
  lower bound.  From \eqnref{f0inv} and $c < \lcv$ (the latter follows from
  \lemref{crit-vals}) we get that
  \[
    \frac{\lcv - f_0\inv(x)}{\lcv} > \frac{c-f_0\inv(x)}{c}
    \geq \e^{-\delta/\alpha} \left( \frac{\lcv - x}{\lcv}
    \right)^{\mrlap{\ooa}}.
  \]
  Hence
  \[
    \frac{\lcv - f_0^{-2}(x)}{\lcv}
    \geq \e^{-\delta/\alpha} \left( \frac{\lcv - f_0\inv(x)}{\lcv}
    \right)^\ooa\!\!
    \geq \e^{-\delta(1+1/\alpha)/\alpha} \left( \frac{\lcv - x}{\lcv}
    \right)^{\mrlap{\ooan{2}}}.
  \]
  so an induction argument and $1+\dotsb+1/\alpha^{n-1} < \alpha/(\alpha-1)$
  finishes the proof of the lower bound of~\eqnref{f0invorb}.

  We finally prove the upper bound of \eqnref{f0invorb}.  From
  $c < \lcv\! < 1$ and~\eqnref{f0inv} we get
  \begin{equation} \label{eq:const-derivs}
    \frac{c}{\lcv}\frac{\lcv - f_0\inv(x)}{c - f_0\inv(x)}
    \leq \frac{(1-c) + (c-f_0\inv(x))}{c - f_0\inv(x)}
    \leq 1 + \frac{\eps\e^{\delta/\alpha}}{c (\lcv-x)^\ooa}.
  \end{equation}
  Assuming that $x \leq c$ we get $\lcv-x \geq \lcv-c$.\footnote{The condition
  $x\leq c$ is unnecessarily strong here; we could get away with $\lcv-x \geq
  k\eps^t$ for some constant $t$ close to (but smaller than) $\alpha$.} By
  \lemref{crit-vals}, $\lcv-c \geq k\eps$ which together with \eqnref{f0inv}
  and~\eqnref{const-derivs} shows that
  \[
    \frac{\lcv-f_0\inv(x)}{\lcv} \leq 
    \left( 1 + \frac{\e^{\delta/\alpha}}{c k^\ooa} \eps^{1-\ooa} \right)
    \frac{c-f_0\inv(x)}{c}
    \leq (1+\gamma\eps^\rho) \e^{\delta/\alpha}
    \left( \frac{\lcv-x}{\lcv} \right)^{\mrlap{\ooa}},
  \]
  for some constant $\gamma>0$ and $\rho = 1-1/\alpha > 0$.  This shows that
  \begin{align*}
    \frac{\lcv-f_0^{-2}(x)}{\lcv}
    &\leq (1+\gamma\eps^\rho) \e^{\delta/\alpha}
    \left( \frac{\lcv-f_0\inv(x)}{\lcv} \right)^\ooa \\
    &\leq \left( (1+\gamma\eps^\rho) \e^{\delta/\alpha} \right)^{1+1/\alpha}
    \left( \frac{\lcv-x}{\lcv} \right)^{\mrlap{\ooan{2}}},
  \end{align*}
  so an induction argument and $1+\dotsb+1/\alpha^{n-1} < \alpha/(\alpha-1)$
  finishes the proof of the upper bound of~\eqnref{f0invorb}.
\end{proof}

Armed with the above lemmas we can start proving invariance.  The first step is
to show that $\eps(\opR f)$ is small.  The proof is complicated by the fact
that we do not know anything about $\eps(\opR f)$.  Once we find some
bound on $\eps(\opR f)$ we can show that it in fact is very small.

\begin{proposition} \label{prop:ubeps}
  There exists $\theta>0$ (not depending on $\lbb$) such that
  \[
    f \in \setLS_\Omega \cap \setK \implies
    \eps(\opR f) \leq \theta\alpha^{-\lbb \sigma/\alpha^2},
  \]
  for $\lbb$ large enough.
\end{proposition}

\begin{proof}
  First we find an upper bound on $\abs{R}$.  Since $f$ is renormalizable
  $f^b(f(R)) \subset C$.  By the mean value theorem there exists $\xi \in f(R)$
  such that $\opD f^b(\xi) \abs{f(R)} = \abs{f^b(f(R))}$.  We estimate
  $\abs{f(R)} \geq \e^{-\delta} \abs{Q_1(R)} \geq \e^{-\delta} v
  (\abs{R}/\eps)^\alpha$.  Taken all together we get
  \begin{equation} \label{eq:ubR}
    \abs{R}^\alpha \leq
    \frac{\e^\delta \eps^\alpha \abs{f(R)}}{v} =
    \frac{\e^\delta \eps^\alpha \abs{f^b(f(R))}}{v \opD f^b(\xi)} \leq
    \frac{\e^\delta \eps^\alpha \abs{C}}{v \opD f^b(\xi)},
    \qquad \xi \in f(R).
  \end{equation}

  Next, we find a lower bound on $\abs{L}$.  Since $f$ is renormalizable and
  $\opR f$ is nontrivial we get that $f^a(f(L)) \supset L$.  By the mean
  value theorem there exists $\eta \in f(L)$ such that $\abs{f^a(f(L))} = \opD
  f^a(\eta) \abs{f(L)}$.  Use these two facts to estimate
  \[
    \abs{L} \leq \abs{f^a(f(L))} = \opD f^a(\eta) \abs{f(L)} \leq
    \opD f^a(\eta) \e^\delta u (\abs{L}/c)^\alpha,
  \]
  and hence
  \begin{equation} \label{eq:lbL}
    \abs{L}^{\alpha-1} \geq \frac{c^\alpha}{\e^\delta u \opD f^a(\eta)},
    \qquad \eta \in f(L).
  \end{equation}

  There are now two cases to consider: either $\abs{L} < \abs{R}$ or $\abs{L}
  \geq \abs{R}$.  The former case will turn out not to hold, but we do not know
  that yet.

  \medskip\noindent
  \textit{Case 1:}  In order to reach a contradiction, we assume that $\abs{L}
  < \abs{R}$.  This implies that $\abs{C} < 2\abs{R}$ so equations
  \eqnref{ubR} and~\eqnref{lbL} show that
  \begin{equation} \label{eq:RoverL1}
    \frac{\abs{R}}{\abs{L}} \leq \bigoh\left(
    \eps^\alpha \frac{\opD f^a(\eta)}{\opD f^b(\xi)}
    \right)^{\mrlap{\xop{\alpha-1}}}.
  \end{equation}
  We would like to apply \lemref{derivs}, but we do not know the position of
  $f^b(\xi)$ in relation to~$c$.  However, we claim that the distortion of
  $f^b$ on $f(R)$ is very small which will allow us to use \lemref{derivs}
  anyway, since
  \[
    \opD f^b(\xi) = \frac{\opD f^b(\xi)}{\opD f^b(f_0^{-b}(c))}
    \frac{1}{\opD f_0^{-b}(c)} \geq
    \Big(
    \opD f_0^{-b}(c) \cdot \exp\left\{\distortion f^b|_{f(R)}\right\}
    \Big)^{\mrlap{-1}}.
  \]
  Note that $f^a(\eta) > f_0\inv(c)$ since $f$ is renormalizable, so we can
  directly apply \lemref{derivs} to estimate $\opD f^a(\eta)$.

  We now prove that the distortion of $f^b$ on $f(R)$ is small.  For monotone
  combinatorics we have
  \[
    f(R) \subset (f_0^{-b-1}(c),f_0^{-b+1}(c)),
  \]
  thus
  \[
    \abs{f_0^{-b+1}(c) - f_0^{-b-1}(c)} \geq \abs{f(R)} \geq
    \eps^{-\delta} v (\abs{R}/\eps)^\alpha,
  \]
  and consequently
  \[
    \frac{\abs{R}}{\eps} \leq \left(
    \frac{\e^\delta}{v} \cdot \abs{f_0^{-b+1}(c) - f_0^{-b-1}(c)}
    \right)^{\mrlap\ooa}
    \to 0, \quad \text{as $\lbb \to \infty$.}
  \]
  This and \lemref{crit-vals} shows that the length of the right component of
  $[0,\lcv] \setminus C$ is much larger than~$C$ (since $\abs{R} > \abs{C}/2$
  by assumption).  Since $f^{-b}|_C$ extends monotonously to $[0,\lcv]$ the
  Koebe lemma implies that the distortion of $f^b|_{f(R)}$ tends to zero as
  $\lbb \to \infty$.  (Note that the left component of $[0,\lcv] \setminus C$
  is of order~$1$ so it is automatically large compared to~$C$.)

  Now that we have control over the distortion, apply \lemref{derivs} to get
  \[
    \eps^\alpha \frac{\opD f^a(\eta)}{\opD f^b(\xi)} = \bigoh\left(
    \eps^{-(a + 1 - \alpha - \ooan{b})} \alpha^{-(b-a)} \right).
  \]
  By \eqnref{Omega} $a+1-\alpha-\ooan{b} \geq \sigma - \ooan{b}$ and we may
  assume that $\sigma > \ooan{b}$ (by choosing $\lbb$ sufficiently large) so
  that the exponent of $\eps$ is negative.  Inserting
  $\eps \geq \alpha^{-\lbb/\alpha}$ we get that the
  right-hand side is at most of the order
  $\alpha^{-t}$, where
  \[
    t = -\frac{\lbb}{\alpha} (\alpha-\sigma-\alpha^{-b}) + \lbb
    = \left(\frac{\sigma+\alpha^{-b}}{\alpha}\right) \lbb.
  \]
  The expression in front of $\lbb$ is positive so $t \to \infty$ as $\lbb
  \to \infty$.  Hence \eqnref{RoverL1} shows that $\abs{R}/\abs{L} \to 0$ as
  $\lbb \to \infty$.  This contradicts the assumption that $\abs{R} >
  \abs{L}$, so we conclude that $\abs{R} \leq \abs{L}$.

  \medskip\noindent
  \textit{Case 2:}  From the argument above we know that $\abs{L} \geq
  \abs{R}$.  In particular, $\abs{C} \leq 2\abs{L}$, so equations \eqnref{ubR}
  and~\eqnref{lbL} show that
  \begin{equation} \label{eq:RoverL2}
    \frac{\abs{R}}{\abs{L}} \leq \frac{\eps}{c}
    \left(\frac{2 \e^{2\delta} u}{v}
    \frac{\opD f^a(\eta)}{\opD f^b(\xi)} \right)^{\mrlap\ooa}.
  \end{equation}
  As in Case~1 we would like to apply \lemref{derivs} but first we need to show
  that the distortion of $f^b$ on~$f(R)$ is small.  In order to so we need an
  upper bound on $\abs{L}$.

  Since $f$ is renormalizable $f(L) \subset C$, so another mean value theorem
  estimate gives
  \[
    2\abs{L} \geq \abs{C} \geq \opD f^a(\zeta) \abs{f(L)} \geq
  \opD f^a(\xi) \e^{-\delta} u (\abs{L}/c)^\alpha,
  \]
  for some $\zeta \in f(L)$.  Now apply \lemref{derivs} to get that
  \begin{equation} \label{eq:ubL}
    \abs{L} \leq \bigoh\left(\eps^{a/(\alpha-1)}\right).
  \end{equation}
  By \eqnref{Omega} $a \geq \alpha - 1 + \sigma$ so once again we get that the
  length of the right component of $[0,\lcv] \setminus C$ is large compared
  to~$C$ (use \lemref{crit-vals} to bound $\lcv$).  The Koebe lemma shows that
  the distortion of $f^b|_{f(R)}$ tends to zero as $\lbb \to \infty$.

  We can now apply \lemref{derivs} to \eqnref{RoverL2} to get that
  \begin{equation} 
    \eps(\opR f) = \frac{\abs{R}}{\abs{L}+\abs{R}} < \frac{\abs{R}}{\abs{L}}
    \leq K \left(\eps^{-(a + 1 - \alpha - \ooan{b})} \alpha^{-(b-a)}
    \right)^{\mrlap\ooa}.
  \end{equation}
  As in the above we may assume that the exponent of $\eps$ is negative, so
  inserting $\eps \geq \alpha^{-\lbb/\alpha}$ we get
  \[
    \eps(\opR f) \leq K_1 \left(
    \alpha^{\lbb (a+1-\alpha-\alpha^{-b})/\alpha - \lbb}
    \right)^\ooa
    \leq K_2 \left(
    \alpha^{\lbb (\alpha-\sigma)/\alpha - \lbb}
    \right)^\ooa
    = K_2 \alpha^{-\lbb \sigma/\alpha^2}\!\!.
  \]
  Let $\theta = K_2$ to finish the proof.
\end{proof}

Knowing that $\eps(\opR f)$ is small it is relatively straightforward to use
the Koebe lemma to prove that the distortion of the diffeomorphic parts of
$\opR f$ is small.  Here we really need the condition that the return time of
the left branch satisfies $a > \alpha - 1$ in order to find some Koebe space.
Also note that we assume negative Schwarzian derivative so that we can apply the
strong version of the Koebe lemma (see \lemref{S-koebe}) which gives explicit
bounds on the distortion.

\begin{proposition} \label{prop:dist}
  If $f \in \setLS_\Omega \cap \setK$, then $\distortion\phi(\opR f) \leq
  \delta$ and $\distortion\psi(\opR f) \leq \delta$,
  for $\lbb$ large enough.
\end{proposition}

\begin{proof}
  From \propref{ubeps} we know that $\abs{L} > \abs{R}$ and thus \eqnref{ubL}
  applies, which shows that $\abs{C}$ is at most of the order
  $\eps^{a/(\alpha-1)}$.  Hence \lemref{crit-vals} shows that the right
  component of $(\rcv,\lcv) \setminus C$ has length of order $\eps$ and the
  left component has length of order~$1$.

  Let $\hat U = f^{-a}_1(C)$ and $\hat V = f^{-b}_0(C)$.  The inverses of
  $f^a|_{\hat U}$ and $f^b|_{\hat V}$ extend monotonously (at least) to
  $(\rcv,\lcv)$ so the Koebe lemma (see \corref{koebe}) implies that the
  distortion of these maps is of the order $\eps^t$, where
  \[
    t = -1 + a/(\alpha-1) > \sigma/(\alpha-1) > 0
  \]
  by \eqnref{Omega}.
  
  Since $\phi(\opR f)$ equals $f^a|_{\hat U} \circ \phi$ and $\psi(\opR f)$
  equals $f^b|_{\hat V} \circ \psi$ (up to rescaling) this shows that
  \begin{equation} \label{eq:dist-Rf}
    \distortion \phi(\opR f) \leq K \eps^t \quad\text{and}\quad
    \distortion \psi(\opR f) \leq K \eps^t.
  \end{equation}
  Note that $\eps^t \leq \theta\alpha^{-t \lbb \sigma/\alpha^2} \ll \delta$
  for $\lbb$ large enough, since $\delta = (1/\lbb)^2$ by \eqnref{K}.
\end{proof}

The final step in the invariance proof is showing that $\eps(\opR f)$ is not
too small.  This is the only place where we use the condition on $\rcv(\opR
f)$.  This condition excludes maps whose renormalization has a trivial right
branch.  Such maps are difficult for us to handle because $\eps(\opR f)$ may be
smaller than the lower bound on~$\eps$.

\begin{proposition} \label{prop:lbeps}
  If $f \in \setLS_\Omega \cap \setK$ and if $1-\rcv(\opR f) \geq \lambda$ for
  some $\lambda > 0$ not depending on $\lbb$,
  then $\eps(\opR f) \geq \alpha^{-\lbb/\alpha}$, for $\lbb$ large enough.
\end{proposition}

\begin{proof}
  First we look for a lower bound on $\abs{R}$.  The condition on
  $\rcv(\opR f)$ gives
  \[
    1-\lambda \geq \rcv(\opR f) = 1 - \frac{\abs{f^b(f(R))}}{\abs{C}}
  \]
  and hence $\abs{f^b(f(R))} \geq \lambda \abs{C} \geq \lambda
  \abs{L}$.\footnote{This is the \emph{only} place where we use the condition
  on $\rcv(\opR f)$.}  On the other hand, the mean value theorem shows that
  there exists $\xi \in f(R)$ such that
  \[
    \abs{f^b(f(R))} = \opD f^b(\xi) \abs{f(R)} \leq
    \opD f^b(\xi) \e^\delta (\abs{R}/\eps)^\alpha.
  \]
  Thus
  \begin{equation} \label{eq:lbR}
    \abs{R}^\alpha \geq
    \frac{\lambda \abs{L} \eps^\alpha}{\e^\delta \opD f^b(\xi)},
    \qquad \xi \in f(R).
  \end{equation}

  Next, we look for an upper bound on $\abs{L}$.  The mean value theorem in
  conjunction with $C \supset f^a(f(L))$ and $2\abs{L} > \abs{C}$, shows that
  $2\abs{L} > \abs{C} \geq \opD f^a(\eta) \e^{-\delta} u(\abs{L}/c)^\alpha$, for
  some $\eta \in f(L)$.  Hence
  \begin{equation} \label{eq:ubL2}
    \abs{L}^{\alpha-1} \leq \frac{2 \e^\delta c^\alpha}{u \opD f^a(\eta)},
    \qquad \eta \in f(L).
  \end{equation}

  Equations \eqnref{lbR} and~\eqnref{ubL2} show that
  \begin{equation} \label{eq:lbRL}
    \frac{\abs{R}}{\abs{L}} \geq \frac{\eps}{c}
    \left( \frac{\lambda u}{2\e^{2\delta}}
    \frac{\opD f^a(\eta)}{\opD f^b(\xi)} \right)^{\mrlap\ooa}.
  \end{equation}
  Now apply \lemref{derivs} (\propref{dist} can be used to bound $\opD f^b(\xi)$
  in case $f^b(\xi) > c$) to get that
  \[
    \eps(\opR f) = \frac{\abs{R}}{\abs{L}+\abs{R}} = \frac{\abs{R}}{\abs{L}}
    \cdot \left(1 + \frac{\abs{R}}{\abs{L}}\right)\inv
    \geq k_0 
    \left(\eps^{-(a + 1 - \alpha - \ooan{b})} \alpha^{-(b-a)}
    \right)^{\mrlap\ooa}.
  \]
  By \eqnref{Omega}, $a+1-\alpha-\alpha^{-b} \geq \sigma-\alpha^{-b}$ which we
  may assume to be positive (by choosing $\lbb$ large), so inserting $\eps
  \leq \theta\alpha^{-\lbb \sigma/\alpha^2}$ in the right-hand side we get
  that $\eps(\opR f) \geq k_1
  \alpha^{-t/\alpha}$, where
  \begin{equation} \label{eq:lbexp}
    t = -\frac{\lbb \sigma}{\alpha^2} \big(\sigma - \alpha^{-\lbb}\big) +
    \big(1 + (\sigma/\alpha)^2 - \beta\big)\lbb
    = \Big( 1 + \frac{\sigma\alpha^{-\lbb}}{\alpha^2} - \beta\Big) \lbb.
  \end{equation}
  We may assume that $\beta > \sigma\alpha^{-\lbb}/\alpha^2$ by
  choosing $\lbb$ large enough.  Hence
  \[
    \eps(\opR f) \geq k_1 \alpha^{-\rho \lbb/\alpha},
    \qquad \rho = 1 + \frac{\sigma\alpha^{-\lbb}}{\alpha^2} - \beta,
  \]
  which is larger than $\eps^{-\lbb/\alpha}$ for $\lbb$ large enough since
  $\rho \downarrow 1-\beta < 1$ as $\lbb \to \infty$.
\end{proof}

The above propositions are all we need to prove invariance:

\begin{proof}[Proof of~\thmref{K-inv}]
  Apply Propositions \ref{prop:ubeps}, \ref{prop:dist} and~\ref{prop:lbeps}.
\end{proof}

To prove \thmref{K-inv2} we need to show that twice renormalizable maps
in~$\setK$ automatically satisfy the condition on $\rcv(\opR f)$.  The only
problem is that twice renormalizable maps may have $\eps(\opR f) <
\alpha^{-\lbb/\alpha}$ in general, but even so we can still apply
\lemref{invorbits} to
$\opR f$ to get some bound on $\rcv(\opR f)$.

\begin{proof}[Proof of~\thmref{K-inv2}]
  If we go through the proof of \propref{lbeps} without using the condition on
  $\rcv(\opR f)$ and instead use $f^b(f(R)) \supset R$, then \eqnref{lbR}
  becomes
  \begin{equation*}
    \abs{R}^{\alpha-1} \geq
    \frac{\eps^\alpha}{\e^\delta \opD f^b(\xi)},
    \qquad \xi \in f(R),
  \end{equation*}
  and \eqnref{lbRL} becomes
  \begin{equation*}
    \frac{\abs{R}}{\abs{L}} \geq \frac{\eps}{c}
    \left( \frac{u}{2\e^{2\delta}}
    \frac{\opD f^a(\eta)}{\opD f^b(\xi)} \right)^{\mrlap{\xop{\alpha-1}}}.
  \end{equation*}
  This time we get that $\eps(\opR f) \geq k_1 \alpha^{-t/(\alpha-1)}$, where
  $t$ is the same as in \eqnref{lbexp}.  However, the important thing to
  note is that we still get a lower bound of the type $\tilde \eps = \eps(\opR
  f) \geq k_1 \alpha^{-K \lbb}$.  This and Propositions~\ref{prop:ubeps}
  and~\ref{prop:dist} show that we can apply \lemref{invorbits} to $\tilde f =
  \opR f$.

  By the above argument we can apply \lemref{invorbits} to $\tilde f$ and
  $\tilde \eps = 1 - \tilde c$ to get that
  \[
    \frac{\tilde c - \tilde f_0^{-b}(\tilde c)}{\tilde c} \geq
    \nu \left(k_1\alpha^{-K \lbb}\right)^\ooan{b} \geq k_2 > 0.
  \]
  Let $\lambda = k_2$.  Note that for monotone combinatorics $\rcv \leq
  f^{-b}_0(c)$ and since $\tilde f$ is renormalizable this shows that
  $1 - \rcv(\tilde f) > \lambda$.
\end{proof}


\section{A priori bounds} 
\label{sec:apb}

In this section we begin exploiting the existence of the relatively compact
`invariant' set of \thmref{K-inv}.  An important consequence of this theorem is
the existence of so-called \emph{a priori bounds} (or \emph{real bounds}) for
infinitely renormalizable maps.  We use the a priori bounds to analyze
infinitely renormalizable maps and their attractors.

From now on we will assume that the sets $\Omega$ and~$\setK$ of
\defref{Omega-K} have been fixed; in particular, we assume that $\lbb$ has been
chosen large enough for \thmref{K-inv2} to hold.

\begin{theorem}[A priori bounds] \label{thm:realbounds}
  \index{a priori bounds}
  If $f \in \setLS_{\rtype} \cap \setK$ is infinitely renormalizable with
  $\rtype \in \Omega^\nats$, then $\{\opR^n f\}_{n\geq0}$ is a relatively
  compact family (in $\setL^0$).
\end{theorem}

\begin{proof}
  This is a consequence of \thmref{K-inv2} and \propref{K-relcpt}.
\end{proof}

\begin{theorem} \label{thm:weak-markov}
  If $f \in \setLS_{\rtype} \cap \setK$ is infinitely renormalizable with
  $\rtype \in \Omega^\nats$, then $f$ satisfies the weak Markov property.
\end{theorem}

\begin{proof}
  Since $f$ is infinitely renormalizable there exists a
  sequence $C_0 \supset C_1 \supset \cdots$ of nice intervals whose lengths
  tend to zero (i.e.\ $C_n$
  is the range of the $n$--th first-return map and this interval is nice since
  the boundary consists of periodic points whose orbits do not enter $C_n$).

  Let $T_n$ denote the transfer map to~$C_n$.  We must show that $T_n$ is
  defined almost everywhere and that there exists $\delta>0$ (not depending
  on~$n$) such that $C_n$ contains a $\delta$--scaled neighborhood
  of~$C_{n+1}$, for every~$n \geq 0$.

  By a theorem of Singer\footnote{
   Singer's theorem is stated for unimodal maps but the statement and proof can
   easily be adapted to Lorenz maps.
  }
  $f$ cannot have a periodic attractor since it would attract at least one of
  the critical values.  This does not happen for infinitely renormalizable maps
  since the critical orbits have subsequences which converge on the critical
  point.  Thus \propref{Tdomain} shows that $T_n$ is defined almost everywhere.

  Let $L_n = C_n \cap (0,c)$ and let $R_n = C_n \cap (c,1)$, where $c$ is the
  critical point of~$f$.  Since $f$ is infinitely renormalizable there exists
  $l_n$ and~$r_n$ such that $f^{l_n}(L_n)$ is in the right component of
  $C_{n-1} \setminus C_n$, and such that $f^{r_n}(R_n)$ is contained in the
  left component of $C_{n-1} \setminus C_n$.  We contend that
  \begin{equation} \label{eq:inf-ratio}
    \inf_n\; \abs{f^{l_n}(L_n)}/\abs{C_n} > 0
    \quad\text{and}\quad
    \inf_n\; \abs{f^{r_n}(R_n)}/\abs{C_n} > 0.
  \end{equation}
  Suppose not, and consider the $\setC^0$--closure of~$\{\opR^n f\}$.  The a
  priori bounds show that this set is compact and hence there exists a
  subsequence
  $\{\opR^{n_k} f\}$ which converges to some~$f_*$.  But then $f_*$ is a
  renormalizable map whose cycles of renormalization contain an interval of
  zero diameter.  This is impossible, hence \eqnref{inf-ratio} must hold.
  
  Equation~\eqnref{inf-ratio} shows that $C_{n-1}$ contains a $\delta$--scaled
  neighborhood of~$C_n$ and that $\delta$ does not depend on $n$.
\end{proof}

\begin{theorem} \label{thm:cantorattractor}
  Assume $f \in \setLS_{\rtype} \cap \setK$ is infinitely renormalizable with
  $\rtype \in \Omega^\nats$.  Let $\Lambda$ be the closure of the orbits of the
  critical values.  Then:
  \begin{compactitem}
    \item $\Lambda$ is a Cantor set,
    \item $\Lambda$ has Lebesgue measure zero,
    \item the Hausdorff dimension of $\Lambda$ is strictly inside $(0,1)$,
    \item the complement of the basin of attraction of $\Lambda$ has zero
      Lebesgue measure.
  \end{compactitem}
\end{theorem}

\begin{proof}
  Let $L_n$ and $R_n$ denote the left and right half of the return interval of
  the $n$--th first-return map, let $i_n$ and $j_n$ be the return times for
  $L_n$ and $R_n$, let $\Lambda_0 = \uint$, and let
  \[
    \Lambda_n =
    \bigcup_{i = 0}^{i_n-1} \clos f^i(L_n) \,\cup
    \bigcup_{j = 0}^{j_n-1} \clos f^j(R_n), \quad n = 1,2,\dots
  \]
  Components of $\Lambda_n$ are called \Index{intervals of generation~$n$} and
  components of $\Lambda_{n-1} \setminus \Lambda_n$ are called \Index{gaps of
  generation~$n$} (see \figref{attractor}).

  Let $I$ be an interval of generation $n$, let $J \subset I$ be an interval of
  generation $n+1$, and let $G \subset I$ be a gap of generation $n+1$.  We
  claim that there exists constants $0 < \mu < \lambda < 1$ such that
  \[
    \mu < \abs{J}/\abs{I} < \lambda
    \quad\text{and}\quad
    \mu < \abs{G}/\abs{I} < \lambda,
  \]
  where $\mu$ and~$\lambda$ do not depend on $I$, $J$ and~$G$.  To see this,
  take the $\setL^0$--closure of $\{\opR^n f\}$.  This set is compact
  in~$\setL^0$, so the infimum and supremum of $\abs{J}/\abs{I}$ over all $I$
  and~$J$ as above are bounded away from $0$ and~$1$ (otherwise there would
  exist an infinitely renormalizable map in~$\setL^0$ with $I$ and~$J$ as above
  such that $\abs{J} = 0$ or $\abs{I} = \abs{J}$).  The same argument holds for
  $I$ and~$G$.  Since $\{\opR^n f\}$ is a subset of the closure the claim
  follows.

  Next we claim that $\Lambda = \bigcap \Lambda_n$.  Clearly $\Lambda \subset
  \bigcap \Lambda_n$ (since the critical values are contained in the closure of
  $f(L_n) \cup f(R_n)$ for each~$n$).  From the previous claim $\abs{\Lambda_n}
  < \lambda\abs{\Lambda_{n-1}}$ so the lengths of the intervals of generation
  $n$ tend to~$0$ as $n \to \infty$.  Hence $\Lambda = \bigcap \Lambda_n$.

  It now follows from standard arguments that $\Lambda$ is a Cantor set of zero
  measure with Hausdorff dimension in~$(0,1)$.

  It only remains to prove that almost all points are attracted
  to~$\Lambda$.  Let $T_n$ denote the transfer map to the $n$--th return
  interval~$C_n$.  By \propref{Tdomain} the domain of~$T_n$ has full measure
  for every~$n$ and hence almost every point visits every~$C_n$.  This finishes
  the proof.
\end{proof}

\begin{figure}
  \begin{center}
    \includegraphics{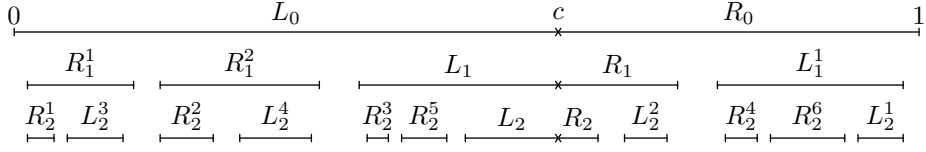}
  \end{center}
  \caption{Illustration of the intervals of generations $0$, $1$ and~$2$ for a
    $(01,100)$--renormalizable map.  Here $L_n^i = f^i(L_n)$ and $R_n^i =
    f^i(R_n)$.  The intersection of all levels $n = 0,1,2,\dots$ is a Cantor
    set, see \thmref{cantorattractor}.}
  \label{fig:attractor}
\end{figure}


\section{Periodic points of the renormalization operator} 
\label{sec:ppts}

In this section we prove the existence of periodic points of the
renormalization operator.  The argument is topological and does not imply
uniqueness even though we believe the periodic points to be unique within each
combinatorial class.\footnote{The conjecture is that the restriction of $\opR$
to the set of infinitely renormalizable maps should contract maps of the same
combinatorial type and this would imply uniqueness.}

The notation used here is the same as in \secref{inv-set}, in particular the
sets $\Omega$ and~$\setK$ are defined in \defref{Omega-K}.  We will implicitly
assume that $\lbb$ has been chosen large enough for \thmref{K-inv} to hold.

\begin{theorem} \label{thm:periodic-points}
  For every periodic combinatorial type $\rtype \in \Omega^\nats$ there exists
  a periodic point of~$\opR$ in $\setL_{\rtype}$.
\end{theorem}

\begin{remark}
  We are not saying anything about the periods of the periodic points.  For
  example, we are \emph{not} asserting that there exists a period-two point of
  type $(\omega,\omega)^\infty$ for some $\omega \in \Omega$ --- all we say is
  that there is a fixed point of type $(\omega)^\infty$.  The point here is
  that $(\omega,\omega)^\infty$ is just another way to write $(\omega)^\infty$
  so these two types are the same.
\end{remark}

To begin with we will consider the restriction $\opR_\omega$ of $\opR$ to some
$\omega \in \Omega$ and show that $\opR_\omega$ has a fixed point.  Fix
$\lambda\in(0,1)$ and let
\[
  \setY = \setLS_\omega \cap \setK, \quad\text{and}\quad
  \setY_\lambda = \{ f \in \setY \mid 1 - \rcv(\opR f) \geq \lambda \}.
\]
Note that if $\lambda$ is made smaller then we may have to compensate by
increasing $\lbb$.  Also note that $\setY_\lambda$ is nonempty for all choices
of $\lambda$.  We will again implicitly assume that $\lbb$ is sufficiently
large for \thmref{K-inv} to hold.

The proof of \thmref{periodic-points} is based on a careful investigation of
the boundary of $\setY$ and the action of $\opR$ on this boundary.  However, we
need to introduce the set~$\setY_\lambda$ because we do not have a good enough
lower bound on $\eps(\opR f)$ for $f\in\setY$, see the discussion after
\thmref{K-inv}.

\begin{definition}
  A branch $B$ of $f^n$ is full\index{branch!full} if $f^n$ maps $B$ onto the
  domain of~$f$; $B$ is trivial\index{branch!trivial} if $f^n$ fixes both
  endpoints of~$B$.
\end{definition}

\begin{proposition} \label{prop:dY}
  \index{renormalization!boundary of}
  The boundary of $\setY$ consists of three parts, namely $f \in
  \bndry\setY$ if and only if at least one of the following conditions hold:
  \begin{compactenum}[(Y1)]
    \item the left or right branch of $\opR f$ is full or trivial,
      \label{dY1}
    \item $\eps(f) = \lb\eps$ or $\eps(f) = \ub\eps$, where $\eps(f) = 1 -
      c(f)$ and \label{dY2}
      \[
        \lb\eps = \min\{\eps(g)\mid g \in \setK\}
        \quad\text{and}\quad \ub\eps = \max\{\eps(g)\mid g \in \setK\}
      \]
    \item $\distortion \phi(f) = \delta$ or $\distortion \psi(f) = \delta$
      ($\delta$ is the same as in \defref{Omega-K}).
      \label{dY3}
  \end{compactenum}
  Also, each condition occurs somewhere on $\bndry\setY$.
\end{proposition}

Before giving the proof we need to introduce some new concepts and recall some
established facts about families of Lorenz maps.

\begin{definition}
  A \Index{slice} (in the
  parameter plane) is any set of the form
  \[
    \slice = \uint^2 \times \{c\} \times \{\phi\} \times \{\psi\},
  \]
  where $c$, $\phi$ and $\psi$ are fixed.  We will permit ourselves to be a
  bit sloppy with notation and write $(u,v) \in \slice$ when it is clear which
  slice we are talking about (or if it is irrelevant).
\end{definition}

A slice $\slice = \uint^2 \times \{c\} \times \{\phi\} \times \{\psi\}$ induces a family of Lorenz maps
\[
  \slice \ni (u,v) \mapsto f_{u,v} = (u,v,c,\phi,\psi) \in \setL.
\]
Any family induced from a slice is \emph{full},\index{full family} by which we
mean that it realizes all possible combinatorics.  See \citep{MdM01} for a
precise definition and a proof of this statement.  For our discussion the only
important fact is the following:

\begin{proposition} \label{prop:full}
  Let $(u,v) \mapsto f_{u,v}$ be a family induced by a slice.  Then this
  family intersects $\setL_{\rtype}$ for every $\rtype$ such that
  $\setL_{\rtype} \neq \emptyset$.  Note that $\rtype$ can be finite or
  infinite.
\end{proposition}

\begin{proof}
  This follows from \citep[Theorem~A]{MdM01}.
\end{proof}

Recall that $C = \clos L \cup R$ is the return interval for a renormalizable
map, and the return times for $L$ and $R$ are $a+1$ and $b+1$, respectively
(see the end of \secref{renorm-op}).

\begin{lemma} \label{lem:R-bndry}
  Assume that $f$ is renormalizable.  Let $(l,c)$ be the branch of $f^{a+1}$
  containing $L$ and let $(c,r)$ be the branch of $f^{b+1}$ containing $R$.
  Then
  \[
    f^{a+1}(l) \leq l
    \quad\text{and}\quad
    f^{b+1}(r) \geq r.
  \]
\end{lemma}

\begin{proof}
  This is a special case of \citep[Lemma~4.1]{MdM01}.
\end{proof}

\begin{proof}[Proof of~\propref{dY}]
  Let us first consider the boundary of $\setL^0_\omega$.  If either branch of
  $\opR f$ is full or trivial, then we can perturb $f$ in $\setC^0$ so that it
  no longer is renormalizable.  Hence (Y\ref{dY1}) holds on
  $\bndry\setL^0_\omega$.  If $f \in \setL^0_\omega$ does not satisfy
  (Y\ref{dY1}) then any sufficiently small $\setC^0$--perturbation of~$f$ will
  still be renormalizable by \lemref{R-bndry}.  Hence the boundary of
  renormalization is exactly characterized by (Y\ref{dY1}).

  Conditions (Y\ref{dY2}) and~(Y\ref{dY3}) are part of the boundary of $\setK$.
  These boundaries intersect $\setLS_\omega$ by \propref{full} and hence these
  conditions are also boundary conditions for $\setY$.
\end{proof}

Fix $1-c_0 = \eps_0 \in (\lb\eps,\ub\eps)$ and let $\slice = \uint^2 \times
\{c_0\} \times \{\id\} \times \{\id\}$.  Let $\rho_t$ be the deformation
retract onto $\slice$ defined by
\[
\rho_t(u,v,c,\phi,\psi) = (u,v,c+t(c_0-c),(1-t)\phi,(1-t)\psi),
  \quad t \in \uint.
\]
In order to make sense of this formula it is important to note that the linear
structure on the diffeomorphisms is that induced from $\setC^0$ via the
nonlinearity operator~$\opN$ (see \remref{linstruct}).  Hence, for example
$t\phi$ is by definition the diffeomorphism $\opN\inv(t\opN\phi)$.  Let
\[
  \opR_t = \rho_t \circ \opR.
\]
The choice of slice is somewhat arbitrary in what follows, except that we will
have to be a little bit careful when chosing $c_0$ as will be pointed out in
the proof of the next lemma.  However, it is important to note that the slice
intersects~$\setY$.

\begin{figure}
  \begin{center}
    \includegraphics{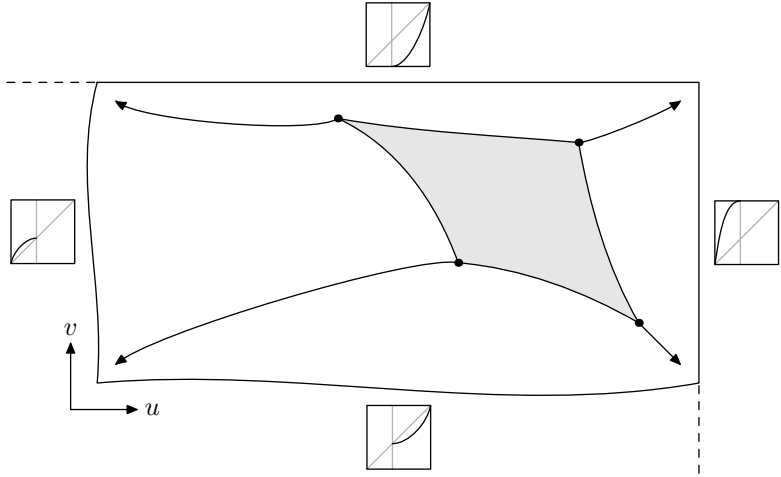}
  \end{center}
  \caption{Illustration of the action of $\rho_1 \circ \opR|_\slice$.  The
           shaded area corresponds to a full island.  The boxes shows what the
           branches of $\rho_1 \circ \opR f$ look like on each boundary piece.}
  \label{fig:slice}
\end{figure}

\begin{lemma}
  There exists $c_0$ such that $\opR_t$ has a fixed point on
  $\bndry\setY_\lambda$ for some $t \in \uint$ if and only if $\opR$ has a
  fixed point on $\bndry\setY_\lambda$.
\end{lemma}

\begin{remark} \label{rem:Ybndry}
  The condition $1 - \rcv(\opR f) \geq \lambda$ roughly states that
  $v(\opR f) \geq 1-\lambda$.  Thus $\setY_\lambda$ has another
  boundary condition given by $\rcv(\opR f) = 1-\lambda$.
  Instead of treating this as separate boundary condition we subsume it
  into (Y\ref{dY1}) by saying that the right branch is trivial also if
  $\rcv(\opR f) = 1-\lambda$.
\end{remark}

\begin{proof}
  The `if' statement is obvious since $\opR = \opR_0$, so assume that $\opR$
  has no fixed point on $\bndry\setY_\lambda$.  Let $f \in \bndry\setY_\lambda$
  and assume that $\opR_t f = f$ for some $t > 0$.  We will show that this is
  impossible.

  To start off choose $\eps_0 \in (\lb{\eps},\ub{\eps})$ and let $c_0 = 1 -
  \eps_0$ as usual (we will be more specific about the choice of $\eps_0$
  later).

  Note that (Y\ref{dY2}) cannot hold for $\opR_t f$ since $\eps_0 \in
  (\lb{\eps},\ub{\eps})$ and hence the same is true for $\eps(\opR_t f)$, since
  $t > 0$ and $\eps(\opR f) \in [\lb{\eps},\ub{\eps}]$ by \thmref{K-inv}.

  Similarly, (Y\ref{dY3}) cannot hold for $\opR_t f$ since the distortion of the
  diffeomorphic parts of $\opR f$ are not greater than~$\delta$ (by
  \thmref{K-inv}) and hence the distortion of the diffeomorphic parts of
  $\opR_t f$ are strictly smaller than~$\delta$ (since $t>0$).\footnote{
    This follows from $\distortion (1-t)\phi < \distortion \phi$ if, $t>0$ and
    $\distortion \phi>0$.
  }


  The only possibility is that $f = \opR_t f$ belongs to the boundary part
  described by condition (Y\ref{dY1}).

  If either branch of $\opR f$ is full then corresponding branch of $\opR_t f$
  is full as well which shows that $f$ cannot be fixed by $\opR_t$, since a
  renormalizable map cannot have a full branch.  Thus one of the branches of
  $\opR f$ must be trivial.

  Assume that the left branch of $\opR f$ is trivial, that is $\lcv(\opR f) =
  c(\opR f)$.  In particular, $\opR f$ is not renormalizable since $\lcv$ for a
  renormalizable map is away from the critical point by \lemref{crit-vals}.
  Because of this lemma we can assure that $\opR_s f$ is not renormalizable for
  all $s \in \uint$ by choosing $\eps_0$ close to $\lb{\eps}$.  In particular,
  $\opR_t f$ is not renormalizable and hence cannot equal~$f$.

  Assume that the right branch of $\opR f$ is trivial (see \remref{Ybndry}).
  Then we may without loss of generality assume that $\lambda > \ub\eps$ and
  hence $\rcv(\opR f) = 1-\lambda$ (just choose $\lambda$ not too small, or
  increase $\lbb$).  In particular $\opR f$ is not renormalizable since that
  requires $\rcv(\opR f)$ to be close to~$0$ by \lemref{crit-vals}.  The same
  holds for $\opR_s f$ for all $s \in \uint$ since $\lambda > \ub{\eps}$.  In
  particular, $f$ cannot be fixed by $\opR_t$ since $f$ is renormalizable.

  We have shown that $f \notin \bndry\setY_\lambda$ which is a contradiction
  and hence we conclude that $\opR_t f \neq f$ for all $t \in \uint$.
\end{proof}

The slice $\slice$ intersects the set $\setL_\omega$ of renormalizable maps of
type~$\omega$ by \propref{full}.  This intersection can in general be a
complicated set, but there will always be at least one connected component $I$
of the interior such that the restricted family $I \ni (u,v) \mapsto f_{u,v}$
is full \citep[see][Theorem~B]{MdM01}.  Such a set $I$ is call a full island.
The action of $\opR$ on a full island is illustrated in \figref{slice}.  Note
that the action of $\opR$ on the boundary of $I$ is given by (Y\ref{dY1}) which
also explains this figure.

\begin{lemma}
  Any extension of $\opR_1|_{\bndry\setY_\lambda}$ to $\setY_\lambda$ has a fixed point.
\end{lemma}

\begin{proof}
  If $\opR_1$ has a fixed point on $\bndry\setY_\lambda$ then there is nothing to
  prove, so assume that this is not the case.

  Let $\slice = \uint^2 \times \{c_0\} \times \{\id\} \times \{\id\}$.  By the
  above discussion there is a full island $I \subset \slice$.  Note that
  $\bndry I \subset \bndry\setY_\lambda$.

  Pick any $R: I \to \slice$ such that $R|_{\bndry I} = \opR_1|_{\bndry I}$.
  Now define the displacement map $\delta: \bndry I \to S^1$ by
  \[
    \delta(x) = \frac{x - R(x)}{\abs{x-R(x)}}.
  \]
  This map is well-defined since $\opR_1$ was assumed not to have any fixed
  points on~$\bndry\setY_\lambda$ and $\bndry I \subset \bndry\setY_\lambda$.  The degree
  of~$\delta$ is nonzero since $I$ is a full island.  This implies that $R$ has
  a fixed point in~$I$, otherwise $\delta$ would extend to all of~$I$ which
  would imply that the degree of~$\delta$ was zero.
  This finishes the proof since $R$ was an arbitrary extension of
  $\opR_1|_{\bndry I}$ and $\bndry I \subset \bndry\setY_\lambda$.
\end{proof}

\begin{proposition}
  $\opR_\omega$ has a fixed point.
\end{proposition}

\begin{proof}
  By the previous two lemmas either $\opR_\omega$ has a fixed point on
  $\bndry\setY_\lambda$ or we can apply \thmref{top-fp}.  In both cases $\opR_\omega$
  has a fixed point.
\end{proof}

\begin{proof}[Proof of \thmref{periodic-points}]
  Pick any sequence $(\omega_0,\dotsc,\omega_{n-1})$ with $\omega_i \in
  \Omega$.  The proof of the previous proposition can be repeated with
  \[
    \opR' = \opR_{\omega_{n-1}} \circ \dotsb \circ \opR_{\omega_0}
  \]
  in place of $\opR$ to see that $\opR'$ has a fixed point~$f_*$.  But then
  $f_*$ is a periodic point of $\opR$ and its combinatorial type is
  $(\omega_0,\dotsc,\omega_{n-1})^\infty$.
\end{proof}


\section{Decompositions} 
\label{sec:decompositions}

In this section we introduce the notion of a decomposition.  We show how to
lift operators from diffeomorphisms to decompositions and also how
decompositions can be composed in order to recover a diffeomorphism.  This
section is an adaptation of techniques introduced in~\citet{Mar98}.

\begin{definition}
  A \Index{decomposition} $\decomp{\phi}: T \to \setD^2(\uint)$ is an ordered
  sequence of diffeo\-morphisms labelled by a totally ordered and at most
  countable set~$T$.  Any such set~$T$ will be called a \Index{time set}.  The
  space $\setD$ is defined in~\appref{nonlin}.

  The space of decompositions $\decomps_T$ over~$T$ is the direct
  product
  \[
    \decomps_T = \prod_T \setD^2(\uint)
  \]
  together with the $\ell^1$--norm
  \[
    \norm{\decomp{\phi}} = \sum_{\tau \in T} \norm{\phi_\tau}.
  \]
  The notation here is $\phi_\tau = \decomp{\phi}(\tau)$.  The distortion of a
  decomposition is defined similarly: \index{distortion!of a decomposition}
  \[
    \distortion{\decomp\phi} = \sum_{\tau \in T} \distortion{\phi_\tau}.
  \]

  The sum of two time sets $T_0 \du T_1$ is the disjoint union
  \[
    T_0 \du T_1 = \{(x,i) \mid x \in T_i, i=0,1\},
  \]
  with order $(x,i) < (y,i)$ if and only~if $x < y$, and $(x,0) < (y,1)$ for
  all $x$, $y$.

  The sum of two decompositions
  \[
    \decomp{\phi}_0 \du \decomp{\phi}_1 \in \decomps_{T_0 \du T_1},
  \]
  where $\decomp{\phi}_i \in \decomps_{T_i}$, is
  defined by $\decomp{\phi}_0 \du \decomp{\phi}_1(x,i) = \decomp{\phi}_i(x)$.
  In other words, $\decomp{\phi}_0 \du \decomp{\phi}_1$ is the diffeomorphisms
  of $\decomp{\phi}_0$ in the order of~$T_0$, followed by the diffeomorphisms
  of $\decomp{\phi}_1$ in the order of~$T_1$.

  Note that $\du$ is noncommutative on time sets as well as on decompositions.
\end{definition}

\begin{remark}
  Our approach to decompositions is somewhat different from that
  of~\citet{Mar98}.  In particular, we require a lot less structure on time sets
  and as such our definition is much more suitable to general combinatorics.
  Intuitively speaking, the structure that \citet{Mar98} puts on time sets is
  recovered from limits of the renormalization operator so we will also get
  this structure when looking at maps in the limit set of renormalization.  We
  simply choose not to make it part of the definition to gain some flexibility.
\end{remark}

\begin{proposition}
  The space of decompositions $\decomps_T$ is a Banach space.
\end{proposition}

\begin{proof}
  The nonlinearity operator takes $\setD^2(\uint)$ bijectively to
  $\setC^0(\uint;\reals)$.  The latter is a Banach
  space so the same holds for $\decomps_T$.
\end{proof}

\begin{definition}
  Let $T$ be a finite time set (i.e.\ of finite cardinality) so that we can
  label $T = \{0,1,\dots,n-1\}$ with the usual order of elements.  The
  \Index{composition operator} $\opO: \decomps_T \to \setD^2$ is defined by
  \[
    \opO\decomp{\phi} = \phi_{n-1} \circ \dotsb \circ \phi_0.
  \]
  The composition operator composes all maps in a decomposition in the order
  of~$T$.  We can also define partial composition operators
  \index{composition operator!partial}
  \[
    \opO_{[j,k]}\decomp{\phi} = \phi_k \circ \dots \circ \phi_j,
    \qquad 0 \leq j \leq k < n.
  \]
  As a notational convenience we will write $\opO_{\leq k}$ instead of
  $\opO_{[0,k]}$ etc.
\end{definition}

Next, we would like to extend the composition operator to countable time sets
but unfortunately this is not possible in general.  Instead of $\setD^2$ we
will work with the space $\setD^3$ with the $\setC^1$--nonlinearity norm:
\[
  \norm{\phi}_1 = \norm{\opN\phi}_{\setC^1} =
  \max_{k=0,1} \{ \abs{\opD^k(\opN\phi)} \},
  \qquad \phi \in \setD^3.
\]
Define $\decomps^3_T = \{ \decomp{\phi}: T \to \setD^3 \mid
\norm{\decomp{\phi}}_1 < \infty \}$, where
\[
  \norm{\decomp{\phi}}_1 = \sum \norm{\phi_\tau}_1.
\]
Note that $\norm{\cdot}$ will still be used to denote the
$\setC^0$--nonlinearity norm.

\begin{proposition} \label{prop:comp-op}
  The composition operator $\opO: \decomps^3_T \to \setD^2$ continuously
  extends to decompositions over countable time sets~$T$.
\end{proposition}

\begin{remark} \label{rem:comp-op}
  It is important to note that there is an inherent loss of smoothness when
  composing a decomposition over a countable time set.  Starting with a bound
  on the $\setC^1$--nonlinearity norm we only conclude a bound on the
  $\setC^0$--nonlinearity norm of the composed map.  This can be generalized;
  starting with a bound on the $\setC^{k+1}$--nonlinearity norm, we can
  conclude a bound on the $\setC^k$--nonlinearity norm for the composed map.

  The reason why we loose one degree of smoothness is because we use the mean
  value theorem for one estimate in the Sandwich Lemma~\ref{lem:sandwich}.  If
  necessary it should be possible to replace this with for example a H\"older
  estimate which would lead to a slightly stronger statement.
\end{remark}

In order to prove this proposition we will need the Sandwich Lemma which in
itself relies on the following properties of the composition operator.

\begin{lemma} \label{lem:finite-decomp-prop}
  Let $\decomp{\phi} \in \decomps_T$ be a decomposition over a finite time
  set~$T$, and let $\phi = O\decomp{\phi}$.  Then
  \[
    \e^{-\norm{\decomp{\phi}}} \leq \abs{\phi'} \leq \e^{\norm{\decomp{\phi}}},
    \quad
    \abs{\phi''} \leq \norm{\decomp{\phi}} \e^{2\norm{\decomp{\phi}}},
    \quad\text{and}\quad
    \norm{\phi} \leq \norm{\decomp{\phi}} \e^{\norm{\decomp{\phi}}}.
  \]
  If furthermore, $\decomp{\phi} \in \decomps^3_T$, then
  \[
    \norm{\phi}_1 \leq (1 + \norm{\decomp{\phi}}) \e^{2\norm{\decomp{\phi}}}
    \norm{\decomp{\phi}}_1.
  \]
\end{lemma}

\begin{remark}
  Note that the lemma is stated for finite time sets, but the way we define the
  composition operator for countable time sets (see the proof
  of~\propref{comp-op}) will mean that the lemma also holds for countable time
  sets.
\end{remark}

\begin{proof}
  The bounds on $\abs{\phi'}$ and $\abs{\phi''}$ follow from an induction
  argument using only \lemref{Nprop}.

  Since $T$ is finite we can label $\decomp{\phi}$ so that $\phi = \phi_{n-1}
  \circ \dotsb \circ \phi_0$.  Let $\psi_i = \opO_{<i}(\decomp{\phi})$ and let
  $\psi_0 = \id$.  Now the bound on $\norm{\phi}$ follows from
  \[
    \opN\phi(x) = \sum_{i=0}^{n-1} \opN\phi_i(\psi_i(x)) \psi'_i(x),
  \]
  which in itself is obtained from an induction argument using the chain rule
  for nonlinearities (see~\lemref{N-chain-rule}).

  Finally, take the derivative of the above equation to get
  \[
    (\opN\phi)'(x) = \sum_{i=0}^{n-1} (\opN\phi_i)'(\psi_i(x)) \psi'_i(x)^2
    + \opN\phi_i(\psi_i(x)) \psi''_i(x).
  \]
  From this the bound on $\norm{\phi}_1$ follows.
\end{proof}

\begin{lemma}[Sandwich Lemma] \label{lem:sandwich}
  \index{Sandwich Lemma}
  Let $\phi = \phi_{n-1} \circ \dotsb \circ \phi_0$ and let $\psi$ be obtained
  by ``sandwiching $\gamma$ inside $\phi$;''  that is,
  \[
    \psi = \phi_{n-1} \circ \dotsb \circ \phi_i \circ \gamma \circ
    \phi_{i-1} \circ \dotsb \circ \phi_0,
  \]
  for some $i \in \{0,\dotsc,n\}$ (with the convention that $\phi_n = \phi_{-1}
  = \id$).

  For every $\lambda$ there exists $K$ such that if $\gamma,\phi_i \in \setD^3$
  and if $\norm{\gamma}_1 + \sum \norm{\phi_i}_1 \leq \lambda$, then
  $\norm{\psi - \phi} \leq K \norm{\gamma}$.
\end{lemma}

\begin{proof}
  Let $\phi_+ = \phi_n \circ \dotsb \circ \phi_i$, and let $\phi_- = \phi_{i-1}
  \circ \dotsb \circ \phi_{-1}$. Two applications of the chain rule for
  nonlinearities gives
  \begin{align*}
    \abs[\big]{\opN\psi(x) - \opN\phi(x)} &=
      \abs[\big]{\opN(\phi_+ \circ \gamma)(\phi_-(x)) - \opN\phi_+(\phi_-(x))}
      \cdot \abs{\phi'_-(x)} \\
    &= \abs[\big]{\opN\phi_+(\gamma(y)) \gamma'(y) - \opN\phi_+(y)
      + \opN\gamma(y)} \cdot \abs{\phi'_-(x)},
  \end{align*}
  where $y = \phi_-(x)$.  By assumption $\opN\phi_+ \in \setC^1$ so by the mean
  value theorem there exists $\eta \in \uint$ such that
  \[
    \opN\phi_+(\gamma(y)) = \opN\phi_+(y) + (\opN\phi_+)'(\eta) \cdot
    \left( \phi(y) - y \right).
  \]
  Hence
  \begin{multline*}
    \abs[\big]{\opN\psi(x) - \opN\phi(x)} \leq
    \abs{\phi'_-(x)} \\
    \cdot \left(
      \abs[\big]{\opN\phi_+(y)} \cdot \abs{\gamma'(y)-1} +
      \abs[\big]{\gamma'(y) \cdot (\opN\phi_+)'(\eta)}
      \cdot \abs{\gamma(y) - y} + \abs{\opN\gamma(y)} \right) \\
    \leq K_1 \cdot \left( K_2 \big(\e^{\norm{\gamma}} - 1\big) +
      K_3 \big(\e^{2\norm{\gamma}} - 1\big) + \norm{\gamma} \right)
      \leq K \norm{\gamma}.
  \end{multline*}
  The constants $K_i$ only depend on~$\lambda$ by \lemref{finite-decomp-prop}.
  We have also used \lemref{Nprop} and \lemref{Nprop2} in the penultimate
  inequality.
\end{proof}

\begin{proof}[Proof of \propref{comp-op}]
  Let $\decomp{\phi} \in \decomps^3_T$ and choose an enumeration $\theta: \nats
  \to T$.  Let $\psi_n$ denote the composition of
  $\{\phi_{\theta(0)},\dots,\phi_{\theta(n-1)}\}$ in the order induced by~$T$.

  We claim that $\{\psi_n\}$ is a Cauchy sequence in $\setD^2$.  Indeed,
  by applying the Sandwich Lemma with $\lambda = \norm{\decomp\phi}_1$ we get a
  constant $K$ only depending on $\lambda$ such that:
  \[
    \norm{\psi_n - \psi_m} \leq
    \sum_{i=m}^{m+n-1} \norm{\psi_{i+1} - \psi_{i}} \leq
    K \sum_{i=m}^{m+n-1} \norm{\phi_{\theta(i)}} \to 0,
    \qquad\text{as $m,n\to\infty$.}
  \]
  Hence $\phi = \lim \psi_n$ exists and $\phi \in \setD^2$.  This also shows
  that $\phi$ is independent of the enumeration~$\theta$ and hence we can
  define $\opO\decomp{\phi} = \phi$.
\end{proof}

We can now use the composition operator to lift operators from $\setD$
to~$\decomps_T$, starting with the zoom operators of \defref{zoom}.

\begin{definition} \label{def:zoom-decomp}
  \index{zoom operator!on decompositions}
  Let $I \subset \uint$ be an interval, let $\decomp{\phi} \in \decomps_T^3$
  and let $I_\tau$ be the image of $I$ under the diffeomorphism
  $\opO_{<\tau}(\decomp{\phi})$.  Define $\opZ(\decomp{\phi}; I) =
  \decomp{\psi}$, where $\psi_\tau = \opZ(\phi_\tau; I_\tau)$, for every $\tau
  \in T$.
\end{definition}

\begin{remark} \label{rem:zoom-equiv}
  An equivalent way of defining the zoom operators on $\decomps_T^3$ is to let
  $I_\tau = \psi_\tau\inv(J)$, where $\psi_\tau = \opO_{\geq
  \tau}(\decomp\phi)$, $J = \phi(I)$, and $\phi = \opO\decomp\phi(I)$.  This is
  equivalent since $\opO\decomp\phi = \opO_{\geq \tau}(\decomp\phi) \circ
  \opO_{<\tau}(\decomp\phi)$.

  The original definition takes the view of zooming in on an interval in the
  domain of the decomposition, whereas the latter takes the view of zooming in
  on an interval in the range of the decomposition.  We will make use of both
  of these points of view.
\end{remark}

Zoom operators on diffeomorphisms are contractions for a fixed interval~$I$ by
\lemref{zoom}.  A similar statement holds for decompositions:

\begin{lemma} \label{lem:dzoom-contract}
  Let $I \subset \uint$ be an interval.
  If $\decomp\phi \in \decomps_T^3$ then
  \[
    \norm{\opZ(\decomp\phi;I)} \leq \e^{\norm{\decomp\phi}} \cdot
    \min\{\abs{I},\abs{\phi(I)}\} \cdot \norm{\decomp\phi},
  \]
  where $\phi = \opO\decomp\phi$.
\end{lemma}

\begin{remark}
  Since we are only dealing with decompositions with very small norm this lemma
  is enough for our purposes.  However, in more general situations
  the constant in front of $\norm{\decomp\phi}$ may not be small enough.
  A way around this is to consider
  decompositions which compose to diffeomorphisms with negative Schwarzian
  derivative.  Then all the intervals $I_\tau$ will have hyperbolic lengths
  bounded by that of~$J$ (notation is as in \remref{zoom-equiv}).  This can
  then be used to show that zoom operators contract and the contraction can be
  bounded in terms of the hyperbolic length of~$J$.
\end{remark}

\begin{proof}
  Using the notation of \defref{zoom-decomp} we have
  \[
    \norm{\opZ(\decomp\phi;I)}
    = \sum_{\tau \in T} \norm{\opZ(\phi_\tau;I_\tau)}
    \leq \sum_{\tau \in T} \abs{I_\tau} \cdot \norm{\phi_\tau}
    \leq \sup_{\tau \in T}\,\abs{I_\tau} \cdot \norm{\decomp\phi}.
  \]
  For every $\tau$ there exists $\xi_\tau \in I$ such that $\abs{I_\tau} =
  (\opO_{<\tau}(\decomp\phi))'(\xi_\tau) \cdot \abs{I}$ which together with
  \lemref{finite-decomp-prop} implies that $\abs{I_\tau} \leq
  \e^{\norm{\decomp\phi}} \cdot \abs{I}$.  Similarly, there exists $\eta_\tau
  \in \phi(I)$ such that $\abs{\phi(I)} =
  (\opO_{\geq\tau}(\decomp\phi))'(\eta_\tau) \cdot \abs{I_\tau}$ so by
  \lemref{finite-decomp-prop} $\abs{I_\tau} \leq \e^{\norm{\decomp\phi}} \cdot
  \abs{\phi(I)}$ as well.
\end{proof}

This contraction property of the zoom operators leads us to introduce the
subspace of pure decompositions (the intuition is that renormalization
contracts towards the pure subspace, see \propref{conv-to-Q}).

\begin{definition} \label{def:pdecomps}
  The subspace of pure decompositions \index{decomposition!pure} $\pdecomps_T
  \subset \decomps_T$ consists of all decompositions $\decomp{\phi}$ such that
  $\phi_\tau$ is a pure map for every~$\tau \in T$.
  
  The subspace of pure maps
  \index{pure map} $\pures \subset \setD^\infty$ consists of restrictions of
  $x^\alpha$ away from the critical point, that is
  \[
    \pures = \big\{ \opZ(x \abs{x}^{\alpha-1};I) \mid
    \intr I \not\ni 0 \big\}.
  \]
  A property of pure maps is that they can be parametrized by one real
  variable.  We choose to parametrize the pure maps by their distortion with a
  sign and call this parameter $s$.  The sign of~$s$ is positive for $I$ to the
  right of~$0$ and negative for $I$ to the left of~$0$.  With this convention
  the graphs of pure maps will look like \figref{pure}.
\end{definition}

\begin{remark} \label{rem:pures}
  Let $\mu_s \in \pures$.  A calculation shows that
  \[
    \distortion \mu_s = \abs{\log\mu_s'(1)/\mu_s'(0)}
  \]
  and from this it is possible to deduce an expression for $\mu_s$:
  \begin{equation} \label{eq:pure}
    \mu_s(x) = \frac{\left(1 +
    \left(\exp\{\frac{s}{\alpha-1}\} - 1\right)x\right)^\alpha -
    1}{\exp\{\frac{\alpha s}{\alpha-1}\} - 1},
    \qquad x \in \uint,\; s \neq 0,
  \end{equation}
  and $\mu_0 = \id$.
  We emphasize that the parametrization is chosen so that $\abs{s}$ equals the
  distortion of~$\mu_s$.  For this reason we call $s$ the \emph{signed}
  distortion of~$\mu_s$. \index{distortion!signed} \figref{pure} shows the
  graphs of $\mu_s$ for different values of~$s$.  Equation~\eqnref{pure} may
  at first seem to indicate that there is some sort of singular behavior at
  $s=0$ but this is not the case; the family $s \mapsto \mu_s$ is smooth.
\end{remark}

\begin{figure}
  \begin{center}
    \includegraphics{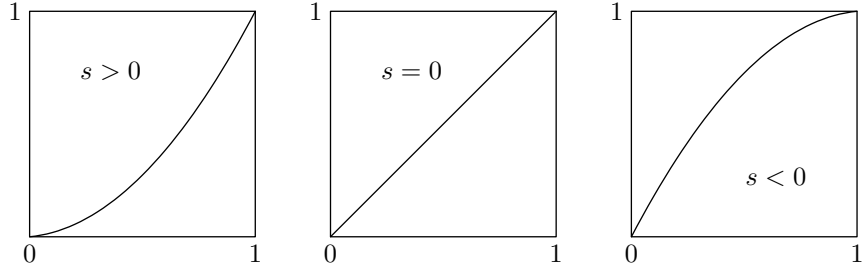}
  \end{center}
  \caption{The graphs of a pure map $\mu_s$ for different values of the signed
  distortion~$s$.}
  \label{fig:pure}
\end{figure}

%

The next two lemmas are needed in preparation for \propref{conv-to-Q}.

\begin{lemma} \label{lem:zoom-contract}
  Let $\phi \in \setD^2$ and let $I \subset \uint$ be an interval.
  Then
  \[
    \distance(\opZ(\phi;I),\pures)
    \leq \abs{I} \cdot \distance(\phi,\pures),
  \]
  where the distance $\distance(\cdot,\cdot)$ is induced by the
  $\setC^0$--nonlinearity norm.
\end{lemma}

\begin{proof}
  A calculation shows that
  \[
    \opN \mu_s(x) = \frac{r_s(\alpha-1)}{1+r_sx},
    \qquad r_s = \exp\left\{\frac{s}{\alpha-1}\right\} - 1.
  \]
  Let $I = [a,b]$ and let $\zeta_I(x) = a + \abs{I} \cdot x$.  Then
  \begin{align*}
    \distance(\opZ(\phi;I),\pures)
      &= \inf_{s \in \reals} \max_{x \in \uint}
      \abs[\big]{\opN(\opZ(\phi;I))(x) - \opN\mu_s(x)} \\
    &= \inf_{r > -1} \max_{x \in \uint} \abs[\bigg]{\abs{I} \cdot
      \opN\phi(\zeta_I(x)) - \frac{r(\alpha-1)}{1+rx}} \\
    &= \inf_{r > -1} \max_{x \in \uint} \abs[\bigg]{ \abs{I} \cdot
      \opN\phi(\zeta_I(x))
      - \frac{r(\alpha-1)}{1 + r(\zeta_I(x) - a)/\abs{I}}} \\
    &= \abs{I} \cdot \inf_{\rho \notin [-\frac{1}{b},-\frac{1}{a}]} \max_{x \in
      I} \abs[\bigg]{ \opN\phi(x) - \frac{\rho(\alpha-1)}{1 + \rho x}},
  \end{align*}
  where $\rho = r/(b-(1+r)a)$.  Note that $1+\rho x$ has a zero in~$\uint$ if
  $\rho \leq -1$, so the infimum is assumed for $\rho > -1$.  Thus
  \[
    \distance(\opZ(\phi;I),\pures)
    = \abs{I} \cdot \inf_{\rho > -1} \max_{x \in
      I} \abs[\bigg]{ \opN\phi(x) - \frac{\rho(\alpha-1)}{1 + \rho x}}.
  \]
  Taking the max over $x \in \uint$ finishes the proof.
\end{proof}

\begin{lemma} \label{lem:zoom-contract-decomp}
  Let $\decomp\phi \in \decomps_T^3$ and let $I \subset \uint$ be an interval.
  Then
  \[
    \distance\big(\opZ(\decomp{\phi}; I), \pdecomps_T\big)
    \leq \e^{\norm{\decomp\phi}} \cdot \min\{\abs{I},\abs{\phi(I)}\} \cdot
    \distance(\decomp{\phi}, \pdecomps_T),
  \]
  where $\phi = \opO\decomp\phi$.
\end{lemma}

\begin{proof}
  Use \lemref{zoom-contract} and a similar argument to that employed in the
  proof of~\lemref{dzoom-contract}.
\end{proof}

The pure decompositions have some very nice properties which we will make use
of repeatedly.

\begin{proposition} \label{prop:pure-S-neg}
  If $\decomp\phi \in \pdecomps_T$ and $\norm{\decomp\phi} < \infty$, then
  $\phi = \opO\decomp\phi$ is in~$\setD^\infty$ and $\phi$ has nonpositive
  Schwarzian derivative.
\end{proposition}

\begin{remark}
  Note that $\norm{\decomp\phi} < \infty$ is equivalent to
  $\distortion\decomp\phi < \infty$, since
  \[
    \distortion\mu = \int_0^1 \abs[\big]{\opN\mu(x)} dx,
  \]
  for pure maps $\mu$.  Hence the norm bound can be replaced by a distortion
  bound and the above proposition still holds.
\end{remark}

\begin{proof}
  Let $\eta$ be the nonlinearity of a pure map. 
  A computation gives
  \[
    \opD^k\eta(x) = \frac{(-1)^k k!}{(\alpha-1)^k} \cdot \eta(x)^{k+1}.
  \]
  Hence, if $\eta$ is bounded then so are all of its derivatives (of course,
  the bound depends on~$k$).  Thus \propref{comp-op} shows that $\phi =
  \opO\decomp\phi$ is well-defined and $\phi \in \setD^k$, for all $k\geq2$
  (use \remref{comp-op}).

  Finally, every pure map has negative Schwarzian derivative so $\phi$ must
  have nonpositive Schwarzian deriviative, since negative Schwarzian is
  preserved under composition by \lemref{S-chain-rule}.
\end{proof}

\begin{notation}
  We put a bar over objects associated with decompositions to distinguish them
  from diffeomorphisms.  Hence $\decomp\phi$ denotes a decomposition, whereas
  $\phi$ denotes a diffeomorphism.  Similarly, $\decomps$ denotes a set of
  decompositions, whereas $\setD$ is a set of diffeomorphisms.

  Given a decomposition $\decomp\phi: T \to \setD$, we use the notation
  $\phi_\tau$ to mean $\decomp\phi(\tau)$ and we call this the diffeomorphism
  at time~$\tau$.  Moreover, when talking about $\decomp\phi$ we consistently
  write $\phi$ to denote the composed map $\opO\decomp\phi$.

  We will frequently consider the disjoint union of all decompositions instead
  of decompositions over some fixed time set~$T$ and for this reason we
  introduce the notation
  \[
    \decomps = \bigsqcup_T \decomps_T
    \quad\text{and}\quad
    \pdecomps = \bigsqcup_T \pdecomps_T.
  \]
\end{notation}


\section{Renormalization of decomposed maps} 
\label{sec:renorm-decomp}

In this section we lift the renormalization operator to the space of decomposed
Lorenz maps (i.e.\ Lorenz maps whose diffeomorphic parts are replaced with
decompositions).  We prove that renormalization contracts towards the subspace
of pure decomposed maps.  This will be used in later sections to compute the
derivative of~$\opR$ on its limit set.

\begin{definition}
  Let $T = (T_0,T_1)$ be a pair of time sets, and let $\decomps_T$ denote the
  product $\decomps_{T_0} \times \decomps_{T_1}$.  The space of decomposed
  Lorenz maps \index{Lorenz map!decomposed} $\lorenzd_T$ over~$T$ is the set
  $\uint^2 \times (0,1) \times \decomps_T$ together with structure induced from
  the Banach space $\reals^3 \times \decomps_T$ with the max norm of the
  products.
\end{definition}

\begin{definition}
  The composition operator \index{composition operator!on decomposed maps}
  induces a map $\lorenzd^3_T \to \setL^2$ which (by slight abuse of notation)
  we will also denote $\opO$.  Explicitly, if $\decomp{f} =
  (u,v,c,\decomp{\phi},\decomp{\psi}) \in \lorenzd_T$, then $f =
  \opO\decomp{f}$ is defined by $f =
  (u,v,c,\opO\decomp{\phi},\opO\decomp{\psi})$.
\end{definition}

We will now define the renormalization operator on the space of decomposed
Lorenz maps.  Formally, the definition is identical to the definition of the
renormalization operator on Lorenz maps.  To illustrate this, let $f =
\opO\decomp{f}$ be renormalizable.  Then, by \lemref{tuple-renorm}, $\opR f =
(u',v',c',\phi',\psi')$, where
\begin{equation} \label{eq:R-uvc}
  u'= \abs{Q(L)}/\abs{U},\qquad v' = \abs{Q(R)}/\abs{V},\qquad
  c' = \abs{L}/\abs{C},
\end{equation}
$\phi' = \opZ(f^a \circ \phi; U)$ and $\psi' = \opZ(f^b \circ \psi; V)$.  Zoom
operators satisfy
\[
  \opZ(g \circ h; I) = \opZ(g; h(I)) \circ \opZ(h; I),
\]
so we can write
\begin{align*}
  \phi' &= \opZ(\psi; Q(U_a)) \circ \opZ(Q; U_a) \circ \cdots \circ
    \opZ(\psi; Q(U_1)) \circ \opZ(Q; U_1) \circ \opZ(\phi; U), \\
  \psi' &= \opZ(\phi; Q(V_b)) \circ \opZ(Q; V_b) \circ \cdots \circ
    \opZ(\phi; Q(V_1)) \circ \opZ(Q; V_1) \circ \opZ(\psi; V).
\end{align*}

\begin{definition}
  \index{renormalization operator!on decomposed maps}
  Define $\opR\decomp{f} = (u',v',c',\decomp{\phi}',\decomp{\psi}')$, where
  $u'$, $v'$, $c'$ are given by~\eqnref{R-uvc} and
  \begin{align*}
    \decomp{\phi}' &= \opZ(\decomp{\phi}; U) \du \opZ(Q; U_1) \du
      \opZ(\decomp{\psi}; Q(U_1)) \du \cdots \du \opZ(Q; U_a) \du
      \opZ(\decomp{\psi}; Q(U_a)), \\
    \decomp{\psi}' &= \opZ(\decomp{\psi}; V) \du \opZ(Q; V_1) \du
      \opZ(\decomp{\phi}; Q(V_1)) \du \cdots \du \opZ(Q; V_b) \du
      \opZ(\decomp{\phi}; Q(V_b)),
  \end{align*}
  where $\opZ(Q;\cdot)$ is now interpreted as a decomposition over a singleton
  time set.  See \figref{R-decomp} for an illustration of the action of~$\opR$.
\end{definition}

\begin{figure}
  \begin{center}
    \includegraphics{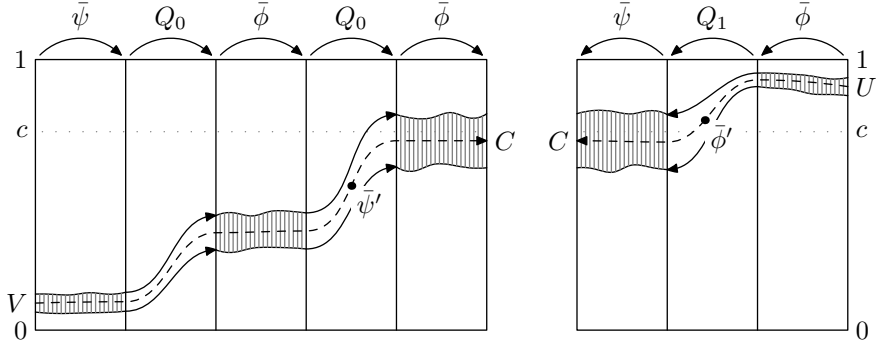}
  \end{center}
  \caption{Illustration of the renormalization operator acting on decomposed
  Lorenz maps.  First the decompositions are `glued' to each other with $Q$
  according to the type of renormalization, here the type is $(01,100)$.  Then
  the interval $C$ is pulled back, creating
  the shaded areas in the picture.  The maps following the dashed arrows from
  $U$ to~$C$ and from $V$ to~$C$ represent the new decompositions before
  rescaling.} \label{fig:R-decomp}
\end{figure}

\begin{definition}
  The domain of~$\opR$ on decomposed Lorenz maps is contained in the disjoint
  union $\lorenzd = \bigsqcup_T \lorenzd_T$ over all time sets~$T$.  Just as
  before we let $\lorenzd_\omega$ denote all $\omega$--renormalizable
  maps in $\lorenzd$; $\lorenzd_{\rtype}$ denotes all maps in $\lorenzd$ such
  that $\opR^i\decomp{f} \in \lorenzd_{\omega_i}$, where $\rtype =
  (\omega_0,\omega_1,\dotsc)$; and $\lorenzd_\Omega = \bigcup_{\omega \in
  \Omega} \lorenzd_\omega$.
\end{definition}

\begin{remark}
  Note that $\opR$ takes the renormalizable maps of $\lorenzd_T$ into
  $\lorenzd_{T'}$, where $T' \neq T$ in general.  This is the reason why we
  have to work with the disjoint union $\bigsqcup_T \lorenzd_T$.
\end{remark}

\begin{lemma}
  The composition operator is a semi-conjugacy.  That is, the following square
  commutes
  \[
    \begin{CD}
      \bigcup \lorenzd^3_\omega @>\opR>> \lorenzd^3 \\
      @VOVV             @VVOV \\
      \bigcup \setL^2_\omega @>\opR>> \setL^2
    \end{CD}
  \]
  and $\opO$ is surjective.
\end{lemma}

\begin{remark}
  This lemma shows that we can use the composition operator to transfer results
  about decomposed Lorenz maps to Lorenz maps.
\end{remark}

\begin{proof}
  The square commutes by definition so let us focus on the surjectivity.  Fix
  $\tau \in T$ and define a map $\Gamma_\tau: \setD \to \decomps_T$ by sending
  $\phi \in \setD$ to the decomposition $\decomp{\phi}: T \to \setD$ defined by
  \[
    \decomp{\phi}(t) = \begin{cases}
      \phi, &\text{if $t = \tau$,} \\
      \id, &\text{otherwise.}

    \end{cases}
  \]
  Then $\opO \circ \Gamma_\tau = \id$ which proves that $\opO$ is surjective
  on~$\decomps_T$ and hence it is also surjective on~$\lorenzd_T$.
\end{proof}

The main result for the renormalization operator on Lorenz maps was the
existence of the invariant set~$\setK$ for types in the set~$\Omega$, see
\secref{inv-set}.  It should come as no surprise that $\setK$ and~$\Omega$ will
be central to our discussion on decomposed maps as well.  The first result in
this direction is the following.

\begin{proposition} \label{prop:conv-to-Q}
  If $\decomp{f} \in \lorenzd^3_{\rtype}$ is infinitely renormalizable with
  $\rtype \in \Omega^\nats$, if $\norm{\decomp\phi} \leq K$ and
  $\norm{\decomp\psi} \leq K$, and if $\opO\decomp{f} \in \setK \cap \setLS$,
  then the decompositions of~$\opR^n \decomp{f}$ are uniformly contracted
  towards the subset of pure decompositions.
\end{proposition}

\begin{proof}
  From the definition of the renormalization operator (and using the
  fact that $\distance(\opZ(Q;I), \pures) = 0$) we get
  \[
    \distance(\decomp{\phi}', \pdecomps) =
    \sum_{i=1}^a \distance(\opZ(\decomp{\psi}; Q(U_i)), \pdecomps)
    + \distance(\opZ(\decomp{\phi}; U), \pdecomps).
  \]
  Now apply \lemref{zoom-contract-decomp} to get
  \[
    \distance(\decomp{\phi}', \pdecomps) \leq
    \e^{\norm{\decomp\psi}}
    \sum_{i=2}^{a+1} \abs{U_i} \distance(\decomp{\psi},\pdecomps) +
    \e^{\norm{\decomp\phi}} \abs{U_1} \distance(\decomp{\phi},\pdecomps).
  \]

  From \secref{inv-set} we get that $\sum \abs{U_i}$ and $\sum \abs{V_i}$ may
  be chosen arbitrarily small (by choosing the return times sufficiently
  large).  Now make these sums small compared with
  $\max\{\e^{\norm{\decomp\phi}},\e^{\norm{\decomp\psi}}\}$ to see that there
  exists $\mu < 1$ (only depending on~$K$) such that
  \[
    \distance(\decomp{\phi}', \pdecomps) + \distance(\decomp{\psi}', \pdecomps)
    \leq \mu \left[ \distance(\decomp{\phi}, \pdecomps) +
    \distance(\decomp{\psi}, \pdecomps) \right]. \qedhere
  \]
\end{proof}


Our main goal is to understand the limit set of the renormalization operator
and the above proposition will be central to this discussion.

\begin{definition} \label{def:attr}
  The set of forward limits of $\opR$ restricted to types in $\Omega$ is
  defined by
  \[
    \attr_\Omega = \bigcap_{n \geq 1} \opR^n \big(
    \bigcup_{\rtype \in \Omega^n} \rdomd_{\rtype} \big).
  \]
\end{definition}

\begin{remark}
  In other words, $\attr_\Omega$ consists of all maps $\decomp{f}$ which have a
  complete past:
  \[
    \decomp{f} = \opR_{\omega_{-1}} \decomp{f}_{-1}, \quad \decomp{f}_{-1} =
    \opR_{\omega_{-2}} \decomp{f}_{-2}, \quad \dots, \qquad \omega_i \in \Omega.
  \]
  This also describes how we can associate each $\decomp{f} \in \attr_\Omega$
  with a left infinite sequence $(\dots,\omega_{-2},\omega_{-1})$.
\end{remark}

\begin{proposition} \label{prop:attr-pure}
  $\attr_\Omega$ is contained in the subset of pure decomposed Lorenz maps.
\end{proposition}

\begin{proof}
  This is a direct consequence of \propref{conv-to-Q}.
\end{proof}

Since $\attr_\Omega$ is contained in the set of pure decomposed maps we will
restrict our attention to this subset from now on.  This is extremely
convenient since pure decompositions satisfy some very strong properties, see
\propref{pure-S-neg}, and it will allow us to compute the derivative at all
points in~$\attr_\Omega$ in \secref{the-derivative}.

Next we would like to lift the invariant set~$\setK$ to the decomposed maps,
but simply taking the preimage $\opO\inv(\setK)$ will yield a set which is too
large\footnote{Any preimage under $\opO$ contains decompositions whose norm is
arbitrarily large.  As an example of how things can go wrong, fix $K > 0$ and
consider $\decomp\phi: \nats \to \setD$ defined by $\phi_{n+1} = \phi\inv_n$
and $\norm{\phi_n} = K$ for every $n$.  Then $\phi_{2n-1} \circ \dotsb \circ
\phi_0 = \id$ for every~$n$, but $\sum\norm{\phi_n} = \infty$.} so we will have
to be a bit careful.

\begin{definition} \label{def:dinvset}
  Let $\delta$, $\setK$ and $\Omega$ be the same as in \defref{Omega-K} and
  let
  \[
    \lb\eps = \min\{\eps(g)\mid g \in \setK\},\quad
    \ub\eps = \max\{\eps(g)\mid g \in \setK\}.
  \]
  Define
  \[
    \dinvset = \{ (u,v,c,\decomp\phi,\decomp\psi) \mid
    \lb{\eps} \leq 1-c \leq \ub{\eps},\;
    \distortion{\decomp\phi} \leq \delta,\;
    \distortion{\decomp\psi} \leq \delta,\;
    \decomp\phi,\decomp\psi \in \pdecomps \},
  \]
  Note that $\dinvset$ is defined analogously to~$\setK$ but with the additional
  assumption that the decompositions are pure.
\end{definition}

\begin{proposition} \label{prop:dinvset}
  If $\decomp f \in \lorenzd_\Omega$ and $1 - \rcv(\opR f) \geq
  \lambda > 0$ for some constant $\lambda$ (not depending on $\lbb$), then
  \[
    f \in \dinvset \implies \opR f \in \dinvset,
  \]
  for $\lbb$ large enough.
\end{proposition}

\begin{proof}
  Let $f = \opO\decomp f = (u,v,c,\phi,\psi)$.  Note first of all that
  $\distortion \decomp\phi \leq \delta$ implies that $\distortion \phi \leq
  \delta$, since $\distortion$ satisfies the subadditivity property
  \[
    \distortion \gamma_2 \circ \gamma_1 \leq
    \distortion \gamma_1 + \distortion \gamma_2.
  \]
  Hence, $f$ automatically satisfies the conditions of \thmref{K-inv}, so all
  we need to prove is that $\distortion \decomp\phi' \leq \delta$ and
  $\distortion \decomp\psi' \leq \delta$.  This is the reason why we define
  $\dinvset$ by a distortion bound instead of a norm bound.  Note that $f$ has
  nonpositive Schwarzian since the decompositions are pure,
  see~\propref{pure-S-neg}.

  We will first show that the norm is invariant, then we transfer this
  invariance to the distortion.  The reason why we consider the norm first is
  because it satisfies the contraction property in \lemref{dzoom-contract}
  which makes it easier to work with.

  From the definition of~$\opR$ and \lemref{dzoom-contract} we get
  \begin{align*}
    \norm{\decomp\phi'} &=
      \norm{\opZ(\decomp\phi;U)} + \sum_{i=1}^a
      \norm{\opZ(\decomp\psi;Q(U_i))} + \norm{\opZ(Q;U_i)} \\
    &\leq \e^{\norm{\decomp\phi}} \norm{\decomp\phi} \cdot \abs{U_1}
      + \e^{\norm{\decomp\psi}} \norm{\decomp\psi} \sum_{i=2}^{a+1} \abs{U_i}
      + \sum_{i=1}^a \norm{\opZ(Q;U_i)}.
  \end{align*}
  The norm of a pure map is determined by how far away its domain is from the
  critical point.  More precisely, we have that
  \[
    \sum_{i=1}^a \norm{\opZ(Q;U_i)} =
    (\alpha-1) \sum_{i=1}^a \frac{\abs{U_i}}{\distance(c,U_i)}.
  \]
  Each term in this sum is bounded by the cross-ratio of $U_i$ inside $[c,1]$.
  Since maps with positive Schwarzian contract cross-ratio, since $\opS f < 0$,
  and since $U_i$ is a pull-back of~$C$ under an iterate of~$f$, this
  cross-ratio is bounded by the cross-ratio $\chi$ of~$C$ inside $[\rcv,1]$.
  Thus, the above sum is bounded by $a (\alpha-1) \chi$.  From the proof of
  \thmref{K-inv} we know that $\chi$ is of the order $\eps^t$ for some $t>0$.
  Since $a < \lbb$ and $\lbb \eps^t \to 0$ we see that the above sum has a
  uniform bound which tends to zero as $\lbb \to \infty$.

  A similar argument for $\decomp\psi'$ gives
  \begin{align*}
    \norm{\decomp\phi'} + \norm{\decomp\psi'} &\leq
      \left(\norm{\decomp\phi} + \norm{\decomp\psi}\right)
      \exp\left\{\norm{\decomp\phi} + \norm{\decomp\psi}\right\}
      \left(\sum \abs{U_i} + \sum \abs{V_i}\right) + m \\
    &= k \left(\norm{\decomp\phi} + \norm{\decomp\psi}\right) + m,
  \end{align*}
  where $m = \sum \norm{\opZ(Q;U_i)} + \sum \norm{\opZ(Q;V_i)}$.
  Hence
  \[
    \norm{\decomp\phi} + \norm{\decomp\psi} \leq \delta
    \implies
    \norm{\decomp\phi'} + \norm{\decomp\psi'} \leq \delta,
    \quad\text{if $\delta \geq m/(1-k)$.}
  \]
  By \defref{Omega-K}, $\delta = (1/\lbb)^2$ and $\eps$ is of the order
  $\alpha^{-\lbb K}\!$, and by the above $m$ is of the order $\lbb\eps^t$.
  Hence $\delta \geq m/(1-k)$ for $\lbb$ large enough.

  The final observation which we use to finish the proof is that if $\gamma \in
  \pures$ then
  \[
    \norm{\gamma} = (\alpha-1) \cdot \left(
    \exp\left\{\frac{\distortion \gamma}{\alpha-1}\right\} - 1 \right).
  \]
  That is $\norm{\gamma} \approx \distortion \gamma$ for pure maps $\gamma$
  with small distortion.  This allows us to slightly modify the above
  invariance argument for the norm so that it holds for the distortion as well.
\end{proof}


\section{The derivative} 
\label{sec:the-derivative}

The tangent space of $\opR$ on the pure decomposed Lorenz maps can be written
$X \times Y$, where $X = \reals^2$ and $Y = \reals \times \ell^1 \times
\ell^1$.  The coordinates on~$X$ correspond to the $(u,v)$ coordinates
on~$\lorenzd_T$.  Let $(x,y) \in X \times Y$ denote the coordinates on the
tangent space and recall that we are using the max norm on the products.  The
derivative of $\opR$ at~$\decomp f$ is denoted
\begin{equation} \label{eq:M}
  \opD\opR_{\decomp f} = M = \begin{pmatrix}
    M_1 & M_2 \\
    M_3 & M_4
  \end{pmatrix},
\end{equation}
where $M_1: \reals^2 \to \reals^2$, $M_2: Y \to \reals^2$, $M_3: \reals^2
\to Y$ and $M_4: Y \to Y$ are bounded linear operators.

Note that the differentiability of $\opR$ follows from the calculations in this
section since they could be carried out to the second order.  However, we have
chosen to only make first order calculations since they are already quite
involved.

\begin{remark}
  The fact that the derivative on the pure decomposed maps can be written as an
  infinite matrix is one of the reasons why we restrict ourselves to the pure
  decompositions.  Deformations of pure decompositions are also easy to
  deal with since they are `monotone' in the sense that the dynamical intervals
  that define the renormalization move monotonically under such deformations.
  This makes it possible to estimate the elements of the derivative matrix.
\end{remark}

\begin{theorem} \label{thm:opnorms}
  There exist constants $k$ and~$K$ such that if $\decomp f \in \dinvset \cap
  \lorenzd_\Omega$ and $1-\rcv(\opR \decomp f) \geq \lambda$ for some $\lambda
  \in (0,1)$ (not depending on~$\decomp f$), then
  \begin{align*}
    \norm{M_1 x} &\geq k \min\{\abs{U}\inv,\abs{V}\inv\} \cdot \norm{x}, &
      \norm{M_2} &\leq K \abs{C}\inv, \\
    \norm{M_3x} &\leq K \rho'
      \left(\frac{\abs{x_1}}{\abs{U}} + \frac{\abs{x_2}}{\abs{V}}\right), &
      \norm{M_4} &\leq K \rho' \abs{C}\inv,
  \end{align*}
  where $\rho' = \max\{\eps',\distortion\decomp\phi',\distortion\decomp\psi'\}$
  and $\lbb$ is sufficiently large.
\end{theorem}

\begin{remark}
  The set $\dinvset$ is introduced in \defref{dinvset} and $\Omega$ is given by
  \defref{Omega-K} as always.  Note that the decompositions of $\decomp f \in
  \dinvset$ are pure and hence $\opO\decomp f \in \setLS$ by
  \propref{pure-S-neg}.  Finally, the condition on $\rcv(\opR \decomp f)$ is
  used to avoid maps whose renormalization has a right branch which is close to
  being trivial (see the proof of \propref{c-partials}).
\end{remark}

\begin{proof}
  The proof of this theorem is split up into a few propositions that are in
  this section.  The estimate for $M_1$ is given in \corref{expansion}.  The
  estimates for $M_2$ and~$M_4$ follow from Propositions \ref{prop:c-partials}
  and~\ref{prop:M2-M4}.  Finally, the estimate for $M_3$ follows from
  Propositions \ref{prop:c-partials} and~\ref{prop:M3}.
\end{proof}

\begin{notation}
  Let $\decomp f = (u,v,c,\decomp\phi,\decomp\psi)$ and as always use primes to
  denote the renormalization $\opR \decomp f =
  (u',v',c',\decomp\phi',\decomp\psi')$.  We introduce special notation for the
  diffeomorphic parts of the renormalization \emph{before} rescaling:
  \begin{align}
    \Phi &= f_1^a \circ \phi,& \Psi &= f_0^b \circ \psi, \label{eq:Phi-Psi-C}
  \end{align}
  so that $\Phi: U \to C$, $\Psi: V \to C$, and $C = (p,q)$.
  Note that $p$ and~$q$ are by definition periodic points of periods $a+1$
  and~$b+1$, respectively.

  We will use the notation $\partial_s t$ to denote the partial derivative
  of~$t$ with respect to~$s$.  In the formulas below we write $\partial t$ to
  mean the partial derivative of~$t$ with respect to any direction.

  The notation $g(x) \asymp y$ is used to mean that there exists $K < \infty$
  not depending on~$g$ such that $K\inv y \leq g(x) \leq K y$ for all $x$ in
  the domain of~$g$.
\end{notation}

The $\partial$
operator satisfies the following rules:
\begin{lemma} \label{lem:partial-props}
  The following expressions hold whenever they make sense:
  \begin{align}
    \partial(f \circ g)(x) &=
      \partial f(g(x)) + f'(g(x)) \partial g(x),
      \label{eq:partial-chain-rule} \\
    \partial\big(f^{n+1}\big)(x) &=
      \sum_{i=0}^n \opD f^{n-i}\big( f^{i+1}(x) \big)
      \partial f\big( f^i(x) \big),
      \label{eq:partial-iterate} \\
    \partial\big(f\inv\big)(x) &=
      -\frac{\partial f\big( f\inv(x) \big)}{f'\big( f\inv(x) \big)}.
      \label{eq:partial-inv}
  \end{align}
  Furthermore, if $f(p)=p$ then
  \begin{equation}
    \partial p = -\frac{\partial f(p)}{f'(p)-1}.
    \label{eq:partial-fixedpt}
  \end{equation}
\end{lemma}

\begin{remark} \label{rem:partial-props}
  The $\partial$ operator clearly also satisfies the product rule
  \begin{equation}
    \partial (f \cdot g)(x) = \partial f(x) g(x) + f(x) \partial g(x).
  \end{equation}
  This and the chain rule gives the quotient rule
  \begin{equation}
    \partial(f/g)(x) =
    \frac{\partial f(x) g(x) - f(x) \partial g(x)}{g(x)^2}.
  \end{equation}
\end{remark}

\begin{proof}
  Equation~\eqnref{partial-chain-rule} implies the other three.  The second
  equation is an induction argument and the last two follow from
  \[
    0 = \partial(x) = \partial\big( f \circ f\inv(x) \big)
    = \partial f\big( f\inv(x) \big) + f'\big( f\inv(x) \big)
    \partial\big( f\inv(x) \big),
  \]
  and
  \[
    \partial(p) = \partial( f(p) ) = \partial f(p) + f'(p) \partial p.
  \]
  Equation~\eqnref{partial-chain-rule} itself can be proved by writing
  $f_\eps(x) = f(x) + \eps \hat f(x)$, $g_\eps(x) = g(x) + \eps \hat g(x)$ and
  using Taylor expansion:
  \begin{align*}
    f_\eps (g_\eps(x)) &=
      f_\eps(g(x)) + \eps f_\eps'(g(x)) \hat g(x) + \bigoh(\eps^2) \\
    &= f(g(x)) + \eps \big\{ \hat f(g(x))
      + f'(g(x)) \hat g(x) \big\} + \bigoh(\eps^2).\qedhere
  \end{align*}
\end{proof}

We now turn to computing the derivative matrix~$M$.  The first three rows
of~$M$ are given by the following formulas.

\begin{lemma} \label{lem:uvc-partials}
  The partial derivatives of $u'$, $v'$ and $c'$ are given by
  \begin{align*}
    \partial u' &=
      \frac{\partial\left(Q_0(c) - Q_0(p)\right) - u' \cdot
      \partial\left(\Phi\inv(q) - \Phi\inv(p)\right)}{\abs{U}}, \\
    \partial v' &=
      \frac{\partial\left(Q_1(q) - Q_1(c)\right) - v' \cdot
      \partial\left(\Psi\inv(q) - \Psi\inv(p)\right)}{\abs{V}}, \\
    \partial c' &=
      \frac{\partial(c - p) - c' \cdot \partial(q - p)}{\abs{C}}.
  \end{align*}
\end{lemma}

\begin{proof}
  Use \eqnref{R-uvc}, \lemref{partial-props} and \remref{partial-props}.
\end{proof}

Let us first consider how to use these formulas when deforming in the $u$, $v$
or $c$ directions (i.e.\ the first three columns of~$M$).  Almost everything in
these
formulas is completely explicit --- we have expressions for $Q_0$ and~$Q_1$ so
evaluating for example $\partial_u Q_0(c)$ is routine.  In order to evaluate
for example the
term $\partial_u \Psi\inv(q)$ we make use of~\eqnref{partial-inv}
and~\eqnref{partial-iterate}.  This involves estimating the sum
in~\eqnref{partial-iterate} which can be done with mean value theorem
estimates.  The terms $\partial p$ and~$\partial q$ are evaluated using
\eqnref{partial-fixedpt} and the fact that $p = \Phi \circ Q_0(p)$ and $q =
\Psi \circ Q_1(q)$.  There are a few shortcuts to make the calculations simpler
as well, for example $\partial_u \Phi = 0$ since $\Phi$ does not contain $Q_0$
which is the only term that depends on~$u$, and so on.

Deforming in the $\decomp\phi$ or $\decomp\psi$ directions (there are countably
many such directions) is similar.  Here we make use of the fact that the
decompositions are pure and we have an explicit formula \eqnref{pure} for pure
maps where the free parameter represents the signed distortion
(see~\remref{pures}), so we can
compute their derivative, partial derivative with respect to distortion etc.
These deformations will affect the partial derivatives of any expression
involving $\Phi$ or~$\Psi$, but all others will not `see' these deformations.
The calculations involved do not make any particular use of which direction we
deform in, so even though there are countably many directions we essentially
only need to perform one calculation for $\decomp\phi$ and another
for~$\decomp\psi$.

We now turn to computing the partial derivatives of $\decomp\phi'$
and~$\decomp\psi'$.

\begin{lemma} \label{lem:diffeo-partials}
  Let $\mu_{s'} = \opZ(\mu_s;I)$, where $\mu_s,\mu_{s'} \in \pures$ and $I =
  [x,y]$.  Then
  \[
    \partial s' =
    \opN\mu_s(y) \partial y
    - \opN\mu_s(x) \partial x
    + \frac{\partial(\opD\mu_s)(y)}{\opD\mu_s(y)}
    - \frac{\partial(\opD\mu_s)(x)}{\opD\mu_s(x)}.
  \]
\end{lemma}

\begin{proof}
  By definition $s = \log\{\opD\mu_s(1)/\opD\mu_s(0)\}$.  Distortion is
  invariant under zooming, so this shows that $s' =
  \log\{\opD\mu_s(y)/\opD\mu_s(x)\}$.  A calculation gives
  \[
    \partial\big(\log \opD\mu_s(x)\big) =
    \frac{\partial(\opD\mu_s)(x)}{\opD\mu_s(x)} + \opN\mu_s(x) \partial x.
    \qedhere
  \]
\end{proof}

By definition $\decomp\phi'$ consists of maps of the form $\opZ(\mu_s;I)$ (as
well as finitely many of the form $\opZ(Q;I)$ but these can be thought of as
$\lim_{s\to\pm\infty} \opZ(\mu_s;I)$).  Hence the above lemma shows us how to
compute the partial derivatives at each time in $\decomp\phi'$.  Note that we
implicitly identify $\reals$ with $\pures$ via $s \mapsto \mu_s$.

In order to use the lemma we also need a way to evaluate the terms $\partial x$
and~$\partial y$.  One way to do this is to express these in terms of $\partial
p$ and~$\partial q$ which have already been computed at this stage.  If we let
$T: I \to [p,q]$ denote the `transfer map' to~$C$, then $p = T(x)$ and hence
\eqnref{partial-chain-rule} shows that
\[
  \partial x = \frac{\partial p - \partial T(x)}{\opD T(x)}.
\]
The terms $\partial T$ and $\opD T$ can be bounded by $\partial \Phi$
and~$\opD\Phi$ (or $\partial \Psi$ and~$\opD\Psi$) all of which have already
been computed as well.

\medskip
We will now compute the $M_1$ part of the derivative matrix.  Note that $M_1$
has nothing to do with decompositions so the following proposition is stated
for nondecomposed Lorenz maps.

\begin{proposition} \label{prop:M1}
  If $f \in \setK \cap \setLS_\Omega$, then
  \begin{align*}
    M_1 =
    \begin{pmatrix}
      \frac{1}{\abs{U}} \left(
      1 + \frac{1-u'}{u} \frac{Q(p)}{\opD f^{a+1}(p) - 1} \right)
      &
      -\frac{1}{\abs{U}} \frac{u'}{v}
      \frac{\opD\Psi(\Psi\inv(q))}{\opD\Phi(\Phi\inv(q))}
      \frac{1 - Q(q)}{\opD f^{b+1}(q) - 1}
      \\
      -\frac{1}{\abs{V}} \frac{v'}{u}
      \frac{\opD\Phi(\Phi\inv(p))}{\opD\Psi(\Psi\inv(p))}
      \frac{Q(p)}{\opD f^{a+1}(p) - 1}
      &
      \frac{1}{\abs{V}} \left(
      1 + \frac{1-v'}{v} \frac{1 - Q(q)}{\opD f^{b+1}(q) - 1} \right)
    \end{pmatrix}
    + M_1^e,
  \end{align*}
  where the error term $M_1^e$ is negligible.
\end{proposition}

\begin{remark} \label{rem:M1}
  From \secref{inv-set} we know that the critical point of the renormalization
  is very close to $1$ and that the distortion of the diffeomorphic parts of
  the renormalization are bounded by $\delta$ (which is very small).  From
  these two facts we can get an idea of the size of the entries of $M_1$.  For
  example, $u'$ is very close to $1$ since $c'$ is (and $\opR f$ is assumed to
  be nontrivial so $u' \geq \e^{-\delta}c'$).  Furthermore, $\opD f^{a+1}(p) =
  \opD(\opR f)(0)$ and $\opD f^{b+1} = \opD(\opR f)(1)$ since an affine change
  of coordinates does not change the derivative, so the distortion bounds for
  $\opR f$ implies that $\opD f^{a+1}(p) \asymp \alpha u' / c'$ and $\opD
  f^{b+1}(q) \asymp \alpha v' / \eps'$ (these expressions come from the
  derivative of $Q(x)$, see \eqnref{standard}).
\end{remark}

\begin{proof}
  We begin by computing $\partial p$ and $\partial q$.  Use $\Phi \circ Q_0(p)
  = p$, $\Psi \circ Q_1(q) = q$, and~\eqnref{partial-fixedpt} to get
  \begin{align}
    \partial_u p &=
      -\frac{\opD\Phi(Q_0(p)) \partial_u Q_0(p)}{\opD f^{a+1}(p) - 1},&
      \partial_u q &=
      -\frac{\partial_u\Psi(Q_1(q))}{\opD f^{b+1}(q) - 1},
      \label{eq:dupq} \\
    \partial_v p &=
      -\frac{\partial_v\Phi(Q_0(p))}{\opD f^{a+1}(p) - 1},&
      \partial_v q &=
      -\frac{\opD\Psi(Q_1(q)) \partial_v Q_1(q)}{\opD f^{b+1}(q) - 1}.
      \label{eq:dvpq}
  \end{align}
  Here we have used that $\partial_u\Phi = 0$ and $\partial_v\Psi = 0$.

  \smallskip Next, let us estimate $\partial_u\Psi$.  Let $x \in V$ and let
  $x_i = f^i \circ \psi(x)$.  From \eqnref{partial-iterate} we get
  \[
    \partial_u\Psi(x) = \partial_u\big(f^b \circ \psi)(x)
    = \partial_u f( x_{b-1}) +
    \sum_{i=1}^{b-1} \opD f^{b-i}(x_i)
    \partial_u f(x_{i-1}),
  \]
  where $\partial_u f(x) = \phi'(Q_0(x)) Q_0(x) / u$.  Note that $\partial_u
  f(x_{i-1}) \leq \e^{2\delta} x_i / u$.  In order to bound the sum we divide
  the estimate into two parts.  Let $n < b$ be the smallest integer such that
  $\opD f(x_i) \leq 1$ for all $i \geq n$.  In the part where $i < n$ we
  estimate
  \[
    \opD f^{b-i}(x_i) x_i = \opD f^{n-i}(x_i) \opD f^{b-n}(x_n) x_i \leq
    K_1 \frac{x_n}{x_i} \opD f(x_{b-1}) x_i \leq K_2 \eps^{1-1/\alpha}.
  \]
  Here we have used the mean value theorem to find $\xi_i \leq x_i$ such that
  $\opD f^{n-i}(\xi_i) = x_n/x_i$  and $\opD f^{n-i}(x_i) \leq K_1 \opD
  f^{n-i}(\xi_i)$, since $\phi$ has very small distortion.  In the part where
  $i \geq n$ we estimate
  \[
    \opD f^{b-i}(x_i) x_i \leq \opD f(x_{b-1}) \leq K \eps^{1-1/\alpha}.
  \]
  Summing over the two parts gives us the estimate
  \[
    \sum_{i=1}^{b-1} \opD f^{b-i}(x_i) \partial_u f(x_{i-1}) \leq
    K (b-1) \eps^{1-1/\alpha}.
  \]
  Hence
  \begin{equation} \label{eq:duPsi}
    \partial_u\Psi(x) = \partial_u f\big(f^{b-1} \circ \psi(x)\big)
    + \bigoh\big(b \eps^{1-1/\alpha}\big) \approx 1.
  \end{equation}

  \smallskip
  We will now estimate $\partial_v\Phi$.  Let $x \in U$ and let $x_i = f^i
  \circ \phi(x)$.  Similarly to the above, we have
  \[
    \partial_v\Phi(x)
    = \partial_v f( x_{a-1}) +
    \sum_{i=1}^{a-1} \opD f^{a-i}(x_i)
    \partial_v f(x_{i-1}),
  \]
  where $\partial_v f(x) = -\psi'(Q_1(x)) (1 - Q_1(x)) / v$.  By the mean value
  theorem there exists $\xi_i \in [x_i,1]$ such that $\opD f^{a-i}(\xi_i) =
  (1-x_a)/(1-x_i)$, since $f^{a-i}(x_i) = x_a$.  From \lemref{derivs} it
  follows that $\opD f^{a-i}(x_i) \asymp \opD f^{a-i}(\xi_i)$.  Putting all of
  this together we get that the sum above is proportional to
  \[
    \sum_{i=1}^{a-1} \opD f^{a-i}(\xi_i) (1 - x_i) = (a-1) (1 - x_a).
  \]
  Thus
  \begin{equation} \label{eq:dvPhi}
    \partial_v\Phi(x) \asymp -a \eps,
  \end{equation}
  since $x_a \in C$ and hence $1-x_a = \eps + \bigoh(\abs{C}) \approx \eps$.

  \smallskip
  We now have all the ingredients we need to compute $M_1$.
  \lemref{uvc-partials} shows that
  \begin{align*}
    \abs{U} \partial_u u' &= \partial_u Q_0(c) - \partial_u Q_0(p)
    - Q_0'(p) \partial_u p \\
    &\quad- u' \left(
    \opD\Phi\inv(q) \partial_u q - \opD\Phi\inv(p) \partial_u p \right).
  \end{align*}
  Here we have used $\partial_u\Phi = 0$.  Now use \eqnref{dupq} to get
  \[
    Q_0'(p) \partial_u p = -\partial_u Q_0(p)
    \frac{\opD f^{a+1}(p)}{\opD f^{a+1}(p) - 1},
    \quad
    \opD\Phi\inv(p) \partial_u p =
    -\frac{\partial_u Q_0(p)}{\opD f^{a+1}(p) - 1}.
  \]
  Thus
  \begin{equation}
    \abs{U} \partial_u u' = 1 +
    \frac{(1 - u')\partial_u Q_0(p)}{\opD f^{a+1}(p) - 1}
     + \frac{u' \partial_u\Psi(Q_1(q))}%
    {\opD\Phi(\Phi\inv(q)) (\opD f^{b+1}(q) - 1)}.
  \end{equation}
  The last term is much smaller than one because of \eqnref{duPsi} and since
  $\abs{\opD\Phi} \gg 1$ (and also $\opD f^{b+1}(q) \asymp v'\alpha/\eps' \geq
  \e^{-\delta}\alpha$).

  From \lemref{uvc-partials} we get
  \begin{align*}
    \abs{V} \partial_v v' &= \partial_v Q_1(q) + Q_1'(q) \partial_v q
    - \partial_v Q_1(c) \\
    &\quad- v' \left(
    \opD\Psi\inv(q) \partial_v q - \opD\Psi\inv(p) \partial_v p \right).
  \end{align*}
  Here we have used $\partial_v\Psi = 0$.  Now use \eqnref{dvpq} to get
  \[
    Q_1'(q) \partial_v q =
    -\frac{\partial_v Q_1(q) \opD f^{b+1}(q)}{\opD f^{b+1}(q) - 1},
    \qquad
    \opD\Psi\inv(q)\partial_v q =
    -\frac{\partial_v Q_1(q)}{\opD f^{b+1}(q) - 1}.
  \]
  Thus
  \begin{equation}
    \abs{V} \partial_v v' = 1 - 
    \frac{(1-v')\partial_v Q_1(q)}{\opD f^{b+1}(q) - 1}
    - \frac{v' \partial_v\Phi(Q_0(p))}%
    {\opD\Psi(\Psi\inv(p))(\opD f^{a+1}(p) - 1)}.
  \end{equation}
  The last term is much smaller than one by \eqnref{dvPhi} and since
  $\abs{\opD\Psi} \gg 1$ (and also $\opD f^{a+1}(p) \asymp \alpha u'/c' \geq
  \e^{-\delta}\alpha$).

  From \lemref{uvc-partials} we get
  \begin{align*}
    \abs{U} \partial_v u' = -Q_0'(p) \partial_v p - u' \Big(
    \partial_v\Phi\inv(q) &+ \opD\Phi\inv(q) \partial_v q \\
    - \partial_v\Phi\inv(p) &- \opD\Phi\inv(p) \partial_v p \Big).
  \end{align*}
  Let us prove that that the dominating term is the one with $\partial_v q$.
  From~\eqnref{dvpq} we get
  \[
    \partial_v q = -\frac{\partial_v Q_1(q)}{Q_1'(q)}
    \frac{\opD f^{b+1}(q)}{\opD f^{b+1}(q) - 1},
  \]
  which diverges as $\lbb \to \infty$, since $\abs{R}/\eps \to 0$ and hence
  $Q_1'(q) \to 0$ (by the proof of~\propref{ubeps}).  From
  \eqnref{dvpq} and~\eqnref{dvPhi} we get that $\partial_v p \to 0$, which
  shows that
  the last term is dominated by the term with $\partial_v q$.  Now,
  $\partial_v\Phi\inv(x) = -\partial_v\Phi(x)/\opD\Phi(x)$, which combined with
  \eqnref{dvPhi} shows that the term with $\partial_v q$ dominates the two
  terms with $\partial_v\Phi\inv$. Furthermore
  \[
    Q_0'(p) \partial_v p = -\frac{\partial_v\Phi(Q_0(p))}{\opD\Phi(Q_0(p))}
    \frac{\opD f^{a+1}(p)}{\opD f^{a+1}(p) - 1},
  \]
  which combined with \eqnref{dvPhi} shows that the term with $\partial_v q$
  dominates the above term.  Thus
  \begin{equation}
    \abs{U} \partial_v u' = 
    u' \frac{\opD\Psi(\Psi\inv(q))}{\opD\Phi(\Phi\inv(q))}
    \frac{\partial_v Q_1(q)}{\opD f^{b+1}(q) - 1} + e,
  \end{equation}
  where the error term $e$ is tiny compared with the other term on the
  right-hand side.

  From \lemref{uvc-partials} we get
  \begin{align*}
    \abs{V} \partial_u v' = Q_1'(q) \partial_u q - v' \Big(
    \partial_u\Psi\inv(q) &+ \opD\Psi\inv(q) \partial_u q \\
    - \partial_u\Psi\inv(p) &- \opD\Psi\inv(p) \partial_u p \Big).
  \end{align*}
  Let us prove that that the dominating term is the one with $\partial_u p$.
  From~\eqnref{dupq} we get
  \[
    \partial_u p = -\frac{\partial_u Q_0(p)}{Q_0'(p)}
    \frac{\opD f^{a+1}(p)}{\opD f^{a+1}(p) - 1},
  \]
  which diverges as $\lbb \to \infty$, since $\abs{L}/c \to 0$ and hence
  $Q_0'(p) \to 0$.  From \eqnref{dupq} and~\eqnref{duPsi} we get that
  $\partial_u q$ is bounded and hence the $\partial_u p$ term dominates the
  second term involving $\partial_u q$.  Now, $\partial_u\Psi\inv(x) =
  -\partial_u\Psi(y)/\opD\Psi(y)$, $y = \Psi\inv(x)$, which combined with
  \eqnref{duPsi} shows that the $\partial_u p$ term dominates the two terms
  involving $\partial_u\Psi\inv$.  Furthermore
  \[
    Q_1'(q) \partial_u q = -\frac{\partial_u\Psi(Q_1(q))}{\opD\Psi(Q_1(q))}
    \frac{\opD f^{b+1}(q)}{\opD f^{b+1}(q) - 1},
  \]
  which combined with \eqnref{duPsi} shows that the $\partial_u p$ term
  dominates the above term.
  Thus
  \begin{equation}
    \abs{V} \partial_u v' = -v'
    \frac{\opD\Phi(\Phi\inv(p))}{\opD\Psi(\Psi\inv(p))}
    \frac{\partial_u Q_0(p)}{\opD f^{a+1}(p) - 1} + e,
  \end{equation}
  where the error term $e$ is tiny compared with the other term on the
  right-hand side.
\end{proof}

\begin{corollary} \label{cor:detM1}
  If $f \in \setK \cap \setLS_\Omega$, then $\det M_1 > 0$ for $\lbb$ large
  enough.
\end{corollary}

\begin{proof}
  From \propref{dist} we get that
  \[
    \frac{\opD\Phi(\Phi\inv(p))}{\opD\Phi(\Phi\inv(q))}
    \frac{\opD\Psi(\Psi\inv(q))}{\opD\Psi(\Psi\inv(p))}
    \leq \e^{2\delta},
  \]
  since distortion is invariant under linear rescaling.  Now use this
  together with \propref{M1} to get
  \[
    \abs{U}\abs{V} \det M_1 >
      1 - \e^{2\delta} \frac{u' v'}{u v}
      \frac{Q(p)(1 - Q(q))}{(\opD f^{a+1}(p) - 1)(\opD f^{b+1}(q) - 1)}.
  \]
  Equation \eqnref{standard} gives
  \[
    \frac{Q(p)}{u} = 1 - \left(\frac{\abs{L}}{c}\right)^\alpha < 1
    \quad\text{and}\quad
    \frac{1-Q(q)}{v} = 1 - \left(\frac{\abs{R}}{\eps}\right)^\alpha < 1.
  \]
  \remref{M1} allows us to estimate
  \[
    \frac{u'}{\opD f^{a+1}(p)-1} \leq \frac{\e^\delta}{\alpha - \e^{2\delta}}
    \quad\text{and}\quad
    \frac{v'}{\opD f^{b+1}(q)-1} \leq \eps'
    \frac{\e^\delta}{\alpha - \e^{2\delta}}.
  \]
  Taken all together we get
  \[
    \abs{U}\abs{V} \det M_1 > 1 - \eps'
    \frac{\e^{4\delta}}{(\alpha-\e^{2\delta})^2} \to 1,
    \quad\text{as $\lbb\to\infty$,}
  \]
  by \propref{ubeps}.  In particular, $\det M_1 > 0$ for $\lbb$ large enough.
\end{proof}

\begin{corollary} \label{cor:expansion}
  There exists $k>0$ such that if $f$ is as above, then
  \[
    \norm{M_1x} \geq k \cdot \min\{\abs{U}\inv,\abs{V}\inv\}\cdot\norm{x}.
  \]
\end{corollary}

\begin{proof}
  Write $M_1$ as
  \[
    M_1 = \begin{pmatrix}
      \phantom{-}\frac{a}{\abs{U}} & -\frac{b}{\abs{V}} \\
      -\frac{c}{\abs{U}} & \phantom{-}\frac{d}{\abs{V}}
    \end{pmatrix}.
  \]
  (Here we have used that the distortion of $\Phi$ and~$\Psi$ are small, so
  $\opD\Phi/\opD\Psi \asymp \abs{V}/\abs{U}$.)  Then
  \[
    M_1\inv = (ad-bc)\inv \begin{pmatrix}
      d\abs{U} & b\abs{U} \\
      c\abs{V} & a\abs{V}
    \end{pmatrix}.
  \]
  We are using the max-norm, hence
  \[
    \norm{M_1\inv} = (ad-bc)\inv \cdot \max\{(b+d)\abs{U},(c+a)\abs{V}\}.
  \]
  It can be checked that $(b+d)/(ad-bc)$ and $(a+c)/(ad-bc)$ are bounded by
  some $K$.  Let $k = 1/K$ to finish the proof.
\end{proof}

\begin{proposition} \label{prop:c-partials}
  If $f \in \setK \cap \setLS_\Omega$ and $1-\rcv(\opR f) \geq \lambda$ for
  some $\lambda\in(0,1)$ (not depending on~$f$), then
  \begin{gather*}
    \partial_c u' \asymp -\abs{C}\inv, \quad
    \partial_c v' \asymp \abs{C}\inv, \quad
    \partial_c c' \asymp -c'\eps'\abs{C}\inv, \\
    \partial_u c' \asymp c' \eps' \abs{U}\inv, \quad
    \partial_v c' \asymp -c' \eps' \abs{V}\inv.
  \end{gather*}
\end{proposition}

\begin{proof}
  A straightforward calculation shows that
  \begin{equation} \label{eq:dcQ}
    \frac{\partial_c Q_0(x)}{Q_0'(x)} = -\frac{x}{c}
    \quad\text{and}\quad
    \frac{\partial_c Q_1(x)}{Q_1'(x)} = -\frac{1-x}{1-c}.
  \end{equation}
  This together with $\Phi \circ Q_0(p) = p$, $\Psi \circ Q_1(q) = q$,
  \eqnref{Phi-Psi-C} and~\eqnref{partial-fixedpt} gives
  \begin{equation*}
    \partial_c p =
      \frac{\tfrac{p}{c}\opD f^{a+1}(p) - \partial_c\Phi(Q_0(p))}%
      {\opD f^{a+1}(p) - 1},
    \quad
    \partial_c q =
      \frac{\tfrac{1-q}{\eps}\opD f^{b+1}(q) - \partial_c\Psi(Q_1(q))}%
      {\opD f^{b+1}(q) - 1}.
  \end{equation*}
  From \eqnref{partial-iterate} and~\eqnref{dcQ} we get
  \begin{align*}
    \partial_c\Phi(x) &= -\frac{1}{\eps} \sum_{i=0}^{a-1}
      \opD f^{a-i}(x_i) \cdot (1-x_i),&
      x_i &= f^i\circ\phi(x),\quad x \in U, \\
    \partial_c\Psi(x) &= -\frac{1}{c} \sum_{i=0}^{b-1}
      \opD f^{b-i}(x_i) \cdot x_i,&
      x_i &= f^i\circ\psi(x),\quad x \in V.
  \end{align*}
  Using a similar argument as in the proof of \propref{M1} this shows that
  \begin{equation*}
    \partial_c\Phi(x) \asymp -a
    \quad\text{and}\quad
    \partial_c\Psi(x) = -\bigoh(b\eps^{1-1/\alpha}),
  \end{equation*}
  and hence $\partial_c p \asymp 1$ and $\partial_c q \asymp 1$.

  Now apply \lemref{uvc-partials} using the fact that $\Phi\inv(p) = Q_0(p)$ to
  get
  \begin{equation*}
    \abs{U} \partial_c u' =
      -(1-u') \partial_c\big(Q_0(p)\big) - u' \partial_c\big(\Phi\inv(q)\big).
  \end{equation*}
  A calculation gives
  \[
    \partial_c\big(Q_0(p)\big) =
    \frac{\opD f^{a+1}(p) \left( \frac{p}{c} - \partial_c\Phi(Q_0(p))\right)}%
    {\opD\Phi(Q_0(p)) \left( \opD f^{a+1}(p) - 1\right)}
    \asymp \frac{1}{\opD\Phi(Q_0(p))}
  \]
  and
  \[
    \partial_c\big(\Phi\inv(q)\big) =
    \frac{\partial_c q - \partial_c\Phi\big(\Phi\inv(q)\big)}%
    {\opD\Phi\big(\Phi\inv(q)\big)}
    \asymp \frac{1}{\opD\Phi\big(\Phi\inv(q)\big)}.
  \]
  (In particular, both terms have the same sign.) But $\opD\Phi(x) \asymp
  \abs{C}/\abs{U}$, so this gives $\partial_c u' \asymp -\abs{C}\inv$.  The
  proof that $\partial_c v' \asymp \abs{C}\inv$ is almost identical.

  From \lemref{uvc-partials} we get
  \[
    \abs{C} \partial_c c' = c' (1 - \partial_c q) + \eps' (1 - \partial_c p),
  \]
  and hence
  \begin{align*}
    \partial_c c' &=
      \frac{c'}{\eps}
      \frac{\eps' \opD f^{b+1}(q) - \eps\abs{C}\inv (1-\partial_c\Psi(Q_1(q)))}%
      {\opD f^{b+1}(q)-1} \\
    &\quad+ \frac{\eps'}{c}
      \frac{c' \opD f^{a+1}(p) - c\abs{C}\inv (1-\partial_c\Phi(Q_0(p)))}%
      {\opD f^{a+1}(p)-1} \\
    &= -\frac{c'\big(1-\partial_c\Psi(Q_1(q))\big)}%
      {\abs{C}(\opD f^{b+1}(q)-1)} -
      \frac{\eps'\big(1-\partial_c\Phi(Q_0(p))\big)}%
      {\abs{C}(\opD f^{a+1}(p)-1)}
      + \bigoh(c'\eps'/\eps).
  \end{align*}
  From \remref{M1} we know that $\opD f^{a+1}(p) \asymp \alpha u'/c'$ and $\opD
  f^{b+1}(q) \asymp \alpha v'/\eps'$.  Note that $u' \approx 1$ for $f \in
  \setK \cap \setLS_\Omega$, but that $v'$ can in general be small (this
  happens if $f$ renormalizes to a map whose right branch is trivial).
  However, the assumption that $1-\rcv(\opR f) \geq \lambda$ implies that $v'
  \geq \e^{-\delta}\lambda$ and hence we may assume that $v'/\eps' \gg 1$ (by
  increasing $\lbb$ if necessary).  Thus, $\opD f^{a+1}(p) \asymp \alpha/c'$
  and $\opD f^{b+1}(q) \asymp \alpha/\eps'$ and by plugging this into the above
  equation we get $\partial_c c' \asymp -c'\eps'\abs{C}\inv$.  (Note that
  $\partial_c\Phi(x)<0$ and $\partial_c\Psi(x)<0$ so there is no cancellation
  happening.)

  Apply \lemref{uvc-partials} to get
  \[
    \abs{C} \partial_u c' = -c' \partial_u q - \eps' \partial_u p.
  \]
  This and the proof of \propref{M1} shows that
  \[
    \partial_u c' =
    \frac{c' \big(\opD f^{a+1}(p)-1\big) \partial_u\Psi(Q_1(q))
    + \eps' \big(\opD f^{b+1}(q)-1\big) \opD\Phi(Q_0(p)) \partial_u Q_0(p)}%
    {\abs{C}\big(\opD f^{a+1}(p)-1\big) \big(\opD f^{b+1}(q)-1\big)}.
  \]
  Since $c' (\opD f^{a+1}(p)-1) \asymp \alpha-c'$, $\eps' (\opD f^{b+1}(q)-1)
  \asymp \alpha - \eps'$, $\abs{\partial_u\Psi} \ll \abs{\opD\Phi}$, and
  $\partial_u Q_0(p) \approx 1$, this shows that
  \[
    \partial_u c' \asymp c' \eps' \frac{\opD\Phi(Q_0(p))}{\abs{C}}
    \asymp \frac{c' \eps'}{\abs{U}}.
  \]
  The proof that $\partial_v c' \asymp -c' \eps' \abs{V}\inv$ is almost
  identical.
\end{proof}

\begin{notation}
We need some new notation to state the remaining propositions.  Each pure map
$\phi_\sigma$ in the decomposition $\decomp\phi$ can be identified with a real
number which we denote $s_\sigma \in \reals$, and each $\psi_\tau$ in the
decomposition $\decomp\psi$ can be identified with a real number $t_\tau \in
\reals$:
\[
  \reals \ni s_\sigma \leftrightarrow \phi_\sigma = \decomp\phi(\sigma) \in
  \pures,
  \qquad
  \reals \ni t_\tau \leftrightarrow \psi_\tau = \decomp\psi(\tau) \in \pures.
\]
We put primes on these numbers to denote that they come from the
renormalization, so $s'_{\sigma'} \in \reals$ is identified with
$\decomp\phi'(\sigma')$ and $t'_{\tau'} \in \reals$ is identified with
$\decomp\psi'(\tau')$.  Note that $\sigma$, $\sigma'$ are used to denote times
for $\decomp\phi$, $\decomp\phi'$, and $\tau$, $\tau'$ are used to denote times
for $\decomp\psi$, $\decomp\psi'$, respectively.
\end{notation}

\begin{proposition} \label{prop:M3}
  There exists $K$ such that if $\decomp f \in \dinvset \cap \lorenzd_\Omega$,
  then
  \begin{align*}
    \abs{\partial_u s'_{\sigma'}} &\leq K\frac{\abs{s'_{\sigma'}}}{\abs{U}},&
    \abs{\partial_v s'_{\sigma'}} &\leq K\frac{\abs{s'_{\sigma'}}}{\abs{V}},&
    \abs{\partial_c s'_{\sigma'}} &\leq K\frac{\abs{s'_{\sigma'}}}{\abs{C}}, \\
    \abs{\partial_u t'_{\tau'}} &\leq K\frac{\abs{t'_{\tau'}}}{\abs{U}},&
    \abs{\partial_v t'_{\tau'}} &\leq K\frac{\abs{t'_{\tau'}}}{\abs{V}},&
    \abs{\partial_c t'_{\tau'}} &\leq K\frac{\abs{t'_{\tau'}}}{\abs{C}}.
  \end{align*}
\end{proposition}

\begin{proof}
  We will compute $\partial_v s'_{\sigma'}$; the other calculations are almost
  identical.  There are four cases to consider depending on which time in the
  decomposition $\decomp\phi'$ we are looking at:
  \begin{inparaenum}[(1)]
    \item $\decomp\phi'(\sigma') = \opZ(\phi_\sigma;I)$, \label{i:s}
    \item $\decomp\phi'(\sigma') = \opZ(\psi_\tau;I)$, \label{i:t}
    \item $\decomp\phi'(\sigma') = \opZ(Q_0;I)$, \label{i:Q0}
    \item $\decomp\phi'(\sigma') = \opZ(Q_1;I)$. \label{i:Q1}
  \end{inparaenum}
  In each case let $I=[x,y]$ and let $T:I \to C$ be the `transfer map'
  to~$C$.  This means that $T = f^i \circ \gamma$ for some $i$
  and $\gamma$ is a partial composition (e.g.\ $\gamma =
  \opO_{\geq\sigma}(\decomp\phi)$ in case~\ref{i:s}) or a pure map (in cases
  \ref{i:Q0} and~\ref{i:Q1}).

  In case~\ref{i:s} \lemref{diffeo-partials} gives
  \[
    \partial_v s'_{\sigma'} =
    \frac{\opN\phi_\sigma(y)}{\opD T(y)} (\partial_v q - \partial_v T(y)) -
    \frac{\opN\phi_\sigma(x)}{\opD T(x)} (\partial_v p - \partial_v T(x)).
  \]
  By \lemref{zoom} $\opN\phi_\sigma(y) = \opN\phi'_{\sigma'}(1)/\abs{I}$ and
  hence
  \[
    \frac{\opN\phi_\sigma(y)}{\opD T(y)} \asymp 
    \frac{\opN\phi'_{\sigma'}(1)/\abs{I}}{\abs{C}/\abs{I}} \asymp 
    \frac{s'_{\sigma'}}{\abs{C}}.
  \]
  Here we have used that the nonlinearity of $\phi'_{\sigma'}$ does not change
  sign so $s'_{\sigma'} = \int \opN\phi'_{\sigma'}$ and that $\int
  \opN\phi'_{\sigma'} \approx \opN\phi'_{\sigma'}(1)$ since the nonlinearity is
  close to being constant (which is true since $\decomp\phi'$ is pure and has
  very small norm).

  We now need to estimate $\partial_v T$ but this can very roughly be bounded
  by $\partial_v\Phi$ since
  \[
    \partial_v T(y) = \partial_v f_1^i(\gamma(y)),
  \]
  so the estimate that was used for $\partial_v\Phi$ in the proof of
  \propref{M1} can be employed.  From the same proof we thus get that
  $\partial_v q$ dominates both $\partial_v p$ and $\partial_v T$.
  
  The above arguments show that
  \[
    \partial_v s'_{\sigma'}
    \asymp
    \frac{s'_{\sigma'}}{\abs{C}} \partial_v q
    \asymp
    -\frac{s'_{\sigma'}}{\abs{C}} \frac{\opD\Psi(Q_1(q))}{\opD f^{b+1}(q)-1}
    \asymp
    -\frac{s'_{\sigma'}}{\abs{V}} \frac{1}{\opD f^{b+1}(q)-1}.
  \]
  This concludes the calculations for case~\ref{i:s}.

  Case~\ref{i:t} is almost identical to case~\ref{i:s}.  Case \ref{i:Q1}
  differs in that \lemref{diffeo-partials} now gives two extra terms
  \begin{align*}
    \partial_v s'_{\sigma'} &=
      \frac{\opN Q_1(y)}{\opD T(y)} (\partial_v q - \partial_v T(y)) -
      \frac{\opN Q_1(x)}{\opD T(x)} (\partial_v p - \partial_v T(x)) \\
    &\quad+ \frac{\partial_v Q'_1(y)}{Q'_1(y)}
      - \frac{\partial_v Q'_1(x)}{Q'_1(x)}.
  \end{align*}
  However, $\partial_v Q_1 = 1/v$ so the last two terms cancel.  The rest of
  the calculations go exactly like in case~\ref{i:s}.  Case~\ref{i:Q0} is
  similar to case~\ref{i:Q1}.
\end{proof}

\begin{remark}
  A key point in the above proof is that deformations in a decomposition
  direction is monotone.  This is what allowed us to estimate the partial
  derivatives of the `transfer map' $T$ by the partial derivatives of $\Phi$
  or~$\Psi$.
\end{remark}

\begin{proposition} \label{prop:M2-M4}
  There exists $K$ and $\rho>0$ such that if $\decomp f \in \dinvset \cap
  \lorenzd_\Omega$ and $1-\rcv(\opR \decomp f) \geq \lambda$ for some
  $\lambda\in(0,1)$ (not depending on~$\decomp f$),
  then
  \begin{gather*}
    \abs{\partial_\star u'} \leq
      \frac{K\eps^\rho}{\abs{C}}, \qquad
    \abs{\partial_\star v'} \leq
      \frac{K\eps^\rho}{\abs{C}}, \qquad
    \abs{\partial_\star c'} \leq
      \frac{Kc'\eps'\eps^\rho}{\abs{C}}, \\
    \abs{\partial_\star s'_{\sigma'}} \leq
      \frac{K\eps^\rho\abs{s'_{\sigma'}}}{\abs{C}}, \qquad
    \abs{\partial_\star t'_{\tau'}} \leq
      \frac{K\eps^\rho\abs{t'_{\tau'}}}{\abs{C}},
  \end{gather*}
  for $\star \in \{s_\sigma,t_\tau\}$.
\end{proposition}

\begin{proof}
  Let us first consider $\partial_{s_\sigma}$, that is deformations in the
  direction of~$\phi_\sigma$.  Since $\phi_\sigma$ is pure we can use
  \eqnref{pure} to compute
  \begin{equation} \label{eq:ds}
    \partial_{s_\sigma} \phi_\sigma(x) \asymp -x(1-x).
  \end{equation}

  From \eqnref{partial-fixedpt} we get
  \[
    \partial_{s_\sigma} p =
    -\frac{\partial_{s_\sigma}\Phi\big(Q_0(p)\big)}{\opD f^{a+1}(p)-1}
    \quad\text{and}\quad
    \partial_{s_\sigma} q =
    -\frac{\partial_{s_\sigma}\Psi\big(Q_1(q)\big)}{\opD f^{b+1}(q)-1}.
  \]
  so the first thing to do is to calculate the partial derivatives of $\Phi$
  and~$\Psi$.

  Let $x \in U$, then
  \begin{align*}
    \partial_{s_\sigma}\Phi(x) &=
      \partial_{s_\sigma}\big( f_1^a \circ \opO_{>\sigma}(\decomp\phi) \circ
      \phi_\sigma \circ \opO_{<\sigma}(\decomp\phi) \big)(x) \\
    &= \opD\big( f_1^a \circ \opO_{>\sigma}(\decomp\phi) \big)
      \big( \opO_{\leq\sigma}(\decomp\phi)(x) \big) \cdot
      \partial_{s_\sigma}\phi_\sigma\big( \opO_{<\sigma}(\decomp\phi)(x) \big).
  \end{align*}
  Note that we have used that $f_1$ does not depend on $s_\sigma$.  From
  \eqnref{ds} we thus get that
  \begin{equation} \label{eq:phi-s}
    \abs{\partial_{s_\sigma}\Phi(x)}
    \leq K' \cdot \opD\Phi(x) (1-x) \leq K \eps.
  \end{equation}

  Let $x \in V$ and let $x_i = f_0^i \circ \psi(x)$.  As in the proof of
  \propref{M1} we have
  \[
    \partial_{s_\sigma}\Psi(x) = \partial_{s_\sigma} f_0(x_{b-1})
    + \sum_{i=1}^{b-1} \opD f_0^{b-i}(x_i) \partial_{s_\sigma} f_0(x_{i-1}).
  \]
  From \eqnref{ds} we get
  \begin{align*}
    \abs{\partial_{s_\sigma} f_0(x_{i-1})} &= 
      \abs[\big]{
      \opD\big( \opO_{>\sigma}(\decomp\phi)\big)
      \big( \opO_{\leq\sigma}(\decomp\phi) \circ Q_0(x_{i-1}) \big)
      \cdot \partial_{s_\sigma}
      \big( \opO_{<\sigma}(\decomp\phi) \circ Q_0(x_{i-1}) \big)} \\
    &\leq K\abs{x_i}.
  \end{align*}
  Using the same estimate as in the proof of \propref{M1} this
  shows that
  \begin{equation} \label{eq:psi-s}
    \abs{\partial_{s_\sigma}\Psi(x)} \leq
    K' (1 - x_b) + \bigoh(b \eps^{1-1/\alpha}) = \bigoh(b \eps^{1-1/\alpha}).
  \end{equation}

  We can now argue as in the proof of \propref{M1} to find bounds on
  $\partial_{s_\sigma} \star$ for $\star \in \{u',v',c'\}$.  From
  \lemref{uvc-partials} we get
  \begin{align*}
    \partial_{s_\sigma} u' &=
      \frac{1-u'}{\abs{U}}\cdot
      \frac{\partial_{s_\sigma}\Phi(Q(p))}{\opD\Phi(Q(p))} \cdot
      \frac{\opD f^{a+1}(p)}{\opD f^{a+1}(p)-1}
      + \frac{u'}{\abs{U}} \frac{\partial_{s_\sigma}\Phi\big(\Phi\inv(q)\big)
      - \partial_{s_\sigma} q}{\opD\Phi\big(\Phi\inv(q)\big)}, \\
    -\partial_{s_\sigma} v' &=
      \frac{1-v'}{\abs{V}}\cdot
      \frac{\partial_{s_\sigma}\Psi(Q(q))}{\opD\Psi(Q(q))} \cdot
      \frac{\opD f^{b+1}(q)}{\opD f^{b+1}(q)-1}
      + \frac{v'}{\abs{V}} \frac{\partial_{s_\sigma}\Psi\big(\Psi\inv(p)\big)
      - \partial_{s_\sigma} p}{\opD\Psi\big(\Psi\inv(p)\big)}, \\
    \partial_{s_\sigma} c' &=
      c' \cdot \frac{\partial_{s_\sigma}\Psi\big(Q_1(q)\big)}%
      {\opD f^{b+1}(q)-1} +
      \eps' \cdot \frac{\partial_{s_\sigma}\Phi\big(Q_0(p)\big)}%
      {\opD f^{a+1}(p)-1}.
  \end{align*}
  Use that $\opD\phi \asymp \abs{C}/\abs{U}$, $\opD\Psi \asymp
  \abs{C}/\abs{V}$, $\opD f^{a+1}(p) \asymp \alpha/c'$ and $\opD f^{b+1}(q)
  \asymp \alpha/\eps'$ (see the proof of \propref{c-partials}) to finish the
  estimates for $\partial_{s_\sigma} u'$, $\partial_{s_\sigma} v'$ and
  $\partial_{s_\sigma} c'$.  Note that $b\eps^r \to 0$ for any $r > 0$ so it is
  clear from \eqnref{phi-s} and~\eqnref{psi-s} that we can find a $\rho>0$ such
  that $\abs{\partial_{s_\sigma}\Phi} < K\eps^\rho$ and
  $\abs{\partial_{s_\sigma}\Psi} < K\eps^\rho$.

  In order to find bounds for $\star \in \{s'_{\sigma'},t'_{\tau'}\}$ we
  argue as in the proof of \propref{M3}.  The last two terms from
  \lemref{diffeo-partials} are slightly different (when nonzero).  In this
  case they are given by
  \[
    \frac{\partial_{s_\sigma}\big(\opD\phi_\sigma\big)(y)}{\opD\phi_\sigma(y)}-
    \frac{\partial_{s_\sigma}\big(\opD\phi_\sigma\big)(x)}{\opD\phi_\sigma(x)}.
  \]
  Using \eqnref{ds} we can calculate this difference.  For $\abs{s_\sigma}\ll1$
  it is close to $y-x$ which turns out to be negligible.  All other details are
  exactly like the proof of~\propref{M3}.

  The estimates for $\partial_{t_\tau}$ are handled similarly.  The only
  difference is the estimates of the partial derivatives of $\Phi$ and~$\Psi$.
  These can be determined by arguing as in the above and the proof of
  \propref{M1} which results in
  \begin{equation}
    \abs{\partial_{t_\tau}\Phi(x)} \leq K\eps^{1-1/\alpha}
    \quad\text{and}\quad
    \abs{\partial_{t_\tau}\Psi(y)} \leq Ka\eps,
  \end{equation}
  for $x \in U$ and $y \in V$.  The remaining estimates are handled identically
  to the above.
\end{proof}


\section{Invariant cone field} 
\label{sec:cone-field}

A standard way of showing hyperbolicity of a linear map is to find an invariant
cone field with expansion inside the cones and contraction in the complement of
the cones.  In this section we show that the derivative of the renormalization
operator has an invariant cone field and that it expands these cones.  However,
our estimates on the derivative are not sufficient to prove contraction in the
complement of the cones so we cannot conclude that the derivative is
hyperbolic.  The results in this section are used in \secref{island-structure}
to study the structure of the parameter plane and in \secref{unstable} to
construct unstable manifolds in the limit set of renormalization.

Let
\[
  \cone(\decomp f,\kappa) = \{ (x,y) \mid \norm{y} \leq \kappa \norm{x} \}
\]
denote the standard horizontal $\kappa$--cone \index{cone field} on the tangent
space at $\decomp f$.  Recall that we decompose the tangent space into a
two-dimensional subspace with coordinate $x$ and a codimension two subspace
with coordinate~$y$.  The $x$--coordinate corresponds to the $(u,v)$--subspace
in~$\lorenzd$.  We use the max-norm so if $z=(x,y)$ then $\norm{z} =
\max\{\norm{x},\norm{y}\}$.

\begin{proposition} \label{prop:cone-field}
  Assume $\decomp f \in \dinvset \cap \lorenzd_\Omega$ and $1-\rcv(\opR\decomp
  f) \geq \lambda$ for some $\lambda \in (0,1)$ (not depending on $\decomp f$).
  Define
  \[
    \lb\kappa(\decomp f) = \lb K
    \max\{\eps,\distortion\decomp\phi,\distortion\decomp\psi\}
    \quad\text{and}\quad
    \ub\kappa(\decomp f) = \ub K
    \min\left\{\frac{\abs{C}}{\abs{U}},\frac{\abs{C}}{\abs{V}}\right\}.
  \]
  It is possible to choose $\ub K$, $\lb K$ (not depending on~$\decomp f$) such
  that if $\kappa \leq \ub\kappa(\decomp f)$, then
  \[
    \opD\opR_{\decomp f}\big(\cone(\decomp f,\kappa)\big) \subset
    \cone\big(\opR\decomp f,\lb\kappa(\opR \decomp f)\big)\;,
  \]
  for $\lbb$ large enough.
  In particular, the cone field $\decomp f \mapsto \cone(\decomp f,1)$ is
  mapped strictly into itself by~$\opD\opR$.
\end{proposition}

\begin{remark} \label{rem:fat-to-thin}
  Note that as $\lbb$ increases, $\lb\kappa \downarrow 0$ and $\ub\kappa
  \uparrow \infty$.  Thus a fatter and fatter cone is mapped into a thinner and
  thinner cone.  In particular, the invariant subspaces inside the thin cone
  and the complement of the fat cone eventually line up with the coordinate
  axes.
\end{remark}

\begin{proof}
  Assume $\norm{y} \leq \kappa \norm{x}$.  Let $z' = Mz$, where $M =
  \opD\opR_{\decomp f}$ as in~\eqnref{M}, $z'=(x',y')$ and $z=(x,y)$.  Then
  \[
    \frac{\norm{x'}}{\norm{y'}} \geq
    \frac{\abs[\big]{\norm{M_1x} - \norm{M_2}\norm{y}}}%
    {\norm{M_3x} + \norm{M_4}\norm{y}} \geq
    \frac{\abs[\big]{\norm{M_1\frac{x}{\norm{x}}} - \kappa\norm{M_2}}}%
    {\norm{M_3\frac{x}{\norm{x}}} + \kappa\norm{M_4}}.
  \]
  We are interested in a lower bound on $\norm{x'}/\norm{y'}$ so this shows
  that we need to minimize
  \[
    g(x) = \frac{\abs[\big]{\norm{M_1x} - \kappa\norm{M_2}}}%
    {\norm{M_3x} + \kappa\norm{M_4}},
  \]
  subject to the constraint $\norm{x} = \max\{\abs{x_1},\abs{x_2}\}=1$.  We can
  write $M_1$ on the form
  \[
    M_1 = \begin{pmatrix}
      \frac{m_{11}}{\abs{U}} & -\frac{m_{12}}{\abs{V}} \\
      -\frac{m_{21}}{\abs{U}} & \frac{m_{22}}{\abs{V}} \\
    \end{pmatrix},
  \]
  where the entries $m_{ij}$ are positive, bounded, and $m_{ii} \geq 1$, by
  \propref{M1}.  Furthermore, by \thmref{opnorms} we know that
  \[
    \norm{M_3x} \leq K \rho' \left( \frac{\abs{x_1}}{\abs{U}} +
    \frac{\abs{x_2}}{\abs{V}} \right).
  \]
  Hence, if $\abs{x_1} = 1$, then 
  \[
    g(x) \geq
    \frac{\max\big\{\abs[\big]{\frac{m_{11}}{\abs{U}} -
    \frac{m_{12} x_2}{\abs{V}}},
    \abs[\big]{\frac{m_{21}}{\abs{U}} - \frac{m_{22} x_2}{\abs{V}}} \big\}
    - \kappa\norm{M_2}}%
    {K \rho' \left( \frac{1}{\abs{U}} + \frac{\abs{x_2}}{\abs{V}} \right) +
    \kappa\norm{M_4}}
    = g_1(x_2)
  \]
  and if $\abs{x_2} = 1$, then
  \[
    g(x) \geq
    \frac{\max\big\{\abs[\big]{\frac{m_{11} x_1}{\abs{U}} -
    \frac{m_{12}}{\abs{V}}},
    \abs[\big]{\frac{m_{21} x_1}{\abs{U}} - \frac{m_{22}}{\abs{V}}} \big\}
    - \kappa\norm{M_2}}%
    {K \rho' \left( \frac{\abs{x_1}}{\abs{U}} + \frac{1}{\abs{V}} \right) +
    \kappa\norm{M_4}}
    = g_2(x_1)
  \]
  Thus we are interested in minimizing $g_i(t)$ for $i=1,2$ and $t \in \uint$
  (note that $g_i(-t) \geq g_i(t)$ for $t \in \uint$ so we do not need to
  consider negative $t$).

  The maps $g_i$ are piecewise M\"obius maps (which are also nonsingular); in
  particular, they are piecewise monotone so any minimum is assumed at $0$,
  $1$, or at a boundary of monotonicity.  A boundary of monotonicity can only
  occur when the two terms inside the max term in the numerator are equal.  By
  solving the equations
  \[
    \frac{m_{11}}{\abs{U}} - \frac{m_{12}t}{\abs{V}} =
    \pm \left( \frac{m_{21}}{\abs{U}} - \frac{m_{22}t}{\abs{V}} \right)
  \]
  we see that $g_1$ has (at most) two points, $t_-$ and $t_+$, where it is
  not monotone on any neighborhood.  These points are
  \[
    t_- = \frac{\abs{V}}{\abs{U}} \frac{m_{11}-m_{21}}{m_{12}-m_{22}}
    \quad\text{and}\quad
    t_+ = \frac{\abs{V}}{\abs{U}} \frac{m_{11}+m_{21}}{m_{12}+m_{22}}.
  \]
  From similar considerations we see that $g_2$ has (at most) two points where
  it is not monotone on any neighborhood, namely $t_-\inv$ and $t_+\inv$.
  Note that we say ``at most'' here since we do not know if $t_\pm \in \uint$
  or if $t_\pm\inv \in \uint$, nor will it turn out to matter.

  Thus, to minimize $g(x)$ we only have to find the minimum of $g_1(0)$,
  $g_2(0)$, $g_1(1) = g_2(1)$, $g_1(t_\pm)$ and $g_2(t_\pm\inv)$.  We will
  calculate these values one at a time.

  Consider $g_1(0)$ first.  From \thmref{opnorms} we get that\footnote{This is
  the only place where the condition on $\rcv(\opR f)$ is used.  It is
  necessary to get the $\rho'$ term in the bound on $\norm{M_4}$.}
  \[
    \norm{M_2} \leq K_1/\abs{C}
    \quad\text{and}\quad
    \norm{M_4} \leq K_2\rho'/\abs{C},
  \]
  and hence
  \[
    g_1(0) = \frac{\max\{m_{11},m_{21}\} - \kappa\norm{M_2}\abs{U}}%
    {K \rho' + \kappa\norm{M_4}\abs{U}}
    \geq
    \frac{1 - \kappa K_1 \abs{U}/\abs{C}}%
    {K \rho' + \kappa K_2 \rho' \abs{U}/\abs{C}}.
  \]
  In the inequality we used the fact that $m_{11} \geq 1$.
  Hence
  \begin{equation} \label{eq:g10}
    \kappa \leq \frac{\abs{C}}{2K_1\abs{U}}
    \quad\implies\quad
    g_1(0) \geq \frac{1}%
    {\rho' \left( 2K + K_2/K_1 \right)}.
  \end{equation}

  Consider $g_2(0)$:
  \[
    g_2(0) = \frac{\max\{m_{12},m_{22}\} - \kappa\norm{M_2}\abs{V}}%
    {K \rho' + \kappa\norm{M_4}\abs{V}}
    \geq
    \frac{1 - \kappa K_1 \abs{V}/\abs{C}}%
    {K \rho' + \kappa K_2 \rho' \abs{V}/\abs{C}}.
  \]
  In the inequality we used \thmref{opnorms} and the fact that $m_{22} \geq 1$.
  Hence
  \begin{equation} \label{eq:g20}
    \kappa \leq \frac{\abs{C}}{2K_1\abs{V}}
    \quad\implies\quad
    g_2(0) \geq \frac{1}%
    {\rho' \left( 2K + K_2/K_1 \right)}.
  \end{equation}

  Consider $g_1(t_\pm)$:
  \begin{align*}
    g_1(t_\pm) &= \frac{ \abs[\big]{m_{11} - m_{12}
    \frac{m_{11} \pm m_{21}}{m_{12} \pm m_{22}}} - \kappa\norm{M_2}\abs{U}}%
    {K \rho' \left( 1 + \abs[\big]{\frac{m_{11} \pm m_{21}}{m_{12} \pm m_{22}}}
    \right) + \kappa\norm{M_4}\abs{U}} \\
    &= \frac{ \abs{m_{11} m_{22} - m_{12}m_{21}}
    - \kappa\norm{M_2}\abs{U}\abs{m_{12} \pm m_{22}}}%
    {K \rho' \left( \abs{m_{11} \pm m_{21}} + \abs{m_{12} \pm m_{22}}
    \right) + \kappa\norm{M_4}\abs{U}\abs{m_{11} \pm m_{21}}}.
  \end{align*}
  There exists $\nu$ such that $\sum m_{ij} \leq \nu$ and by \corref{detM1}
  there exists $\mu>0$ such that $m_{11}m_{22}-m_{12}m_{21} \geq \mu$, so
  \[
    g_1(t_\pm) \geq \frac{\mu - \kappa \nu K_1 \abs{U}/\abs{C}}%
    {K \rho' \nu + \kappa \nu K_2 \rho' \abs{U}/\abs{C}},
  \]
  where we once again have used \thmref{opnorms}.  Hence
  \begin{equation} \label{eq:g1t}
    \kappa \leq \frac{\mu\abs{C}}{2K_1\nu\abs{U}}
    \quad\implies\quad
    g_1(t_\pm) \geq \frac{1}%
    {\rho' \left( \frac{2K\nu}{\mu} + \frac{K_2}{K_1} \right)}.
  \end{equation}

  An almost identical calculation for $g_2(t\inv_\pm)$ results in:
  \begin{equation} \label{eq:g2t}
    \kappa \leq \frac{\mu\abs{C}}{2K_1\nu\abs{V}}
    \quad\implies\quad
    g_2(t\inv_\pm) \geq \frac{1}%
    {\rho' \left( \frac{2K\nu}{\mu} + \frac{K_2}{K_1} \right)}.
  \end{equation}

  Finally, consider $g_1(1)$:
  \[
    g_1(1) = \frac{\max\big\{\abs[\big]{\frac{m_{11}}{\abs{U}} -
    \frac{m_{12}}{\abs{V}}},
    \abs[\big]{\frac{m_{21}}{\abs{U}} - \frac{m_{22}}{\abs{V}}} \big\}
    - \kappa\norm{M_2}}%
    {K \rho' \left( \frac{1}{\abs{U}} + \frac{1}{\abs{V}} \right) +
    \kappa\norm{M_4}}.
  \]
  We need to minimize the numerator, so introduce a variable $s$ and assume
  that
  \[
    \frac{m_{11}}{\abs{U}} = s \frac{m_{12}}{\abs{V}}, \quad s \in \reals.
  \]
  Let $H(s) = \max\{h_1(s), h_2(s)\}$, where
  \begin{align*}
    h_1(s) &= \abs[\bigg]{\frac{m_{11}}{\abs{U}} - \frac{m_{12}}{\abs{V}}}
    = \frac{m_{12}}{\abs{V}} \abs{s-1}, \\
    h_2(s) &= \abs[\bigg]{\frac{m_{21}}{\abs{U}} - \frac{m_{22}}{\abs{V}}}
    = \frac{1}{\abs{V}} \abs[\Big]{\frac{m_{12}m_{21}}{m_{11}} s - m_{22}}.
  \end{align*}
  The equation $H(s) = 0$ has two solutions: $s_1 = 1$ and $s_2 =
  m_{11}m_{22}/(m_{12}m_{21})$.  Note that $h_i$ is decreasing to the left of
  $s_i$ and increasing to the right of $s_i$, for $i=1,2$.  Also, $s_1 < s_2$
  by \corref{detM1} so $H(s)$ assumes its minimum at $s_\star$, where $s_\star$
  is defined by $h_1(s_\star) = h_2(s_\star)$ and $s_1 < s_\star < s_2$.
  Solving this equation gives
  \[
    s_\star = \frac{m_{11} (m_{12}+m_{22})}{m_{12} (m_{11}+m_{21})}.
  \]
  Thus
  \[
    \min H(s) = H(s_\star) = \frac{m_{11}m_{22} - m_{12}m_{21}}%
    {\abs{V} (m_{11}+m_{21})}.
  \]
  Note that we can also write $h_1(s) = m_{11}\abs{1-s\inv}/\abs{U}$ and
  thus
  \[
    \min H(s) = h_1(s_\star) = 
    \frac{m_{11}m_{22} - m_{12}m_{21}}%
    {\abs{U} (m_{12}+m_{22})},
  \]
  which shows that
  \[
    H(s) \geq \frac{\mu}{\nu} \max\{\abs{U}\inv, \abs{V}\inv\}.
  \]
  Putting all of this together, we arrive at
  \[
    g_1(1) \geq \frac{\frac{\mu}{\nu} \max\{\abs{U}\inv, \abs{V}\inv\}
    - \kappa K_1/\abs{C}}%
    {K \rho' 2 \max\{\abs{U}\inv, \abs{V}\inv\} + \kappa \rho' K_2/\abs{C}}.
  \]
  Hence
  \begin{equation} \label{eq:g11}
    \kappa \leq \frac{\mu\abs{C} \max\{\abs{U}\inv, \abs{V}\inv\}}%
    {2\nu K_1}
    \quad\implies\quad
    g_1(1) \geq \frac{1}%
    {\rho' \left( \frac{4\nu K}{\mu} + \frac{K_2}{K_1} \right)}.
  \end{equation}

  From \eqnref{g10}, \eqnref{g20}, \eqnref{g1t}, \eqnref{g2t} and \eqnref{g11}
  we get that, if $\norm{y} \leq \kappa\norm{x}$ and
  \[
    \kappa \leq \frac{\min\{1,\mu/\nu\}}{2K_1}
    \min\left\{ \frac{\abs{C}}{\abs{U}}, \frac{\abs{C}}{\abs{V}} \right\},
  \]
  then
  \[
    \norm{y'} \leq \norm{x'} \rho' (2K \max\{1,2\nu/\mu\} + K_2/K_1).\qedhere
  \]
\end{proof}

\begin{proposition} \label{prop:cone-expansion}
  Let $\decomp f \in \dinvset \cap \lorenzd_\Omega$ and $1-\rcv(\opR\decomp f)
  \geq \lambda$ for some $\lambda \in (0,1)$ (not depending on $\decomp f$).
  Then $\opD\opR$ is strongly expanding on the cone field $\decomp f \mapsto
  \cone(\decomp f,1)$.  Specifically, there exists $k>0$ (not depending
  on~$\decomp f$) such that
  \[
    \norm{\opD\opR_{\decomp f} z} \geq
    k \cdot \min\{\abs{U}\inv,\abs{V}\inv\} \cdot \norm{z},
    \qquad\forall z \in \cone(\decomp f,1) \setminus\{0\}.
  \]
\end{proposition}

\begin{proof}
  Use \corref{expansion} to get
  \[
    \norm{Mz} \geq \norm{M_1x+M_2y} \geq
    \abs[\big]{k\cdot\min\{\abs{U}\inv,\abs{V}\inv\} - \norm{M_2}} \cdot
    \norm{x}.
  \]
  Now use the fact that $\norm{z} = \norm{x}$ for $z \in \cone(\decomp f,1)$
  to finish the proof.
\end{proof}


\section{Archipelagos in the parameter plane} 
\label{sec:island-structure}

The term \emph{archipelago} was introduced by \citet{MdM01} to describe the
structure of the domains of renormalizability in the parameter plane for
families of Lorenz maps.  In this section we show how the information we have
on the derivative of the renormalization operator can be used to prove that the
structure of archipelagos must be very rigid.

Fix $c_*$, $\phi_*$, $\psi_*$ and let $F: \uint^2 \to \setL$ denote the
associated family of Lorenz maps
\[
  (u,v) = \lambda \mapsto F_\lambda = (u,v,c_*,\phi_*,\psi_*).
\]
We will assume that $F_\lambda \in \setK \cap \setLS_\Omega$ (see
\defref{Omega-K}) and that $\lbb$ has been fixed (and is large enough).

\begin{definition}
  An \Index{archipelago} $A_\omega \subset \uint^2$ of type $\omega \in \Omega$
  is the set of $\lambda$ such that $F_\lambda$ is $\omega$--renormalizable.
  An \Index{island} of $A_\omega$ is a connected component of the interior
  of~$A_\omega$.
\end{definition}

For the family $\lambda \mapsto F_\lambda$ we have the following very strong
structure theorem for archipelagos \citep[this should be contrasted
with][]{MdM01}.  Note that $c_*$, $\phi_*$ and~$\psi_*$ are arbitrary, so the
results in this section holds for \emph{any} family such that $F_\lambda \in
\setK \cap \setLS_\Omega$.

\begin{theorem} \label{thm:island}
  For every $\omega \in \Omega$ there exists a unique island~$I$ such that the
  archipelago $A_\omega$ equals the closure of~$I$.  Furthermore, $I$ is
  diffeomorphic to a square.
\end{theorem}

\begin{remark}
  This theorem shows that the structure of $A_\omega$ is very rigid.  Note that
  the structure of archipelagos is much more complicated in general.
  There may be multiple islands, islands need not be square, there may be
  isolated points, etc.
\end{remark}

\begin{theorem} \label{thm:cantor-archipelago}
  For every $\rtype \in \Omega^\nats$ there exists a unique $\lambda$ such that
  $F_\lambda$ has combinatorial type $\rtype$.  The set of all such $\lambda$
  is a Cantor set.
\end{theorem}

The family $F_\lambda$ is monotone\index{monotone family}, by which we mean
that $u \mapsto F_{(u,v)}(x)$ is strictly increasing for $x \in (0,c_*)$, and
$v \mapsto F_{(u,v)}(x)$ is strictly decreasing for $x \in (c_*,1)$.  As a
consequence, if we let
\[
  M_{(u,v)}^+ = \{ (x,y) \mid x \geq u, y \leq v\}
  \quad\text{and}\quad
  M_{(u,v)}^- = \{ (x,y) \mid x \leq u, y \geq v\},
\]
then
\[
  \mu \in M_\lambda^+ \implies F_\mu(x) > F_\lambda(x)
  \quad\text{and}\quad
  \mu \in M_\lambda^- \implies F_\mu(x) < F_\lambda(x),
\]
for all $x \in (0,1) \setminus \{c\}$.  In other words, deformations in
$M_\lambda^+$ moves both branches up, deformations in $M_\lambda^-$ moves
both branches down.  This simple observation is key to analyzing the structure
of archipelagos.

\begin{definition} \label{def:R-proj}
  Let $\pi_S:\reals^3\to\reals^2$ be the projection which takes the rectangle
  $[c,1] \times [1-c,1] \times \{c\}$ onto $S = [\nicefrac12,1]^2$
  \[
    \pi_S(x,y,c) = \left( 1 - \frac{1-x}{2(1-c)}, 1 - \frac{1-y}{2c} \right),
  \]
  and let $H$ be the map which takes $(u,v,c,\phi,\psi)$ to the height of
  its branches ($c$ is kept around because $\pi_S$ needs it)
  \[
    H(u,v,c,\phi,\psi) = (\phi(u), 1 - \psi(1-v), c).
  \]
  Now define $R: A_\omega \to S$ by
  \[
    R(\lambda) = \pi_S \circ H \circ \opR(F_\lambda).
  \]
\end{definition}

\begin{figure}
  \begin{center}
    \includegraphics{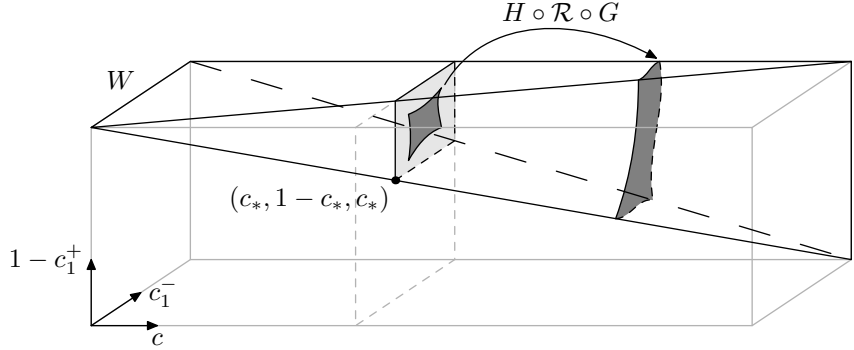}
  \end{center}
  \caption{Illustration of the action of $\opR$ on the family $F_\lambda$.  The
  dark gray island is mapped onto a set which is wrapped around the wedge~$W$.
  That is, the boundary of the island is mapped into the boundary of~$W$ with
  nonzero degree.  Note that in this illustration we project the image
  of~$\opR$ to $\reals^3$ via the map~$H$.  The maps $H$ and~$G$ convert
  between critical values $(\lcv,\rcv)$ and $(u,v)$--parameters.  Explicitly
  $G(\lcv,1-\rcv,c_*) = (\phi_*\inv(\lcv),1 -
  \psi_*\inv(\rcv),c_*,\phi_*,\psi_*)$.}
  \label{fig:R}
\end{figure}

\begin{remark}
  The action of $R$ can be understood by looking at \figref{R}.  The boundary
  of an island $I$ is mapped into the boundary of the wedge~$W$ by the map $H
  \circ \opR$.  The four boundary pieces of the wedge correspond to when the
  renormalization has at least one full or trivial branch.  Note that the image
  of $\bndry I$ in $\bndry W$ will not in general lie in a plane, instead it
  will be bent around somewhat.  For this reason we project down to the
  square~$S$ via the projection~$\pi_S$.  This gives us the final operator~$R:
  A_\omega \to S$.
\end{remark}

\begin{proposition} \label{prop:square}
  Let $I \subset A_\omega$ be an island.  Then $R$ is an orientation-preserving
  diffeomorphism that takes the closure of~$I$ onto~$S$.
\end{proposition}

\begin{remark}
  This already shows that the structure of archipelagos is very rigid.  First
  of all every island is full, but there are also exactly one of each type of
  extremal points, and exactly one of each type of vertex.  In other words,
  there are no degenerate islands of any type!  Extremal points and vertices
  are defined in \citet{MdM01}, see also the caption of \figref{island}.
\end{remark}

\begin{proof}
  By definition $R$ maps $I$ into~$S$ and $\bndry I$ into $\bndry S$.  We claim
  that $\opD R_\lambda$ is orientation-preserving for every $\lambda \in \clos
  I$.\footnote{The notation $\opD R_\lambda$ is used to denote the derivative
  of $R$ at the point~$\lambda$.}  Assume that the claim holds (we will prove
  this soon).
  
  We contend that $R$ maps $\clos I$ onto $S$.  If not, then $R(\bndry I)$ must
  be strictly contained in $\bndry S$, since the boundaries are homeomorphic to
  the circle and $R$ is continuous.  But then $\opD R_\lambda$ must be singular
  for some $\lambda \in \bndry I$ which contradicts the claim.

  Hence $R: \clos I \to S$ maps a simply connected domain onto a simply
  connected domain, and $\opD R$ is a local isomorphism.  Thus $R$ is in fact a
  diffeomorphism.

  We now prove the claim.  A computation gives
  \[
    \opD\pi_S(x,y,c) = \begin{pmatrix}
      (2(1-c))\inv & 0        & \star\\
      0            & (2c)\inv & \star
    \end{pmatrix},
  \]
  and
  \[
    \opD H_{(u,v,c,\phi,\psi)} = \begin{pmatrix}
      \phi'(u) & 0       & \dots \\
      0        & \psi'(1-v) & \dots \\
      \star    & \star   & \dots
    \end{pmatrix}.
  \]
  The top-left $2\times2$ matrix is orientation-preserving in both cases and
  the same is true for $\opD\opR$ by \corref{detM1}.  Thus $\opD R_\lambda$ is
  orientation-preserving.
\end{proof}

\begin{figure}
  \begin{center}
    \includegraphics{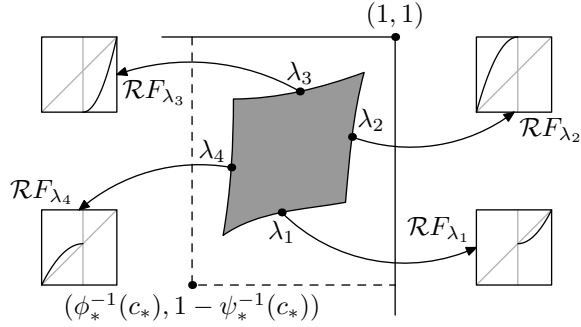}
  \end{center}
  \caption{Illustration of a full island for the family $F_\lambda$.  The
    boundary corresponds to when at least one branch of the renormalization
    $\opR F_\lambda$ is either full or trivial.  The top right and bottom left
    corners are extremal points\index{extremal point}; the top left and bottom
    right corners are vertices\index{vertex}.}
  \label{fig:island}
\end{figure}

\begin{lemma} \label{lem:triv-crossing}
  Assume $f^m(\lcv) = c = f^n(\rcv)$ for some $m,n > 0$.
  Let $(l,c)$ and~$(c,r)$ be branches of $f^m$ and~$f^n$, respectively.
  Then $f^m(l) \leq l$ and $f^n(r) \geq r$.  In particular, $f$ is
  renormalizable to a map with trivial branches.
\end{lemma}

\begin{proof}
  In order to reach a contradiction we assume that $f^m(l) > l$.  Then
  $f^{im}(l)
  \uparrow x$ for some point $x \in (l,c]$ as $i \to \infty$, since $f^m(\lcv)
  = c$.  Since $l$ is the left endpoint of a branch there exists $t$ such that
  $f^t(l) = \rcv$.  Hence $f^{m-t}(\rcv) = l$ so the orbit of $\rcv$ contains
  the orbit of~$l$.  But the orbit of $\rcv$ was periodic by assumption which
  contradicts $f^{im}(l) \uparrow x$.  Hence $f^m(l) \leq l$.

  Now repeat this argument for $r$ to complete the proof.
\end{proof}

\begin{definition}
  Define
  \begin{align*}
    \trivl &= \big\{ \lambda \in \uint^2 \;\big\vert\;
      F_\lambda^{a+1}(c_*^-) = c_*
      \text{ and } F_\lambda^i(c_*^-) > c_*, i=1,\dots,a \big\}, \\
    \trivr &= \big\{ \lambda \in \uint^2 \;\big\vert\;
      F_\lambda^{b+1}(c_*^+) = c_*
      \text{ and } F_\lambda^i(c_*^+) < c_*, i=1,\dots,b \big\}.
  \end{align*}
  (The notation here is $g(c_*^-) = \lim_{x\uparrow c_*} g(x)$ and
  $g(c_*^+) = \lim_{x\downarrow c_*} g(x)$.)
\end{definition}

\begin{lemma} \label{lem:triv-curves}
  The set $\trivl$ is the image of a curve $v \mapsto (g(v),v)$.  The map $g$
  is differentiable and takes $[1-\psi_*\inv(c_*),1]$ into
  $[\phi_*\inv(c_*),1)$.

  Similarly, $\trivr$ is the image of a curve $u \mapsto (u,h(u))$ where $h$ is
  differentiable and takes $[\phi_*\inv(c_*),1]$ into $[1-\psi_*\inv(c_*),1)$.
\end{lemma}

\begin{proof}
  Define
  \[
    g(v) = \phi_*\inv \circ (\psi_* \circ Q_1)^{-a}(c_*)
    \quad\text{and}\quad
    h(u) = 1 - \psi_*\inv \circ (\phi_* \circ Q_0)^{-b}(c_*).
  \]
  Note that $Q_1$ depends on $v$ and $Q_0$ depends on $u$ so $g$ and~$h$ are
  well-defined maps.  It can now be checked that these maps define $\trivl$
  and~$\trivr$.
\end{proof}

\begin{lemma} \label{lem:cross}
  Assume that $\trivl$ crosses $\trivr$ and let $\lambda \in \trivl \cap
  \trivr$.  Then the crossing is transversal and there exists $\rho > 0$ such
  that if $r < \rho$, then the complement of $\trivl \cup \trivr$ inside the
  ball $B_r(\lambda)$ consists of four components and exactly one of these
  components is contained in the archipelago~$A_\omega$.
\end{lemma}

\begin{proof}
  To begin with assume that the crossing is transversal so that the complement
  of $\trivl \cup \trivr$ in $B_r(\lambda)$ automatically consists of four
  components for $r$ small enough.  Note that $\trivl \cup \trivr$ does not
  intersect $M_\lambda^+ \cup M_\lambda^- \setminus \{\lambda\}$.  Hence,
  precisely one component will have a boundary point $\mu \in \trivl$ such that
  $\trivr$ intersects $M_\mu^+$.  Denote this component by~$N$.  Note that if
  we move from $\mu$ inside $N \cap M_\mu^+$ then the left critical value of
  the return map moves above the diagonal.  If we move in $N \cap M_\mu^+$ from
  a point in $\trivr$ then the right critical value of the return map moves
  below the diagonal.

  By \lemref{triv-crossing} $F_\lambda$ is renormalizable and moreover the
  periodic points $p_\lambda$ and~$q_\lambda$ that define the return interval
  of $F_\lambda$ are hyperbolic repelling by the minimum principle.  Hence, if
  we deform $F_\lambda$ into $N$ it will still be renormalizable since $N$
  consists of $\mu$ such that $F^{a+1}_\mu(c^-)$ is above the diagonal and
  $F^{b+1}_\mu(c^+)$ is below the diagonal.  By choosing $r$ small enough all
  of $N$ will be contained in $A_\omega$.
  
  Note that if we deform into any other component (other than~$N$) then at
  least one of the critical values of the return map will be on the wrong side
  of the diagonal and hence the corresponding map is not renormalizable.  Thus
  only the component $N$ intersects $A_\omega$.

  Now assume that the crossing is not transversal.  Then we may pick $\lambda$
  in the intersection $\trivl \cap \trivr$ so that it is on the boundary of an
  island (by the above argument).  But then $\lambda$ must be at a transversal
  intersection since islands are square by \propref{square} and the curves
  $\trivl$ and $\trivr$ are differentiable.  Hence every crossing is
  transversal.
\end{proof}

\begin{figure}
  \begin{center}
    \includegraphics{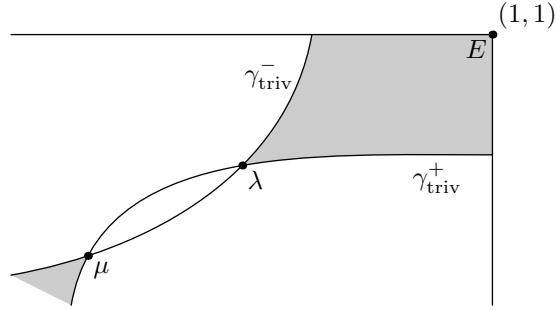}
  \end{center}
  \caption{Illustration of the proof of~\thmref{island}.  Both $\lambda$
  and~$\mu$ must be in the boundary of islands, which lie inside the shaded
  areas.  These two islands have opposite orientation which is impossible.}
  \label{fig:island-proof}
\end{figure}

\begin{proof}[Proof of~\thmref{island}]
  From \propref{square} we know that every island must contain an extremal
  point which renormalizes to a map with only trivial branches, and hence every
  island must be adjacent to a crossing between the curves $\trivl$
  and~$\trivr$.  We claim that there can be only one such crossing and hence
  uniqueness of islands follows.  Note that there is always at least one island
  by \propref{full}.

  By \lemref{triv-curves} $\trivl$ and $\trivr$ terminate in the upper and
  right boundary of $\uint^2$, respectively.  Let $\lambda$ be the crossing
  nearest the points of termination in these boundaries.  Let $E$ be the
  component in the complement of $\trivl \cup \trivr$ in $\uint^2$ that
  contains the point $(1,1)$.  The geometrical configuration of $\trivl$ and
  $\trivr$ is such that $E$ must contain the piece of $A_\omega$ adjacent
  to~$\lambda$ as in \lemref{cross}.  To see this use the fact that
  deformations in the cones $M_\lambda^+$ moves both branches of~$F_\lambda$
  up.

  In order to reach a contradiction assume that there exists another crossing
  $\mu$ between $\trivl$ and~$\trivr$ (see \figref{island-proof}).  By
  \lemref{cross} there is an island attached to this crossing but the
  configuration of $\trivl$ and $\trivr$ at~$\mu$ is such that this island is
  oriented opposite to the island inside~$E$.  But $R$ is
  orientation-preserving so both islands must be oriented the same way and
  hence we reach a
  contradiction.  The conclusion is that there can be no more than one crossing
  between $\trivl$ and $\trivr$ as claimed.

  Finally, the entire archipelago equals the closure of the island since the
  derivative of $R$ is nonsingular at every point in the archipelago.  Hence
  every point in the archipelago must either be contained in an island or on
  the boundary of an island.
\end{proof}

\begin{proof}[Proof of~\thmref{cantor-archipelago}]
  By \thmref{island} there exists a unique sequence of nested
  squares\footnote{By a \emph{square} we mean any set diffeomorphic to the unit
  square.}
  \[
    I_1 \supset I_2 \supset I_3 \supset \dotsm
  \]
  such that $\lambda \in I_k$ implies that $F_\lambda$ is renormalizable of
  type $(\omega_0,\dotsc,\omega_{k-1})$.

  Let $F: \uint^2 \to \setL$ be the map $\lambda \mapsto F_\lambda$, let $p:
  \setL \to \uint^2$ be the projection $(u,v,c,\phi,\psi) \mapsto (u,v)$, and
  let $G_k: I_k \to \uint^2$ be defined by $G_k = p \circ \opR^k \circ F$.  The
  set of tangent vectors $v$ to $F(I_k)$ are all horizontal, so the image of
  $v$ under $\opD\opR$ is in a horizontal cone with a very small angle by
  \propref{cone-field}.  This cone is invariant under $\opD\opR$ by the same
  proposition and furthermore it is strongly expanded by
  \propref{cone-expansion}.\footnote{Note that $\setL$ can be embedded in
  $\lorenzd$ by sending $\phi$ and $\psi$ to singleton decompositions, which is
  how we can apply the propositions from \secref{cone-field} even though they
  are stated for decomposed Lorenz maps.} \propref{K-relcpt} shows that $\setK$
  is relatively compact so there are uniform bounds on $\abs{U}$ and $\abs{V}$
  and hence \propref{cone-expansion} shows that there exists $\mu > 1$ such
  that
  \[
    \norm{\opD G_k} \geq \mu^k.
  \]
  By construction $G_k(I_k) \subset \uint^2$ so the image is bounded, which
  together with the lower bound on $\norm{\opD G_k}$ shows that the diameter of
  $I_k$ shrinks at an exponential rate.  In particular $\bigcap I_k$ is a
  point.  That the union of all such points is a Cantor set is a standard
  argument (using that each type $\rtype \in \Omega^\nats$ has a unique
  associated sequence of squares $\{I_k\}$ and that each such sequence shrinks
  at a uniform exponential rate).
\end{proof}


\section{Unstable manifolds} 
\label{sec:unstable}

The norm used on the tangent space does not give good enough estimates to see a
contracting subspace so we cannot quite prove that the limit set of~$\opR$ is
hyperbolic.  However, these estimates did give an expanding invariant cone
field and in this section we will show how this gives us unstable manifolds at
each point of the limit set.

Instead of trying to appeal to the stable and~unstable manifold theorem for
dominated splittings to get
local unstable manifolds we directly construct global unstable manifolds by
using all the information we have about the renormalization operator and its
derivative.  This is done by defining a graph transform and showing that it
contracts some suitable metric similarly to the Hadamard proof of the stable
and~unstable manifold theorem.  We are only able to show that the resulting
graphs are $\setC^1$ since we do not have hyperbolicity.  Our proof is an
adaptation of the proof of Theorem~6.2.8 in~\citet{KH95}.

\begin{definition}
  Let $\attr_\Omega$ be as in \defref{attr} and define the \Index{limit set of
  renormalization} for types in~$\Omega$ by
  \[
    \limitset_\Omega = \attr_\Omega \cap \lorenzd_{\Omega^\nats}.
  \]
\end{definition}

\begin{remark}
  Here $\lorenzd_{\Omega^\nats}$ denotes the set of infinitely
  renormalizable maps with combinatorial type in~$\Omega^\nats$ and
  $\attr_\Omega$ can intuitively be thought of as the \emph{attractor} for
  $\opR$.  The set $\Omega$ is the same as in \secref{inv-set}, as always.

  Note that by \propref{attr-pure}
  \[
    \limitset_\Omega \subset \uint^2 \times (0,1) \times \pdecomps^2,
  \]
  where $\pdecomps$ denotes the set of pure decompositions, see
  \defref{pdecomps}.
\end{remark}

\begin{theorem} \label{thm:unstable}
  \index{unstable manifold}
  For every $\decomp f = (u,v,c,\decomp\phi,\decomp\psi) \in \limitset_\Omega$
  there exists a unique global unstable manifold $\mfd^u(\decomp f)$.  The
  unstable manifold is a graph
  \[
    \mfd^u(\decomp f) = \big\{ \big(\xi,\sigma(\xi)\big) \mid
    \xi \in I \big\},
  \]
  where $\sigma: I \to (0,1) \times \pdecomps^2$ is
  $\kappa$--Lipschitz for some $\kappa \ll 1$ (not depending on~$\decomp f$).
  The domain $I \subset \reals^2$ is essentially given by
  \[
    \pi\left( \opR(\lorenzd_\omega) \cap
    \big(\uint^2 \times \{c\} \times \{\decomp\phi\} \times \{\decomp\psi\}\big)
    \right),
  \]
  where $\pi$ is the projection onto the $(u,v)$--plane, and $\omega$ is
  defined by $\decomp f$ being in the image $\opR(\lorenzd_\omega)$.
  Additionally, $\mfd^u$ is $\setC^1$.
\end{theorem}

\begin{remark}
  Note that in stark contrast to the situation in the `regular' stable and
  unstable manifold theorem we get \emph{global} unstable manifolds which are
  graphs and that these are almost completely \emph{straight} due to the
  Lipschitz
  constant being very small.  The statement about the domain~$I$ is basically
  that $I$ is ``as large as possible.''  This will be elaborated on in the
  proof.

  Another thing to note is that we cannot say anything about the uniqueness of
  $\decomp f \in \limitset_\Omega$ for a given combinatorics.  That is, given
  \[
    \rtype = (\dotsc,\omega_{-1},\omega_0,\omega_1,\dotsc)
  \]
  we cannot prove that
  there exists a unique $\decomp f \in \limitset_\Omega$ realizing this
  combinatorics.  Instead we see a foliation of the set of maps with
  type~$\rtype$ by unstable manifolds.  If we had a hyperbolic structure on
  $\limitset_\Omega$ this problem would go away.
\end{remark}

\begin{theorem} \label{thm:monotonicity}
  Let $\decomp f \in \limitset_\Omega$ and let $\rtype \in \Omega^\nats$.  Then
  $\mfd^u(\decomp f)$ intersects the set of infinitely renormalizable maps of
  combinatorial type~$\rtype$ in a unique point, and the union of all such
  points over $\rtype \in \Omega^\nats$ is a Cantor set.
\end{theorem}

\begin{proof}
  \thmref{unstable} shows that the unstable manifolds are straight (see the
  above remark) and hence
  \lemref{surface-tube} enables us to apply the same arguments as in
  \thmref{cantor-archipelago}.
\end{proof}

\begin{lemma} \label{lem:surface-tube}
  There exists $\kappa$ close to~$1$ such that if $\gamma: [0,1]^2 \to (0,1)
  \times \pdecomps^2$ is $\kappa$--Lipschitz and $\graph \gamma \subset
  \dinvset$, then $\lorenzd_\omega \cap \graph \gamma$ is diffeomorphic to a
  square, for every $\omega \in \Omega$.
\end{lemma}

\begin{proof}
  By \thmref{island} the set $\lorenzd_\omega \cap \dinvset$ is a tube for
  every $\omega\in\Omega$ (by \emph{tube} we mean that the set is diffeomorphic
  to $\uint^2 \times X$ for some set $X$).  Take a tangent vector at a point in
  $\bndry(\opR\lorenzd_\omega) \cap \dinvset$.  Such a tangent will lie in the
  complement of a cone $\cone_{\kappa} = \{\norm{y} \leq \kappa\norm{x}\}$ for
  $\kappa<1$ close to~$1$, since the projection of the image of a tube to the
  $(u,v,c)$--subspace will look like a slightly deformed cut-off part of the
  wedge in~\figref{R} and the maximum angle of a tangent vector in the boundary
  of the wedge is exactly~$1$.  By \propref{cone-field}, $\opD\opR\inv$ maps
  the complement of $\cone_{\kappa}$ into itself and hence every tube ``lies in
  the complement of $\cone_\kappa$''.  That is, a tangent vector at a point in
  the boundary of a tube lies in the complement of~$\cone_\kappa$, so the tubes
  cut the $(u,v)$--plane at an angle which is smaller than~$1/\kappa$.

  Now if we choose $\kappa$ as above, then the graph of~$\gamma$ will also
  intersect every tube on an angle.  Hence the intersection is diffeomorphic
  to a square.  The main point here is that with $\kappa$ chosen properly,
  $\gamma$ cannot `fold over' a tube and in such a way create an intersection
  which is not simply connected.
\end{proof}

\begin{proof}[Proof of \thmref{unstable}]
  The proof is divided into three steps:
  \begin{inparaenum}[(1)]
    \item definition of the graph transform~$\Gamma$,
    \item showing that $\Gamma$ is a contraction,
    \item proof of $\setC^1$--smoothness of the unstable manifold.
  \end{inparaenum}

  \medskip
  \textit{Step 1.}
  From \lemref{crit-vals} we know that the parameters $u$ and~$v$ for any map
  in $\dinvset$ are uniformly close to~$1$ so there exists $\mu \ll 1$ such
  that if we define the `block'
  \[
    \blok = [1-\mu,1]^2 \times (0,1) \times \pdecomps^2 \cap \dinvset,
  \]
  then $\lorenzd_\Omega \cap \dinvset \subset \blok$, $1-\mu > \phi\inv(c)$ and
  $\mu > \psi\inv(c)$ for all $(u,v,c,\phi,\psi) \in \opO(\blok)$.  In other
  words, the block $\blok$ is defined so that it contains all maps in
  $\dinvset$ which are renormalizable of type in $\Omega$ and the square
  $[1-\mu,1]^2$ is contained in the projection of the image
  $\opR(\lorenzd_\Omega \cap \dinvset)$ onto the $(u,v)$--plane.

  Fix $\decomp f_0 \in \limitset_\Omega$ and $\kappa \in (\lb\kappa,1)$, where
  $\lb\kappa$ is the supremum of $\lb\kappa(\decomp f)$ defined in
  \propref{cone-field} and $\kappa$ is small enough so that
  \lemref{surface-tube} applies.  Associated with $\decomp f_0$ are two
  bi-infinite sequences $\{\omega_i\}_{i\in\ints}$ and $\{\decomp
  f_i\}_{i\in\ints}$ such that $\opR_{\omega_i} \decomp f_i = \decomp f_{i+1}$
  for all $i \in \ints$.  Now define $\graphs_i$, the ``unstable graphs
  centered on~$\decomp f_i$,'' as the set of $\kappa$--Lipschitz maps
  $\gamma_i: [1-\mu,1]^2 \to (0,1) \times \pdecomps^2$ such that $\graph
  \gamma_i \subset \blok$ and $\gamma_i(\xi_i) =
  (c_i,\decomp\phi_i,\decomp\psi_i)$, where $\decomp f_i =
  (\xi_i,c_i,\decomp\phi_i,\decomp\psi_i)$.

  Let $\graphs = \prod_i\graphs_i$.  We will now define a metric on $\graphs$.
  Let
  \[
    \distance_i(\gamma_i,\theta_i) = \sup_{\xi\in[1-\mu,1]^2}
    \frac{\abs{\gamma_i(\xi) - \theta_i(\xi)}}%
    {\abs{\xi-\xi_i}},
    \qquad \gamma_i, \theta_i \in \graphs_i,
  \]
  and define
  \[
    \distance(\gamma,\theta) =
    \sup_{i\in\ints} \distance_i(\gamma_i,\theta_i),
    \qquad \gamma,\theta \in \graphs.
  \]
  This metric turns $(\graphs,\distance)$ into a complete metric space.  Note
  that it is not enough to simply use a $\setC^0$--metric since we do not have
  a contracting subspace of~$\opD\opR$.  The denominator in the definition of
  $\distance_i$ is thus necessary to turn the graph transform into a
  contraction.

  We can now define the graph transform $\Gamma: \graphs \to \graphs$ for
  $\decomp f_0$.  Let $\gamma_i \in \graphs_i$ and define $\Gamma_i(\gamma_i)$
  to be the $\gamma'_{i+1} \in \graphs_{i+1}$ such that
  \[
    \graph \gamma'_{i+1} =
    \opR_{\omega_i}(\graph \gamma_i \cap \lorenzd_{\omega_i}) \cap \blok.
  \]
  Let us discuss why this is a well-defined map $\Gamma_i: \graphs_i \to
  \graphs_{i+1}$.
  \lemref{surface-tube} shows that
  $\opR_{\omega_i}(\graph \gamma_i \cap \lorenzd_{\omega_i})$ is the graph of
  some map $I \subset \reals^2 \to (0,1) \times \pdecomps^2$, where $I$ is
  simply connected.  That $I \supset [1-\mu,1]^2$ is a consequence of how
  $\blok$ was chosen.  Finally, this map is $\kappa$--Lipschitz by
  \propref{cone-field}.

  Actually, we have cheated a little bit here since
  \propref{cone-field} is stated for maps satisfying the extra condition
  \[
    1 - \rcv(\opR f) \geq \lambda,
  \]
  for some $\lambda \in (0,1)$ not depending on $\decomp f \in
  \limitset_\Omega$.  In defining the graph transform we should intersect
  $\lorenzd_{\omega_i}$ with the set defined by this condition before mapping
  it forward by $\opR_{\omega_i}$.  Otherwise we do not have enough information
  to deduce that the entire image is $\kappa$--Lipschitz as well.  However,
  this problem is artificial.  We are free to choose the constant $\lambda$ as
  close to~$0$ as we like and we would still get the invariant cone field
  (although $\lbb$ may need to be increased).  All this means is that
  domain~$I$ of the theorem is slightly smaller than it should be (we have to
  cut out a small part of the graph where $v$ is very close to~$0$ but $v$ is
  still allowed to range all the way up to~$1$ so this amounts to a very small
  part of the domain).  This is one reason why we say that ``$I$ is essentially
  given by \ldots'' in the statement of the theorem.  The other reason is that
  the intersection with $\opR(\lorenzd_\omega)$ should be taken with a surface
  with a small angle and not a surface which is parallel to the $(u,v)$--plane.

  The graph transform is now defined by
  \[
    \Gamma(\gamma) = \big\{\Gamma_i(\gamma_i)\big\}_{i\in\ints},
    \qquad \gamma = \{\gamma_i\}_{i\in\ints} \in \graphs.
  \]

  We claim that $\Gamma$ is a contraction on $(\graphs,\distance)$ and hence
  the contraction mapping theorem implies that~$\Gamma$ has a unique fixed
  point $\gamma^* \in \graphs$.  The global unstable manifolds along $\{\decomp
  f_i\}$ are then given by
  \[
    \mfd^u(\decomp f_{i+1}) =
    \graph \Gamma_i(\gamma^*_i), \quad \forall i \in \ints.
  \]
  In particular, this proves existence and uniqueness of the global unstable
  manifold at $\decomp f_0$.  That these are the \emph{global} unstable
  manifolds is a consequence of $\lorenzd_\Omega \cap \dinvset \subset \blok$.
  Furthermore, the Lipschitz constant for these graphs is much smaller than~$1$
  since we can pick $\kappa$ close to~$\lb\kappa$.  Again, we are cheating a
  little bit here since we have to cut out a small part of the domain of the
  graph as discussed above.

  \medskip
  \textit{Step 2.}
  We now prove that $\Gamma$ is a contraction.  The focus will be on $\Gamma_i$
  for now and to avoid clutter we will drop subscripts on elements of
  $\graphs_i$ and $\graphs_{i+1}$.  Pick $\gamma,\theta \in \graphs_i$ and let
  $\gamma' = \Gamma_i(\gamma)$ and $\theta' = \Gamma_i(\theta)$.  Note that
  $\gamma',\theta' \in \graphs_{i+1}$.

  We write
  \[
    \opR \decomp f = (A(\xi,\eta),B(\xi,\eta)),
  \]
  where $\decomp
  f = (\xi,\eta)$, $\xi \in \reals^2$ and $A(\xi,\eta) \in
  \reals^2$.  Let $A_\gamma(\xi) = A(\xi,\gamma(\xi))$ and
  similarly $B_\gamma(\xi) = B(\xi,\gamma(\xi))$.  With this
  notation the action of $\Gamma_i$ is given by
  \[
    \big(\xi,\gamma(\xi)\big) \mapsto
    \big(A_\gamma(\xi), B_\gamma(\xi)\big) =
    \big(\xi',\gamma'(\xi')\big).
  \]
  Hence
  \[
    \distance_{i+1}(\gamma',\theta') =
      \sup_{\xi'} 
      \frac{\norm{\gamma'(\xi') - \theta'(\xi')}}%
      {\norm{\xi' - \xi_{i+1}}}
    = \sup_{A_\gamma(\xi)} 
      \frac{\norm{\gamma' \circ A_\gamma(\xi)
      - \theta' \circ A_\gamma(\xi)}}%
    {\norm{A_\gamma(\xi) - A_\gamma(\xi_i)}}.
  \]
  Recall that the notation here is $(\xi_i,\gamma(\xi_i)) = \decomp
  f_i$ and $(\xi_{i+1},\gamma'(\xi_{i+1})) = \decomp f_{i+1}$.

  The last numerator can be estimated by
  \begin{align*}
    &\norm{\gamma' \circ A_\gamma(\xi) -
      \theta' \circ A_\gamma(\xi)} \\
    &\qquad\leq \norm{\gamma' \circ A_\gamma(\xi) -
      \theta' \circ A_\theta(\xi)} +
      \norm{\theta' \circ A_\gamma(\xi) -
      \theta' \circ A_\theta(\xi)} \\
    &\qquad\leq \norm{B_\gamma(\xi) - B_\theta(\xi)} +
      \kappa\norm{A_\gamma(\xi) - A_\theta(\xi)} \\
    &\qquad\leq \left(\norm{M_4} + \kappa\norm{M_2}\right)
      \norm{\gamma(\xi) - \theta(\xi)}.
  \end{align*}
  The denominator can bounded by \propref{cone-expansion}
  \[
    \norm{A_\gamma(\xi) - A_\gamma(\xi_i)}
    \geq
    k \cdot \min\{\abs{U}\inv,\abs{V}\inv\} \cdot
    \norm{\xi - \xi_i}.
  \]
  Thus
  \[
    \distance_{i+1}(\gamma',\theta') \leq
    \frac{
    (\norm{M_4} + \kappa\norm{M_2})
    }{k \cdot \min\{\abs{U}\inv,\abs{V}\inv\}}
    \distance_i(\gamma,\theta)
    = \nu \distance_i(\gamma,\theta).
  \]
  \thmref{opnorms} shows that $\nu \ll 1$ uniformly in the index~$i$.  Hence
  $\Gamma$ is a (very strong) contraction.

  \medskip
  \textit{Step 3.}
  \TODO{prove this}
  Going from Lipschitz to $\setC^1$ smoothness of the unstable manifold is a
  standard argument.  See for example \citet[Chapter~6.2]{KH95}.
\end{proof}


\appendix

\section{A fixed point theorem} 
\label{sec:fixpt-thm}

The following theorem is an adaptation of \citet[Theorem~4.7]{GD03}.

\begin{theorem} \label{thm:top-fp}
  Let $X \subset Y$ where $X$ is closed and $Y$ is a normal topological space.
  If $f:X \to Y$ is homotopic to a map $g: X \to Y$ with the property that
  every extension of $g|_{\bndry X}$ to~$X$ has a fixed point in~$X$, and if
  the homotopy $h_t$ has no fixed point on~$\bndry X$ for every $t \in \uint$,
  then $f$ has a fixed point in~$X$.
\end{theorem}

\begin{remark}
  Note that the statement is such that $X$ must have nonempty interior.  This
  follows from the assumption that $g$ has a fixed point (since it is an
  extension of $g|_{\bndry X}$) but the requirement on the homotopy implies
  that $g$ has no fixed point on~$\bndry X$.
\end{remark}

\begin{proof}
  Let $F_t$ be the set of fixed points of $h_t$ and let $F = \bigcup F_t$.
  Since $g$ must have a fixed point $F$ is nonempty.  Since $h_t$ has no fixed
  points on~$\bndry X$ for every $t$, $F$ and $\bndry X$ are disjoint.

  We claim that $F$ is closed.  To see this, let $\{x_n \in F\}$ be a
  convergent sequence, let $x = \lim x_n$.  Note that $x \in X$ since $F
  \subset X$ and $X$ is closed.  By definition there exists $t_n \in \uint$
  such that $x_n = h(x_n,t_n)$.  Pick a convergent subsequence $t_{n_k} \to t$.
  Since $x_n$ is convergent $h(x_{n_k},t_{n_k}) = x_{n_k} \to x$, but at the
  same time $h(x_{n_k},t_{n_k}) \to h(x,t)$ since $h$ is continuous.  Hence
  $h(x,t) = x$, that is $x \in F$ which proves the claim.

  Since $Y$ is normal and $\bndry X$ and $F$ are disjoint closed sets there
  exists a map $\lambda: X \to \uint$ such that $\lambda|_F = 0$ and
  $\lambda|_{\bndry X} = 1$.  Define $\bar{g}(x) = h(x; \lambda(x))$.  Then
  $\bar{g}$ is an extension of $g|_{\bndry X}$ since if $x \in \bndry X$, then
  $\bar{g}(x) = h(x,1) = g(x)$.  Hence $\bar{g}$ has a fixed point $p \in X$.
  However, $p$ must also be a fixed point of $f$ since $p = \bar{g}(p) =
  h(p,\lambda(p))$ so that $p \in F$ and consequently $p = \bar{g}(p) = h(p,0)
  = f(p)$.
\end{proof}


\section{The nonlinearity operator} 
\label{sec:nonlin}

In this appendix we collect some results on the nonlinearity operator.  The
proofs are either simple calculations or can be found for example in the
appendix of \citep{Mar98}.

\begin{definition}
  Let $\setC^k(A;B)$ denote the set of $k$ times continuously differentiable
  maps $f:A\to B$ and let $\setD^k(A;B)\subset\setC^k(A;B)$ denote the subset
  of orientation-preserving homeomorphisms whose inverse lie in $\setC^k(B;A)$.
  
  As a notational convenience we write $\setC^k(A)$ instead of $\setC^k(A;A)$,
  and $\setC^k$ instead of $\setC^k(A;B)$ if there is no need to specify $A$
  and~$B$ (and similarly for $\setD^k$).
\end{definition}

\begin{definition}
  The nonlinearity operator\index{nonlinearity operator}
  $\opN:\setD^2(A;B)\to\setC^0(A;\reals)$ is defined by
  \begin{equation} \label{eq:nonlin}
    \opN\phi = \opD\log\opD\phi.
  \end{equation}
  We say that $\opN\phi$ is the \Index{nonlinearity} of $\phi$.
\end{definition}

\begin{remark}
  Note that
  \[
    \opN\phi = \frac{\opD^2\phi}{\opD\phi}.
  \]
\end{remark}

\begin{definition}
  The \Index{distortion} of $\phi \in \setD^1(A;B)$ is defined by
  \[
    \distortion\phi = \sup_{x,y\in A}\log
    \frac{\opD\phi(x)}{\opD\phi(y)}.
  \]
\end{definition}

\begin{remark}
  We think of the nonlinearity of $\phi \in \setD^2(A;B)$ as the density for
  the distortion of~$\phi$. \index{nonlinearity!as density for distortion}  To
  understand this remark, let $d\mu = \opN\phi(t) dt$.  Assuming $\opN\phi$ is
  a positive function, then $\mu$ is a measure and
  \[
    \distortion\phi = \int_A d\mu,
  \]
  since by \eqnref{nonlin}
  \[
    \int_x^y \opN\phi(t) dt = \log \frac{\opD\phi(y)}{\opD\phi(x)}.
  \]
  If $\opN\phi$ is negative, then $-\opN\phi(t)$ is a density.  The only
  problem with the interpretation of $\opN\phi$ as a density occurs when it
  changes sign.  Intuitively speaking, we can still think of the nonlinearity
  as a \emph{local} density of the distortion (away from the zeros
  of~$\opN\phi$).

  Note that $\opN\phi$ does not change sign in the important special case of
  $\phi$ being a pure map (i.e.\ a restriction of $x^\alpha$).  So the
  (absolute value of the) nonlinearity \emph{is} the density for the distortion
  of pure maps.
\end{remark}

\begin{lemma}
  The kernel of $\opN:\setD^2(A;B)\to\setC^0(A;\reals)$ equals the
  orientation-preserving affine map that takes $A$ onto~$B$.
\end{lemma}

\begin{lemma}
  The nonlinearity operator $\opN:\setD^2(A;B)\to\setC^0(A;\reals)$ is a
  bijection. In the specific case of $A=B=\uint$ the inverse is given by
  \begin{equation}
    \opN\inv f(x) = \frac{\int_0^x\exp\{\int_0^s f(t)\intd t\}\intd s}%
      {\int_0^1\exp\{\int_0^s f(t)\intd t\}\intd s}.
  \end{equation}
\end{lemma}

\begin{lemma}[The chain rule for the nonlinearity operator]
  \label{lem:N-chain-rule}
  \index{chain rule!for nonlinearities}
  If $\phi,\psi\in\setD^2$ then
  \begin{equation}
    \opN(\psi\circ\phi) = \opN\psi\circ\phi\cdot\opD\phi + \opN\phi.
  \end{equation}
\end{lemma}

\begin{definition}
  We turn $\setD^2(A;B)$ into a Banach space by inducing the usual linear
  structure and uniform norm of $\setC^0(A;\reals)$ via the nonlinearity
  operator.  That is, we define
  \begin{align}
    \alpha\phi+\beta\psi &= \opN\inv\left(\alpha\opN\phi
      + \beta\opN\psi\right), \\
    \norm{\phi} &= \sup_{t\in A}~\abs{\opN\phi(t)},
  \end{align}
  for $\phi,\psi\in\setD^2(A;B)$ and $\alpha,\beta\in\reals$.
\end{definition}

\begin{lemma} \label{lem:Nprop}
  If $\phi\in\setD^2(A;B)$ then
  \begin{gather}
    \e^{-\abs{y-x}\cdot\norm{\phi}} \leq \frac{\opD\phi(y)}{\opD\phi(x)}
      \leq \e^{\abs{y-x}\cdot\norm{\phi}}, \label{eq:dist} \\
    \frac{\abs{B}}{\abs{A}}\cdot \e^{-\norm{\phi}} \leq \opD\phi(x) 
      \leq \frac{\abs{B}}{\abs{A}}\cdot \e^{\norm{\phi}}, \label{eq:deriv} \\
    \abs{\opD^2\phi(x)} \leq \frac{\abs{B}}{\abs{A}}\cdot
      \norm{\phi}\cdot\e^{\norm{\phi}}, \label{eq:sndderiv}
  \end{gather}
  for all $x,y\in A$.
\end{lemma}

\begin{lemma} \label{lem:Nprop2}
  If $\phi,\psi\in\setD^2(A;B)$ then
  \begin{gather}
    \abs{\phi(x)-\psi(x)} \leq \big(\e^{2\norm{\phi-\psi}} - 1\big)
      \cdot \min\{\phi(x),1-\phi(x)\}, \\
    \e^{-\norm{\phi-\psi}} \leq \frac{\opD\phi(x)}{\opD\psi(x)}
      \leq \e^{\norm{\phi-\psi}},
  \end{gather}
  for all $x\in A$.
\end{lemma}

\begin{definition}
  Let $\zeta_J:\uint\to J$ be the affine orientation-preserving map
  taking $\uint$ onto an interval~$J$.

  Define the \Index{zoom operator} $\opZ:\setD^2(A;B)\to\setD^2(\uint)$ by
  \begin{equation}
    \opZ\phi = \zeta\inv_B \circ \phi \circ \zeta_A.
  \end{equation}
\end{definition}

\begin{remark}
  Note that if $\phi \in \setD(A;B)$, then $B = \phi(A)$ so $\opZ\phi$ only
  depends on $\phi$ and $A$ (not on~$B$).  We will often write $\opZ(\phi;A)$
  instead of $\opZ\phi$ in order to emphasize the dependence on~$A$.
\end{remark}

\begin{lemma} \label{lem:zoom}
  If $\phi\in\setD^2(A;B)$ then
  \begin{align}
    \opZ(\phi\inv) &= (\opZ\phi)\inv, \label{eq:zoominv} \\
    \opN(\opZ\phi) &= \abs{A} \cdot \opN\phi \circ \zeta_A, \\
    \norm{\opZ\phi} &= \abs{A}\cdot\norm{\phi}. \label{eq:zoomcontract}
  \end{align}
\end{lemma}

\section{The Schwarzian derivative} 
\label{sec:schwarz}

In this appendix we collect some results on the Schwarzian derivative.  Proofs
can be found in \citet[Chapter~IV]{dMvS93}.

\begin{definition}
  The \Index{Schwarzian derivative} $\opS:\setD^3(A;B)\to\setC^0(A;\reals)$ is
  defined by
  \begin{equation}
    \opS f = \opD(\opN f) - \frac{1}{2}(\opN f)^2.
  \end{equation}
\end{definition}

\begin{remark}
  Note that
  \[
    \opS f = \frac{\opD^3 f}{\opD f}
      - \frac{3}{2}\left[\frac{\opD^2 f}{\opD f}\right]^2\!\!\!.
  \]
\end{remark}

\begin{lemma}[The chain rule for the Schwarzian derivative]
  \label{lem:S-chain-rule}
  \index{chain rule!for Schwarzian derivative}
  If $f,g\in\setD^3$, then
  \begin{equation}
    \opS(f\circ g) = \opS f\circ g\cdot (\opD g)^2 + \opS g.
  \end{equation}
\end{lemma}

\begin{lemma}[Koebe Lemma] \label{lem:S-koebe}
  \index{Koebe Lemma}
  If $f\in\setD^3((a,b);\reals)$ and $\opS f\geq0$, then
  \begin{equation}
    \abs{\opN f(x)} \leq 2\cdot\big[\min\{\abs{x-a},\abs{x-b}\}\big]\inv.
  \end{equation}
\end{lemma}

\begin{corollary} \label{cor:koebe}
  Let $\tau>0$ and let $f\in\setD^3(A;B)$.  If $f$ extends to a map
  $F\in\setD^3(I;J)$ with $\opS F<0$ and if $J\setminus B$ has two components,
  each having length at least $\tau\abs{B}$, then
  \[
    \norm{\opZ f} \leq \e^{2/\tau}\cdot 2/\tau.
  \]
\end{corollary}

\bibliographystyle{plainnat}
\bibliography{lrhs}


\end{document}